\documentclass[english,a4paper]{amsart}

\usepackage[utf8]{inputenc}
\usepackage[T1]{fontenc}
\usepackage{amsmath,amssymb,amsthm,mathabx,esint,bbm,dsfont,tikz-cd,mathtools}

\usepackage[english]{babel}
\usepackage[all]{xy}
\usepackage{xcolor}
\usepackage{pict2e}
\usepackage{hyperref}
\usepackage{mathrsfs}
\usepackage[margin = 1 in]{geometry}
\theoremstyle{plain}
\newtheorem{Th}{Theorem}[section]
\newtheorem{Lem}[Th]{Lemma}
\newtheorem{Cor}[Th]{Corollary}
\newtheorem{Prop}[Th]{Proposition}
\newtheorem{Def}[Th]{Definition}
\newtheorem{Conj}[Th]{Conjecture}
\newtheorem{Rem}[Th]{Remark}

\newcommand{\NN}{\mathbb{N}}

\newcommand{\KK}{\mathbb{K}}
\renewcommand{\AA}{\mathbb{A}}
\newcommand{\prox}{\mathrm{prox}}

\newcommand{\ie}{\textit{i.e.},\;}
\newcommand{\dist}{\mathrm{d}}

\newcommand{\ZZ}{\mathbb{Z}}

\newcommand{\PP}{\mathbb{P}}
\newcommand{\EE}{\mathbb{E}}
\newcommand{\CC}{\mathbb{C}}
\newcommand{\RR}{\mathbb{R}}

\newcommand{\rank}{\mathbf{rk}}
\newcommand{\sqz}{\mathrm{sqz}}
\newcommand{\Wedge}{\textstyle\bigwedge}

\newcommand{\eps}{\varepsilon}

\subjclass[2020]{Primary 60B20, 60F10, 37H15, 60B15, 60F15, 60F25}

\title[Pivoting technique for products of random matrices]{Limit theorems for a strongly irreducible product of independent random matrices under optimal moment assumptions}

\author{Axel P\'eneau}
\date{\today}

\begin{document}

	\begin{abstract}
		Let $\nu$ be a probability distribution over the semi-group of square matrices of size $d \ge 2$. We assume that $\nu$ is proximal, strongly irreducible and that $\nu^{*n}\{0\}=0$ for all integers $n\in\NN$. We consider the sequence $\overline\gamma_n:=\gamma_0\cdots\gamma_{n-1}$ for $(\gamma_k)_{k \in\NN}$ independent of distribution law $\nu$. 
		We denote by $\sqz(\overline{\gamma}_n)$ the logarithm of the ratio of the two top singular values of $\overline{\gamma}_n$.
		We show that $(\sqz(\overline{\gamma}_n))_{n\in\NN}$ escapes to infinity linearly and satisfies exponential large deviations inequalities below its escape rate. 
		We also show that the image of a generic line by $\overline{\gamma}_n$ as well as its eigenspace of its maximal eigenvalue both converge to the same random line $l^\infty$ at an exponential speed. 
		This is an extension of results by Guivarc'h and Raugi.
		
		If we moreover assume that the push-forward distribution $N_*\nu$ is $\mathrm{L}^p$ for $N: g\mapsto\log\left(\|g\|\|g^{-1}\|\right)$ and for some $p\ge 1$, then we show that $-\log\dist(l^\infty, H)$ is uniformly $\mathrm{L}^p$ for all proper subspace $H \subset \RR^d$. Moreover the logarithm of each coefficient of $\overline{\gamma}_n$ is almost surely equivalent to the logarithm of the norm. 
		This is an extension of results by Benoist and Quint which were themselves quantitative versions of results by Furstenberg and Kesten.

		To prove these results, we do not rely on the existence of the stationary measure nor on the existence of the Lyapunov exponent. Instead we describe an effective way to group the i.i.d. factors into i.i.d. random words that are aligned in the Cartan decomposition. We moreover have an explicit control over the moments. 
	\end{abstract}
	
	\maketitle
	
	%\tableofcontents

	\section{Introduction}
	
	\subsection{Preliminaries}
	
	Let $\KK = \RR$ be the field of real number endowed with the usual absolute value $|\cdot|$.
	Let $d \ge 2$ be an integer. Let $\mathrm{End}(\KK^d) \simeq \mathrm{Mat}_{d \times d}(\KK)$ be the set of square matrices which we identify with the semi-group of linear maps from $\KK^d$ to itself.
	Let $\nu$ be a probability distribution on $\mathrm{End}(\KK^d)$ and let $(\gamma_n)_{n\in\NN}\sim\nu^{\otimes\NN}$ be a random i.i.d\footnote{The abbreviation "i.i.d" is short for "independent and identically distributed" but we will prefer to write $\sim\nu^{\otimes\NN}$ as it allows us to specify the distribution in question as well as the set of indices of the sequence.} sequence of matrices. 
	We will denote by $(\overline\gamma_n)_{n\in\NN} := (\gamma_0 \cdots \gamma_{n-1})_{n\in\NN}$ the random walk of step $\nu$. 
	
	Given $g$ a square $d \times d$ matrix, let $\rho_1(g) \ge \rho_2(g) \ge \dots \ge \rho_d(g)$ be the moduli of its eigenvalues. 
	If $g$ is not nilpotent, we define the spectral gap or quantitative proximality of $g$ as:
	\begin{equation}\label{eqn:defprox}
		\prox(g) := \log\left(\frac{\rho_1(g)}{\rho_2(g)}\right).
	\end{equation}
	If $g$ is nilpotent, we define $\sqz(g)  = 0$.
	We will always consider measures that are strongly irreducible and proximal in the following sense.

	\begin{Def}[Proximality]\label{def:prox}
		Let $E$ be a vector space and let $\nu$ be a probability distribution on $\mathrm{End}(E)$. We say that $\nu$ is proximal if both of the following conditions are satisfied:
		\begin{gather}
			\exists n\in\NN,\; \nu^{*n}\{\gamma \in\mathrm{End}(E)\,|\,\prox(\gamma) > 0\} > 0. \\
			\forall n\in\NN,\; \nu^{*n}\{0\} = 0.
		\end{gather}
	\end{Def}
	
	\begin{Def}[Strong irreducibility]
		Let $E$ be a vector space and let $\nu$ be a probability distribution on $\mathrm{End}(E)$. 
		We say that $\nu$ is irreducible if:
		\begin{equation}
			\forall f \in E^*\setminus\{0\},\, \forall v \in E\setminus\{0\},\, \exists n \in\NN,\; \nu^{*n}\left\{\gamma\in\mathrm{End}(E)\,\middle|\,f \gamma v \neq 0\right\} > 0.
		\end{equation}
		We say that $\nu$ is strongly irreducible if:
		\begin{gather*}
			\forall N \in \NN,\, \forall (f_1, \dots, f_N) \in (E^*\setminus\{0\})^N,\, \forall v \in E\setminus\{0\},\, \exists n \in\NN,\; \nu^{*n}\left\{\gamma\in\mathrm{End}(E)\,\middle|\,\prod_{i = 1}^N f_i \gamma v \neq 0 \right\} > 0,\\
			\text{and} \quad
			\forall N \in \NN,\, \forall f \in E^*\setminus\{0\},\, \forall (v_1, \dots, v_N) \in (E\setminus\{0\})^N,\, \exists n \in\NN,\; \nu^{*n}\left\{\gamma\in\mathrm{End}(E)\,\middle|\,\prod_{j = 1}^N f \gamma v_j \neq 0 \right\} > 0.
		\end{gather*}
	\end{Def}
	
	We will work with real valued matrices but all the results still hold for complex valued matrices or for matrices with coefficients in a ultra-metric locally compact field with the same proofs. We simply need to replace the Euclidean norm with a Hermitian norm or a ultra-metric norm.
	
	Without making any moments assumptions, we will study the behaviour of the projective class $[\overline{\gamma}_n]$ for all $n \in \NN$ and not only asymptotically.
	
	All the following results are corollaries of Theorem \ref{th:pivot}, which is the main theorem of this article. 
	In fact Theorem \ref{th:pivot} follows from Lemma \ref{lem:schottky} and Theorem \ref{th:pivot-extract}.

	\subsection{Regularity results with optimal moment assumptions}
	
	Let $\KK = \RR$.
	Given two Euclidean spaces $(E, \|\cdot\|)$ and $(F,\|\cdot\|)$, we write $\mathrm{Hom}(E,F)$ for the vector space of linear maps from $E$ to $F$, we endow it with the operator norm $h \mapsto \|h\| := \max_{x\in E\setminus\{0\}}\frac{\|hx\|}{\|x\|}$. Given $E$ a Euclidean space, we define $E^* := \mathrm{Hom}(E, \KK)$ to be the dual space of $E$. 
	Given a matrix $h \in \mathrm{Hom}(E,F)$, we write $h^* \in \mathrm{Hom}(F^*, E^*)$ for the composition by $h$ on the right.
	Let $h \in\mathrm{End}(E) := \mathrm{Hom}(E, E)$.
	Denote by $\rho_1(h)$ the limit of  $\|h^n\|^{\frac{1}{n}}$, we call $\rho_1(h)$ the spectral radius of $h$.
	Note that it is the modulus of the maximal eigenvalue of $h$.
	We denote by $\mathrm{GL}(E)$ the group of isomorphisms of $E$. 
	Given $g \in \mathrm{GL}(E)$, we write $N(g) := \log\|g\| + \log\|g^{-1}\|$.
	
	Let $\nu$ be a probability measure on $\mathrm{GL}(E)$ and let $(\gamma_n)\sim \nu^{\otimes\NN}$.
	
	A long standing question is whether the sequence of rescaled entries $\left(\left|(\overline{\gamma}_n)_{i,j}\right|^{\frac{1}{n}}\right)_{n\in\NN}$ converges almost surely for all $1 \le i, j \le d$. 
	We know from~\cite{porm} that if $\EE(\log\|\gamma_0\|) < + \infty$, and without any other assumptions, then the sequence of norms $\left(\left\|\overline{\gamma}_n\right\|^{\frac{1}{n}}\right)_{n\in\NN}$ converges almost surely to a finite non-random limit that we denote by $\rho_1(\nu)$. 
	Furstenberg and Kesten also show that $\rho_1(\nu) > 1$ when $\nu$ is strongly irreducible and supported on the group $\mathrm{SL}(E)$.
	With the above assumptions and when moreover $\EE(N(\gamma_0)^2) < +\infty$ Xiao, Grama and Liu prove in~\cite{xiao2021limit} that the random sequence of rescaled entries $\left(\left|(\overline{\gamma}_n)_{i,j}\right|^{\frac{1}{n}}\right)_{n\in\NN}$ converges almost surely to $\rho_1(\nu)$ for all $i,j$.
	The following theorem allows ut to get rid of the assumption $\EE(N(\gamma_0)^2) < +\infty$.

	\begin{Th}[Strong law of large numbers for the coefficients and for the spectral radius]\label{th:slln}
		Let $E$ be a Euclidean vector space.
		Let $\nu$ be a strongly irreducible and proximal probability measure on $\mathrm{GL}(E)$. There exist constants $C, \beta > 0$ such that for all $f \in E^* \setminus \{0\}$, all $v \in E \setminus \{0\}$, for all $n\in\NN$ and for $\overline\gamma_n \sim \nu^{*n}$, we have for all $t \in \RR_{\ge 0}$:
		\begin{equation}\label{eq:dom-coef}
			\PP\left(\log\frac{\|f\|\|\overline{\gamma}_n\|\|v\|}{|f \overline{\gamma}_n v|} \ge t\right)\le C\exp(-\beta n) + \sum_{k=1}^{+\infty} C \exp(-\beta k) N_*\nu\left(\frac{t}{k}, +\infty\right).
		\end{equation}
		Moreover:
		\begin{equation}\label{eq:dom-radius}
			\forall t \ge 0, \; \PP\left(\log\frac{\|\overline{\gamma}_n\|}{\rho_1(\overline{\gamma}_n)}\ge t\right)\le \sum_{k=1}^{+\infty} C \exp(-\beta k) N_*\nu\left(\frac{t}{k}, +\infty\right).
		\end{equation}
	\end{Th}
	
	We prove this result in Section \ref{sec:lln}
	Note that \eqref{eq:dom-coef} implies that for all non-random sequence $\alpha_n \to 0$ and for all $1 \le i,j \le \dim(E)$, the sequence $\left(\frac{\left|(\overline{\gamma}_n)_{i,j}\right|^{\alpha_n}}{\left\|\overline{\gamma}_n\right\|^{\alpha_n}}\right)_{n\in\NN}$ converges weakly in distribution to the Dirac measure at $1$, without any moment assumption.
	
	Given $C, \beta > 0$ and $\nu$ a probability measure on $\mathrm{GL}(E)$, we denote by $\zeta^{C,\beta}_\nu$ the probability distribution on $\RR_{\ge 0}$ characterized by:
	\begin{equation}\label{eq:def-zetanu}
		\forall t \ge 0,\, \zeta^{C,\beta}_\nu(t,+\infty) := \min\left\{1,\sum_{k=1}^{+\infty} C \exp(-\beta k)N_*{\nu}\left(\frac{t}{k}, +\infty\right)\right\}.
	\end{equation}

	\begin{Cor}[Almost sure convergence of the coefficients]\label{cor:lim-coef}
		Let $E$ be a Euclidean space and let $ \nu $ be a strongly irreducible and proximal probability measure on $\mathrm{GL}(E)$.
		Let $(\gamma_n)\sim \nu^{\otimes\NN}$.
		Assume that $\EE(\log\|\gamma_0\|)$ and $ \EE(\log\|\gamma_0^{-1}\|)$ are both finite. Then for all $f \in E^*\setminus\{0\}$ and all $v \in E \setminus\{0\}$, we have almost surely:
		\begin{equation*}
			\lim_{n\to\infty}\frac{\log|f \overline\gamma_n v|}{n} = \lim_{n\to\infty}\frac{\log\|\overline\gamma_n\|}{n} = \log(\rho_1(\nu)).
		\end{equation*}
	\end{Cor}
	
	\begin{proof}
		By Lemma \ref{lem:sumlp}, for all $p \in (0,+\infty)$, there exist a constant $D_p$ such that $M_p(\zeta^{C,\beta}_\nu) \le D_p M_p(N_*\nu)$, where $M_p$ is the $p$-th moment of a measure.
		Therefore, if we assume that $M_1(N_*\nu)= < +\infty$, then $M_1(\zeta_\nu^{C,\beta}) < +\infty$. 
		Then by Theorem \ref{th:slln}, for all $\eps > 0$, we have:
		\begin{align*}
			\sum_{n\in\NN} \PP\left(\log(\|f\|\|\overline{\gamma}_n\|\|v\|) - \log(|f \overline{\gamma}_n v|) \ge n \eps \right) & \le \sum_{n\in\NN} C \exp(-\beta  n) + \sum_{n\in\NN} \zeta_\nu^{(C, \beta)}(n \eps, + \infty) \\
			& \le \frac{C}{\beta} + \eps^{-1}M_1(\zeta_\nu^{C,\beta}) < +\infty.
		\end{align*}
		Then by Borel-Cantelli's Lemma, we have $n^{-1}\log\frac{\|f\|\|\overline{\gamma}_n\|\|v\|}{|f \overline{\gamma}_n v|} \to 0$ almost surely. Then we can apply~\cite[Theorem~1]{porm} which tells us that $n^{-1}\log{\|\overline{\gamma}_n\|} \to \log(\rho_1(\nu))$.
	\end{proof}
	
	The following Corollary is about the regularity of the stationary measure.
	The formulation \eqref{eq:regxislp} is analogous to the regularity result for the stationary measure on hyperbolic groups \cite[Proposition~5.1]{Benoist2016CentralLT}. 
	This is also an improvement of \cite[Proposition~4.5]{CLT16}.

	Let $E$ be a Euclidean space and $V \subset E$ be a proper subspace and let $0 < r \le 1$.
	We define $\mathcal{N}_r(V) := \{l \in\mathbf{P}(E)\,|\,\exists v\in V\setminus\{0\},\;\dist([v],l) < r\}$.
	The weak and strong polynomial moments are defined in Definition \ref{def:lpnorm}.
	
	\begin{Cor}[Regularity of the measure]\label{cor:regxi}
		Let $E$ be a Euclidean vector space.
		Let $\nu$ be a strongly irreducible and proximal probability measure on $\mathrm{GL}(E)$. Let $C, \beta$ be as in Theorem \ref{th:slln} and let $\xi_\nu^\infty$ be the $\nu$-stationary measure as in Theorem \ref{th:limit-uv}. Then we have:
		\begin{equation}\label{eq:regxi}
			\forall V \in\mathrm{Gr}(E) \setminus\{E\}, \forall 0 < r \le 1,\; \xi_\nu^\infty (\mathcal{N}_r(V)) \le \zeta_\nu^{C,\beta} (|\log(r)|, +\infty)
		\end{equation}
		Let $p >0$.
		If we assume that $N_*\nu$ has finite strong $\mathrm{L}^p$ moment, then there exists a constant $C'$ such that:
		\begin{equation}\label{eq:regxislp}
			\forall V \in\mathrm{Gr}(E) \setminus\{E\},\; \int_{l \in \mathbf{P}(E)}|\log\dist(\mathbf{P}(V), l)|^p d\xi_\nu^\infty(l) \le C'.
		\end{equation}
		If we assume that $N_*\nu$ has finite weak $\mathrm{L}^p$ moment, then there exists a constant $C'$ such that:
		\begin{equation}\label{eq:regxiwlp}
			\forall V \in\mathrm{Gr}(E) \setminus\{E\}, \forall 0 < r < 1,\; \xi_\nu^\infty (\mathcal{N}_r(V)) \le C'|\log(r)|^{-p}.
		\end{equation}
	\end{Cor}
	
	Note that by Lemma \ref{lem:sumlp}, the probability distribution $\zeta_\nu^{C,\beta}$ is in the same integrability class as $N_*\nu$.
	Inequalities \eqref{eq:regxislp} and \eqref{eq:regxiwlp} follow directly from that observation and from \eqref{eq:regxi}.
	We prove Theorem \ref{th:slln} and \eqref{eq:regxi} from Corollary \ref{cor:regxi} in Section \ref{sec:lln}.

	\subsection{Contraction results without moment assumptions}

	Let $E$ be a Euclidean vector space, and let $1 \le k \le \dim(E)$ be an integer.
	We denote by $\bigwedge^k E$ the $k$-th exterior product of $E$, \ie the minimal-up-to-isomorphism space that factorises all alternate $k$-linear maps. 
	It naturally comes with a $k$-linear alternate map $E^k \to \bigwedge^k E; (v_1, \dots, v_k) \mapsto v_1 \wedge \cdots \wedge v_k$.
	We endow $\bigwedge^k E$ with the canonical Euclidean metric, which is characterized by the fact that for all family $(v_1, \dots, v_k) \in E^k$, one has $\|v_1 \wedge \cdots \wedge v_k\| \le \|v_1\|\cdots\|v_k\|$, with equality when the family $(v_1, \dots, v_k)$ is orthogonal.
	
	Let $E$ and $F$ be Euclidean spaces and let $h \in\mathrm{Hom}(E,F)$. We define the squeeze coefficient or logarithmic singular gap of $h$ as follows:
	\begin{equation}\label{eqn:defsqz}
			\mathrm{sqz}(h) := \log\left(\frac{\|h\|\|h\|}{\|h\wedge h\|}\right).
	\end{equation}
	It is the logarithm of the ratio between the first and second largest singular values (counted with multiplicity).
	Note that then by the spectral theorem, for all square matrix $h$ which is not nilpotent, we have $\frac{\sqz(h^n)}{n} \to \prox(h)$.

	\begin{Th}[Quantitative estimate of the escape speed]\label{th:escspeed}
		Let $E$ be a Euclidean space and let $\nu$ be a proximal and strongly irreducible probability distribution on $\mathrm{End}(E)$. Let $(\gamma_n)\sim\nu^{\otimes\NN}$. Write $\overline\gamma_n := \gamma_0\cdots\gamma_{n-1}$ for all $n$.
		Then there exists a positive constant $\sigma(\nu)\in (0,+\infty]$ such that almost surely $\frac{\sqz(\overline\gamma_n)}{n} \to \sigma(\nu)$. Moreover, we have the following large deviations inequalities:
		\begin{equation}\label{eqn:ldsqz}
			\forall\alpha<\sigma(\nu),\;\exists C, \beta>0 ,\; \forall n\in\NN,\; \PP(\sqz(\overline\gamma_n)\le\alpha n \cup \prox(\overline\gamma_n)\le\alpha n) \le C \exp(-\beta n).
		\end{equation}
	\end{Th}
	
	Let $\nu$ be a probability measure on $\mathrm{GL}(E)$.
	Let $(\gamma_n)\sim \nu^{\otimes\NN}$. 
	Assume that $\EE(\log\|\gamma_0\|) < +\infty$. 
	Then we know from sub-additivity~\cite{porm} that $\frac{\log\|\overline{\gamma}_n\|}{n} \to \log(\rho_1(\nu))$. 
	Let $\log(\rho_2(\nu))$ be the second Lyapunov exponent of $\nu$. 
	Again by sub-additivity, $\frac{\log\|\overline{\gamma}_n \wedge \overline{\gamma}_n\|}{n} \to \log(\rho_1(\nu)) + \log(\rho_2(\nu))$. 
	Hence $\frac{\sqz(\overline{\gamma}_n)}{n} \to \log(\rho_1(\nu)) - \log(\rho_2(\nu))$, which is therefore equal to $\sigma(\nu)$ from Theorem \ref{th:escspeed}.
	A celebrated result by Guivarc'h and Raugi~\cite{GR86} asserts that this difference is positive when $\nu$ is strongly irreducible and proximal. 
	Without the first moment assumption, the Lyapunov coefficients $\rho_1(\nu)$ and $\rho_2(\nu)$ do not make sense and in general, the sequence $\|\overline{\gamma}_n\|^{\frac{1}{n}}$ does not converge almost surely.
	Still Theorem \ref{th:escspeed} above shows that the limit $\sigma(\nu) = \lim \frac{\sqz(\overline\gamma_n)}{n}$ still exists and is a positive constant.
	In that sense, the first part of the above theorem is an extension of Guivarc'h and Raugi's theorem to all strongly irreducible and proximal probability measures.
	Moreover, the quantitative estimates \eqref{eqn:ldsqz} are new even in the setting of~\cite{GR86}.
	In fact they are key to our approach. 
	We deduce the qualitative convergence from the strong quantitative estimates.
		
	We denote by $\mathbf{P}(E)$ the projective space associated to $E$ \ie the set of vector lines in $E$. Write $[\cdot] : E\setminus\{0\} \to \mathbf{P}(E)$ for the projection map. We endow $\mathbf{P}(E)$ with the metric:
	\begin{equation}
		\dist:([x],[y])\mapsto\frac{\|x\wedge y\|}{\|x\|\|y\|}.
	\end{equation} 
	Let $h$ be a square matrix such that $\prox(h) > 0$. 
	Then the top eigenvalue of $h$ is simple and real. 
	We write $E^+(h)$ for the associated eigenspace which is a real line.
	
	\begin{Th}[Quantitative convergence of the image]\label{th:cvspeed}
		Let $ \nu $ be a strongly irreducible and proximal probability distribution on $\mathrm{End}(E) $. Let $(\gamma_n) \sim \nu^{\otimes\NN}$. There exists a random line $l^\infty \in \mathbf{P}(E)$ such that for all $\alpha < \sigma(\nu)$, there exist constants $C,\beta > 0$, such that:
		\begin{gather}\label{eqn:erc}
			\forall v \in E \setminus \{0\},\,
			\forall n\in\NN,\; 
			\PP\left(\dist([\overline\gamma_n v],l^\infty) \ge \exp(-\alpha n)\,|\,\overline\gamma_n v \neq 0\right) \le C\exp(-\beta n),\\
			\forall n\in\NN,\; \PP\left(\prox(\overline{\gamma}_n) = 0 \cup \dist(E^+(\overline{\gamma}_n),l^\infty) \ge \exp(-\alpha n)\right) \le C\exp(-\beta n)\label{eqn:lim-E+}
		\end{gather}
	\end{Th}
	
	We moreover show in Proposition \ref{prop:essker} that the set of vectors $v \in E$ such that $\sup_{n\in\NN} \PP(\overline\gamma_n v = 0) > 0$ is a countable union of proper subspaces of $E$. 
	We denote this set by $\underline{\ker}(\nu)$.
	In this proposition, we also show that $\PP(\overline\gamma_n v = 0)$ is bounded away from $1$, uniformly in $n \in \NN$ and $v \in E \setminus\{0\}$.

	Note that if two random lines $l^\infty$ and $l'^\infty$ satisfy \eqref{eqn:erc} then we have $l^\infty = l'^\infty$ almost surely.  
	We define $\xi_\nu^\infty$ to be the distribution of $l^\infty$. 
	Then $\xi_\nu^\infty$ is the only $\nu$-stationary measure on $\mathbf{P}(E)$ in the sense that $\nu * \xi^\infty_\nu = \xi^\infty_\nu$.
	Moreover, we have the following exponential mixing property.
	
	\begin{Cor}[Proximality implies exponential mixing]\label{cor:ergodic}
		Let $ \nu $ be any strongly irreducible and proximal distribution on $\mathrm{End}(E)$. There is a unique $\nu$-stationary probability distribution $\xi_\nu^\infty$ on $\mathbf{P}(E)$. Moreover, there exist constants $C, \beta$ such that for all probability distribution $\xi$ on $ \mathbf{P}(E)\setminus\underline{\ker}(\nu) $ and for all Lipschitz function $ f: \mathbf{P}(E) \to \RR$ with Lipschitz constant $\lambda(f)$, we have:
		\begin{equation}
			\forall n\in\NN,\;\left|\int_{\mathbf{P}(E)}f\mathrm{d}\xi_\nu^\infty-\int_{\mathbf{P}(E)}f\mathrm{d}\nu^{*n}*\xi\right|\le \lambda(f) C\exp(-\beta n).
		\end{equation}
	\end{Cor}
	
	Note that saying that $\xi$ is supported on $\mathbf{P}(E)\setminus\underline{\ker}(\nu)$ is not very restrictive because any measure that gives measure $0$ to all hyperplanes would satisfy that condition. 
	However, $\xi_\nu^\infty$ itself may give positive measure to some hyperplanes. 
	For example if $\nu$ is the barycentre of the Haar measure on the group of isometries and a Dirac mass $\delta_\pi$ at a projection endomorphism $\pi$, then $\xi^\infty_\nu$ is the average of the isometry-invariant measure and of the Dirac mass on the image of $\pi$. In particular $\xi^\infty_\nu$ gives positive measure to any hyperplane that contains the image of $\pi$. 
	
	Note that if $\nu$ is supported on $\mathrm{GL}(E)$, then $\underline{\ker}(\nu) = \{0\}$. 
	The existence and uniqueness of the stationary measure are well known in this case. 
	This was in fact the first step towards the formalization of boundary theory by Furstenberg~\cite{furstenberg1973boundary}. 
	Even in this case, with the pivoting technique, we get regularity results for the stationary measure which are better that the ones obtained using ergodic theory.
	
	Let $p \in (0, +\infty)$ and let $\eta$ be a probability measure on $\RR_{\ge 0}$. 
	We say that $\eta$ is strongly $\mathrm{L}^p$ if $M_p(\eta) := \int_{t=0}^{+\infty} t^{p-1}\eta(t, +\infty) dt < + \infty$ and we say that $\eta$ is weakly $\mathrm{L}^p$ if $W_p(\eta) := \sup_{t \ge 0}t^{p}\eta(t, +\infty) < +\infty$.
	
	Given $E$ a vector space and $k \le \dim(E)$, we denote by $\mathrm{Gr}_k(E)$ the set of $k$-dimensional subspaces of $E$.

\subsection{Alignment and pivotal extraction}\label{sec:intro-pivot}

	An important tool that we will use is the notion of alignment of matrices that we define as follows:
	
	\begin{Def}[Coarse alignment of matrices]\label{def:ali}
		Let $g,h$ be two matrices whose product is well defined. Let $0 < \eps \le 1$, we say that $g$ is $\eps$-coarsely aligned to $h$ and we write $g\AA^\eps h$ if we have:
		\begin{equation}\label{eqn:def-ali}
			\|g h\| \ge \eps \|g\| \|h\|.
		\end{equation} 
	\end{Def}
	
	An important observation is that \eqref{eqn:def-ali} (together with the sub-multiplicativity of the norm on $\bigwedge^2 E$) implies that $\sqz(g h) \ge \sqz(g) + \sqz(h) - 2 |\log(\eps)|$ (see Lemma \ref{lem:c-prod}). 
	
	Using the pivoting technique, we will prove theorem \ref{th:pivot} below. To give a precise statement we need to introduce some notations. 
	
	Let $\Gamma = \mathrm{GL}(E)$ or $\Gamma = \mathrm{End}(E)$. We will write $\widetilde\Gamma$ for the semi-group of words with letters in $\Gamma$ \ie the set of all tuples $\bigsqcup_{l\in\NN}\Gamma^l$, (where $\Gamma^l$ is identified with $\Gamma^{\{0, \dots, l-1\}}$ and endowed with the product $\sigma$-algebra for all $l \in\NN$) that we endow with the concatenation product
	\begin{equation*}
		\begin{array}{crcl}
			\odot : & \widetilde\Gamma \times\widetilde\Gamma & \longrightarrow & \widetilde\Gamma \\
			&((\gamma_0, \dots,\gamma_{k-1}),(\gamma'_0,\dots,\gamma'_{l-1})) \in\Gamma^k \times\Gamma^l & \longmapsto & (\gamma_0, \dots,\gamma_{k-1},\gamma'_0,\dots,\gamma'_{l-1})\in\Gamma^{k+l}.
		\end{array}
	\end{equation*}
	We also define the length functor:
	\begin{equation*}
		L :\; \widetilde\Gamma \longrightarrow \NN ; \; (\gamma_0, \dots,\gamma_{k-1}) \longmapsto k,
	\end{equation*}
	and the product functor:
	\begin{equation*}
		\Pi : \; \widetilde\Gamma \longrightarrow \Gamma ;\; (\gamma_0, \dots,\gamma_{k-1}) \longmapsto \gamma_0\cdots\gamma_{k-1}.
	\end{equation*}
	Moreover, for all $0 \le k < l$, we define $\chi_k^l : \Gamma^l \to \Gamma$ to be the $k$-th coordinate projection.
	
	Let $I$ be a countable set, let $(\zeta_i)_{i\in I}$ be a family of probability distributions on $\RR_{\ge 0}$. Let $\eta$ be a probability distribution on $\RR_{\ge 0}$.
	We say that $\eta$ dominates the family $(\zeta_i)_{i\in I}$ if there exists a constant $ C $ such that $\zeta_i(t,+\infty)\le C\eta(t/C-C,+\infty)$ for all $t \in\RR_{\ge 0}$ and all $i\in I$. 
	
	Let $(\eta_i)_{i\in I}$ be a family of probability distribution on $\RR_{\ge 0}$. We say that $(\eta_i)$ has a bounded exponential moment if there exist constants $C, \beta> 0$ such that $\eta_i(t, +\infty) \le C\exp(-\beta t)$ for all $t \in \RR$ and all $i \in I$. 
	Note that saying that a family $(\eta_i)_{i\in I}$ has a bounded exponential moment is not the same as saying that each $\eta_i$ has a finite exponential moment because the exponent $\beta$ and the constant $C$ may depend on the index $i \in I$.
	We say that a family of random variables has a bounded exponential moment if the family of their distributions have.
	
	Given $A$ a measurable event \ie a measurable subset of a measurable space $X$, we write $\mathds{1}_A$ for the indicator function of $A$, it is the measurable function that takes value $1$ on $A$ and value $0$ on $X \setminus A$.

	\begin{Th}[Pivotal extraction]\label{th:pivot}
		Let $E$ be a Euclidean vector space. Let $\Gamma \in \{\mathrm{End}(E),\mathrm{GL}(E)\}$ and let $N$ be a continuous map defined on $\Gamma$.
 		Let $\nu$ be a strongly irreducible and proximal probability distribution over $ \Gamma$. Let $\rho < 1$ and let $K \in\NN$. There exist $0< \eps \le 1$ and three probability distributions $(\tilde{\kappa}_0,\tilde{\kappa}_1,\tilde{\kappa}_2)$ supported on $\widetilde{\Gamma}$ that satisfy conditions \eqref{1} to \eqref{6}. For all $i\in\{0,1,2\}$, we write $\kappa_i := \Pi_*\tilde\kappa_i$. 
		\begin{enumerate}
			\item We have $\tilde{\kappa}_0\odot(\tilde\kappa_1\odot\tilde\kappa_2)^{\odot\NN} = \nu^{\otimes\NN}$, we say that $\tilde{\kappa}_0\otimes(\tilde\kappa_1\otimes\tilde\kappa_2)^{\otimes\NN}$ is an extraction of $\nu^{\otimes\NN}$.\label{1}
			\item The push-forward measures $L_*\tilde\kappa_0$ and $L_*\tilde\kappa_2$ have a bounded exponential moment and $L_*\tilde\kappa_1 = \delta_m$ is the Dirac mass at a positive integer denoted by $m$.\label{2}
			\item The measure $\tilde\kappa_1$ has compact support in $\tilde\Gamma$ and ${\kappa}_1\{\gamma\in\Gamma\,|\,\sqz(\gamma)\ge K |\log(\eps)| + K\log(2)\} = 1$. \label{3}
			\item Given $(g_n)_{n\in\NN}\sim\kappa_0\otimes(\kappa_1\otimes\kappa_2)^{\otimes\NN}$, and $0\le i < j < k \in\NN$, we have $g_i\cdots g_{j-1} \AA^{\frac{\eps}{4}} g_{j} \cdots g_{k-1}$ almost surely.\label{4}
			\item For all $g \in \Gamma$, we have $ \kappa_1\{\gamma\in\Gamma\,|\,g \AA^\eps \gamma\}\ge 1-\rho$ and $ \kappa_1 \{\gamma\in\Gamma\,|\, \gamma \AA^\eps g\}\ge 1-\rho$.\label{5}
			\item Let $i\in\{0,2\}$ and let $k < l$ be integers such that $L_*\tilde{\kappa}_i\{l\} > 0$. Let:
			\begin{equation*}
				\zeta_{i,k,l} := N_*(\chi_k^l)_*\frac{(\mathds{1}_{L=l})\tilde\kappa_i}{L_*\tilde\kappa_i\{l\}}
			\end{equation*}
			be the push-forward by $N$ of the conditional distribution of the $k$-th marginal of $\tilde{\kappa}_i$ relatively to the event $L(\tilde{g}) = l$. Then the family $(\zeta_{i,k,l})$ is dominated by the push-forward measure $N_*\nu$.\label{6}
		\end{enumerate}
	\end{Th}
	
	Only points \eqref{1} to \eqref{5} are used in the proofs of Theorems \ref{th:escspeed} and \ref{th:cvspeed} and point \eqref{6} is more technical and is only used in the proof of Theorem \ref{th:slln}. 
	
	Note that if we moreover assume that $N$ is sub-additive, then points \eqref{6} and \eqref{2} imply that for $i \in \{0,2\}$ the distribution $N_*\kappa_i$ is virtually dominated by $N_*\nu$, 
	in the sense that there exist constants $C, \beta > 0$ such that $N_* \kappa_i(t,+\infty) \le \sum_{k = 1}^{+\infty}C\exp(-\beta k)N_*\nu(t/(C k)-C,+\infty)$ for all $t \in\RR_{\ge 0}$.
	This is a consequence of Lemmas \ref{lem:descrisum} and Lemma \ref{lem:simplification}. 
	Then by Lemma \ref{lem:sumlp}, it means that if $N_*\nu$ has finite $p$-th moment, then $N_*\kappa_i$ also has.
	
	Note also that if $N_*\nu$ has a finite exponential moment, then $N_*\kappa_i$ also has for all $i \in \{0,1,2\}$. 
	However, this is not a consequence of \eqref{6} but a consequence of \eqref{2} and of Lemma \ref{lem:sumexp}.

	\subsection{Background}
	
	The study or products of random matrices bloomed with the eponym article~\cite{porm} where Furstenberg and Kesten construct an escape speed for the logarithm of the norm using the sub-additivity. This proof was generalized by Kingman's sub-additive ergodic Theorem~\cite{K68}. This article followed the works of Bellman~\cite{Bellman} who showed the almost sure convergence of the rescaled logarithms of coefficients as well as a central limit theorem for one specific example.
	In~\cite{porm} Furstenberg and Kesten show that we have a law of large numbers for the norm under a strong $\mathrm{L}^1$ moment condition for $\log\|\cdot\|$.
	For matrices that have positive entries and under an $\mathrm{L}^\infty$ moment condition, they show that moreover, we have a law of large numbers for the coefficients (entries) and under an additional $\mathrm{L}^{2+\delta}$ moment assumption, they show that we have a central limit Theorem.
	These works on matrices inspired the theory of measurable boundary theory for random walks on groups~\cite{furstenberg1973boundary}. In~\cite{livrebl}, Bougerol and Lacroix give an overview of the field of study with applications to quantum physics. 
	
	In~\cite{GR86}, Guivarc'h and Raugi show a qualitative version of Theorem \ref{th:escspeed}: in the case when $\nu$ is proximal and strongly irreducible, the two top Lyapunov exponents are distinct. In~\cite{GR89} the same authors show that we have almost sure convergence of the limit flag for totally strongly irreducible distributions. 
	In~\cite{Goldsheid1989LyapunovIO} Goldsheid and Margulis show that the distribution $\nu$ is proximal and totally strongly irreducible when the support of $\nu$ generates a Zariski-dense sub-group of $\mathrm{SL}(E)$. 
	
	In~\cite{livreBQ} Yves Benoist and Jean-François Quint give an extensive state of the art overview of the field of study with an emphasis on the algebraic properties of semi-groups. Later, in~\cite{xiao2021limit} Xiao, Grama and Liu use~\cite{CLT16} to show that coefficients satisfy a law of large numbers under some technical $ \mathrm{L}^2 $ moment assumption. 
	We can also mention~\cite{grama2020zeroone} and~\cite{xiao2022edgeworth} that give other probabilistic estimates for the distribution of the coefficients. 
	The strong law of large numbers and central limit-theorem for the spectral radius were proven by Aoun and Sert in~\cite{aoun2020central} and in~\cite{aoun2021law} under an $\mathrm{L}^2$ moment assumption. 
	
	The importance of alignment of matrices was first noted in~\cite{AMS-esmigroup} along with the importance of Schottky sets. Those notions were then used by Aoun in~\cite{RA2011} where he uses it to show that independent draws of an irreducible random walk that has finite exponential moment generate a free group outside of an exponentially rare event (note that the pivoting technique allows us to drop the finite exponential moment assumption). In~\cite{cuny2016limit} and~\cite{cuny2017komlos}, Cuny, Dedecker, Jan and Merlevède give KMT estimates for the behaviour of $\left(\log\|\overline\gamma_n\|\right)_{n\in\NN}$ under $\mathrm{L}^p$ moment assumptions for $p > 2$.
	
	The main difference between these previous works and this paper is that the measure $\nu$ has to be supported on the General Linear group $\mathrm{GL}(E)$ for the above methods to work.
	Indeed, they rely of the existence of the stationary measure $\xi_\nu^\infty$ on $\mathbf{P}(E)$, which is a consequence of the fact that $\mathrm{GL}(E)$ acts continuously on $\mathbf{P}(E)$, which is compact.
	Some work has been done to study non-invertible matrices in the specific case of matrices that have real positive coefficients. In~\cite{porm}, Furstenberg and Kesten show limit laws for the coefficients under an $\mathrm{L}^\infty$ moment assumption, in~\cite{AM87} and~\cite{KS87_semigroup} Mukherjea, Kesten and Spitzer show some limit theorems for matrices with non-negative entries that are later improved by Hennion in~\cite{HH97} and more recently improved by Cuny, Dedecker and Merlevède in~\cite{cuny2023limit}. 
	
	In~\cite{lepage82}, Le Page shows the exponential mixing property by exhibiting a spectral gap for the action of $ \nu $ on the projective space under some moments assumptions on $ \nu $.
	The large deviations inequalities were already known for the norm in the specific case of distributions having finite exponential moment by the works of Sert~\cite{sert2018large}.

	\subsection{Method used}\label{toy-model}
	
	To prove the results, we use Markovian extractions. The idea is to adapt the following "toy model" construction to the case of matrices.

	Let $ G=\langle a,b,c|a^2=b^2=c^2=\mathbf{1}_\Gamma\rangle $ be the free right angle Coxeter group with $ 3 $ generators. One can see the elements of $G$ as reduced words in $\{a,b,c\}$, \ie finite sequences of letters of type $(x_1, \dots, x_n)$ without double letters in the sense that $x_i \neq x_{i+1}$ for all $1 \le i < n$. 
	We write $\Sigma := \{a,b,c\}^{(\NN)}$ for the set of words in the alphabet $\{a,b,c\}$.
	We write $\mathbf{1}_\Sigma$ for the empty word, which is the identity element of $\Sigma$. 
	We write $ \odot $ for the concatenation product on $\Sigma$ and $\Pi: \Sigma \to G$ the word reduction map which is a monoid morphism.
	
	We consider the simple random walk on the $3$-tree, seen as the Cayley graph of $G$. Draw a random independent uniformly distributed sequence of letters $(l_n)_{n\in\NN} \in \{a,b,c\}^\NN$. Then for every $ n\in\NN $, write $g_n := l_0 \cdots l_{n-1} \in G$ for the position of the random walk at step $n$ and $\tilde{g}_n := (l_0, \dots, l_{n-1})$ the word encoding the trajectory of the random walk up to step $ n $. Then we know that $ (g_n) $ almost surely escapes to a point in $\partial G$, the set of infinite simple words. To prove it, we can show, using Markov's inequality, that $ \PP(g_n=\mathbf{1}_G) \le \left(\frac{8}{9}\right)^{n/2} $.
	Indeed, given $n \in\NN$, if $|g_n|\ge 1$, then $|g_{n+1}| = |g_n| + 1$ with probability $\frac{2}{3}$ and $|g_{n+1}| = |g_n| - 1$ with probability $\frac{1}{3}$ and if $|g_n| = 0$, then $|g_{n+1}| = |g_n| + 1$ with probability $1$. 
	It implies that:
	\begin{equation}
		\forall n\in\NN,\;\EE\left(\sqrt{2}^{-|g_{n+1}|}\right) \le \frac{2\sqrt{2}}{3} \EE\left(\sqrt{2}^{-|g_{n}|}\right). \quad \text{Hence} \quad \forall n\in\NN,\; \EE\left(\sqrt{2}^{-|g_{n}|}\right) \le \left(\frac{8}{9}\right)^{n/2}.\label{eq:markov-exp}
	\end{equation}
	Therefore $(g_n)$ visits $\mathbf{1}_G$ only finitely many times. After that it gets trapped in a branch (the set of simple words starting with a given letter $ x_1 \in \{a,b,c\}$). Then using the same argument, $(g_n)$ visits the first node of this branch only finitely many times and then escapes along the branch starting with $x_1x_2$ for some $x_2 \neq x_1$ and by induction, one can show that $(g_n)$ escapes along a branch $ (x_1, x_2, \dots) $ (\ie an infinite reduced word). 
	
	By symmetry, one can show that for all $k > 1$, the distribution of the letter $x_k$ knowing $x_1, \dots, x_{k-1}$ is the uniform distribution on $\{a,b,c\} \setminus \{x_{k-1}\}$.
	For all $k \ge 1$, we define the $k$-th pivotal time of the sequence $(l_n)$ as $t_k := \min\{t \in \NN\,|\, \forall j \ge t,\, |g_j| \ge k\}$.
	For example $t_0 = 0$ and $t_1$ is the first time after the last visit in $\mathbf{1}_G$.
	Then for $k \ge 2$, the time $t_k$ follows the time of last visit in the closed ball of radius $k-1$.
	An interesting observation is that for all $k \ge 1$, we have $x_k = l_{t_k - 1} = l_{t_{k-1}} l_{t_{k-1}+1} \cdots l_{t_k-1}$.
	
	Then instead of drawing the sequence $(l_n)_{n\in\NN}$ of letters, we can draw the limit $(x_n)_{n\in\NN}$ first and then the letters $(l_n)_{n\in\NN}$ as follows.
	
	Write $X = \{a,b,c,s\}$, ($s$ like "start") and endow $X$ with a transition kernel $p$ such that $p(i,j) = \frac{1}{2} $ for all $i \neq j \in \{a,b,c\}$ and $p(s,i) = \frac{1}{3}$ for $i \in \{a,b,c\}$.
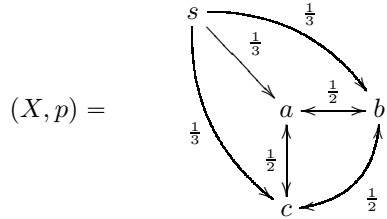
	\begin{equation*}
		\xymatrix{& s\ar[rd]^{\frac{1}{3}}\ar@/^1pc/[rrd]^{\frac{1}{3}}\ar@/_1pc/[rdd]_{\frac{1}{3}}&&\\
			(X,p)=& &a\ar@{<->}[r]^{\frac{1}{2}}\ar@{<->}[d]_{\frac{1}{2}}&b\ar@/^1pc/@{<->}[ld]^{\frac{1}{2}}\\
			& & c &}
	\end{equation*}
	Let $x_0 = s$ and draw a Markov chain $(x_n)_{n\in\NN}$ in $(X,p)$. It means that we have:
	\begin{equation*}
		\forall n\in\NN,\forall l\in X, \PP(x_{n+1}=l\,|\,x_0,\dots,x_{n}) = p(x_{n},l).
	\end{equation*}
	Then the sequence $(x_k)_{k\ge 1}$ has the same distribution as the sequence $ l_{t_k-1} $ defined above. Moreover, the distribution of the word $(l_{t_k}, \dots, l_{t_{k+1}-1})$ only depends on $x_k$ and $x_{k+1}$ and not on the time $k \ge 1$. Write $\tilde\nu_{a,b}$ for the distribution of $(l_{t_k}, \dots, l_{t_{k+1}-1})$ knowing that $l_{t_k-1} = a$ and $l_{t_{k+1}-1} = b$ and write $\tilde\nu_{s,a}$ for the distribution of the word $(l_0, \dots, l_{t_1-1})$ knowing that $l_{t_1-1} = a$. Both are probability distributions on $ \Sigma$. In the same fashion, we define the whole decoration:
	
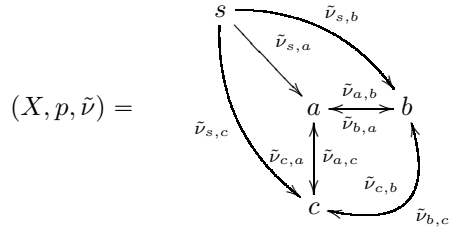
\begin{equation*}
		\xymatrix{
			& s \ar[rd]^{\tilde\nu_{s,a}} \ar@/^1pc/[rrd]^{\tilde\nu_{s,b}} \ar@/_1pc/[rdd]_{\tilde\nu_{s,c}} &&\\
			(X,p,\tilde\nu)= & & a \ar@{<->}[r]^{\tilde\nu_{a,b}}_{\tilde\nu_{b,a}} \ar@{<->}[d]^{\tilde\nu_{a,c}}_{\tilde\nu_{c,a}} & b \ar@/^2pc/@{<->}[ld]^{\tilde\nu_{b,c}}_{\tilde\nu_{c,b}}\\
			& & c &}
	\end{equation*}
	Then instead of drawing the $(l_n)$ 's uniformly and independently, one can simply draw a random sequence of words $(\tilde{w}_k)$ with distribution $\bigotimes\tilde\nu_{x_k,x_{k+1}}$ relatively to $(x_n)$. 
	Then for every $ k\in\NN $, the random word $\tilde{w}_k$ has the distribution of $(l_{t_k},\dots,l_{t_{k+1}-1})$ and the infinite word $W = \bigodot_{k=0}^\infty \tilde{w}_k\in\{a,b,c\}^\NN $ has the distribution of the infinite word $L = (l_0,l_1,l_2,\dots)$. 
	Note also that for all $k\in\NN$, one has $\Pi\tilde{w}_k = x_{k+1}$ and $\tilde{w}_k$ has no prefix whose product is $x_k$.
	
	Now, we consider a filtration $(\mathcal{F}_k)_{k\ge 0}$ such that $ x_k $ and $\tilde{w}_{k-1}$ are $ \mathcal{F}_k $-measurable for all $k \ge 1$, the distribution of $ x_{k+1} $ knowing $ \mathcal{F}_k $ is $ p(x_k,\cdot) $ and the distribution of $\tilde{w}_k$ knowing $ \mathcal{F}_k $ and $x_{k+1}$ is $ \tilde\nu_{x_k,x_{k+1}} $. Now the fact that a time $t$ is pivotal or not is decided as soon as $\tilde{w}_0 \odot \cdots \odot \tilde{w}_{k-1}$ has length at least $t$. In particular the event ($ t $ is a pivotal time) is $\mathcal{F}_t$-measurable. However, given $ (\mathcal{C}_n)_{n\in\NN} $ the cylinder filtration associated to the random sequence $(l_n)_{n\in\NN}$, the event ($ t $ is a pivotal time) is never $\mathcal{C}_n$-measurable whatever the choice of $ n,t\in\NN $.
	
	This construction gives a proof of the exponential large deviations inequalities for the random walk $(g_n)$. This is not the simplest proof but it shows how and why we want to use the setting of Markovian extractions.

	\begin{equation}\label{eqn:ldevtree}
		\exists\sigma>0,\; \forall \eps>0,\; \exists C,\beta > 0, \;\forall n\in\NN,\PP\left(\left||g_n| - n\sigma\right| \ge \eps n\right)\le C\exp(-\beta n).
	\end{equation}

	\begin{proof}
		Let $(l_0, l_1, l_2, \dots) = \tilde{w}_0 \odot \tilde{w}_1 \odot \tilde{w}_2 \odot \cdots$ be as above. 
		We associate to every integer $n \in \NN$ a pair of indices $k \in \NN$, $r \in \{0, \dots, |\tilde{w}_k|-1\}$ such that $n = |\tilde{w}_0| + \dots + |\tilde{w}_{k-1}| + r$. 
		This means that $l_{n-1}$ is the $r$-th letter of $ \tilde{w}_k $ and then by triangular inequality, we have $k - r \le |g_n| \le k + r$ because $k = |x_1\cdots x_k|$ and $r \ge |l_{n-r} \cdots l_{n-1}|$. 
		
		Then note that the lengths $(|\tilde{w}_k|)_{k\ge 1}$ are independent, identically distributed random variables that are independent of $|\tilde{w}_0|$.
		Moreover, they all have a finite exponential moment by \eqref{eq:markov-exp}. 
		By Lemma \ref{lem:curstep}, $r$ also has an exponential moment which is uniformly bounded in $n$.
		
		Let $\sigma := \frac{1}{\mathbb{E}(|\tilde{w}_1|)} = \frac{1}{3}$ and let $ \eps>0 $. Then by the classical large deviations inequalities (see Lemma \ref{lem:proba:ldev} and Lemma \ref{lem:proba:ldevcompo} \eqref{compo:rec}), we have:
		$ \PP(|k - n\sigma |\ge n\eps /2)\le C\exp(-\beta' n) $ 
		for some $ C,\beta'>0 $ and for all $ n $. 
		Now note that $||g_n| - n \sigma | \le |k - n \sigma| + r$ so we have \eqref{eqn:ldevtree} by Lemma \ref{lem:proba:ldevcompo} \eqref{compo:shift}.
	\end{proof}

	\subsection{About the pivoting technique}
	
	In the second part of this article we mainly use the tools introduced in~\cite{pivot}, some of them having been introduced or used in former works like~\cite{bmss20} where Adrien Boulanger, Pierre Mathieu, Cagri Sert and Alessandro Sisto state large deviations inequalities from below for random walks in discrete hyperbolic groups or~\cite{devin} where Mathieu and Sisto show some bi-lateral large deviations inequalities in the context of distributions that have a finite exponential moment. 
	In~\cite{pivot} Sébastien Gouëzel uses the pivoting technique in the setting of hyperbolic groups to get large deviations estimates bellow the escape speed and to show the continuity of the escape speed. 
	For us, the most interesting part of Gouëzel's work is the "toy model" described in section 2. 
	In~\cite{choi} Inhyeok Choi applies the pivoting technique to show results that are analogous to the ones of Gouëzel for the mapping class group of an hyperbolic surface. 
	In~\cite{CFFT22}, Chawla, Forghani, Frisch and Tiozzo use another view of the pivoting technique and the results of~\cite{pivot} to show that the Poisson boundary of random walk with finite entropy on a group that has an acylindrical action on an hyperbolic space is in fact the Gromov Boundary of said space. 
	I believe that similar method can be used to describe the Poisson boundary of a totally strongly irreducible random walk that has finite entropy, in the sense of Conjecture \ref{conj:poisson}.

	\subsection{Structure of this paper}

	In Section \ref{kak} of this article, we state some local-to-global properties for alignment of matrices. 
	In Section \ref{markov}, we state some preliminary results about random products of non-invertible matrices. 
	In section \ref{sec:pivot} Theorem \ref{th:pivot-extract}, we state an abstract version of the construction of the pivoting extraction using the pivoting technique as in~\cite[section~2]{pivot} and prove Theorem \ref{th:pivot} as a corollary of that statement. 
	Then in Section \ref{results} we give complete proofs of Theorems \ref{th:escspeed}, \ref{th:cvspeed} and \ref{th:llnc} using the pivoting technique and Theorem \ref{th:pivot}.
	Section \ref{anex:prob} is an appendix where we prove classical results for real valued random variables, we state these lemmas in a convenient way to be able to use them through this paper.

	\section{Local-to-global properties for the alignment of matrices}\label{kak}
	
	In this section, we describe the geometry of the monoid $\Gamma := \mathrm{End}(E)$ for $E$ a Euclidean space. 
	We can think of $\KK = \RR$ but all the proofs work the same when $\KK = \CC$ or when $\KK$ is a ultra-metric field. 
	
	Given $E$ a $\KK$-vector space, we will identify $E$ with $\mathrm{Hom}(\KK, E)$.
	Note that up to choosing a canonical basis for all Euclidean spaces, linear maps between Euclidean spaces can be seen as matrices.
	Moreover, vectors and linear form can also be seen as matrices.

	We want to translate ideas of hyperbolic geometry into the language of products of endomorphisms. 
	The idea is to exhibit a local-to global property in the same fashion as~\cite[Theorem~4]{Cannon1984}. 
	That way we can adapt the arguments of~\cite{pivot} to the setting of products of random matrices.

	\subsection{Alignment and squeezing coefficients}\label{sec:ali-sqz}
	
	We remind the definition of the singular gap and of the distance in the projective space. Note that given $x,y$ two vectors, we have the characterization $\|x \wedge y\| =  \min_{a\in\KK} \|x - y a\|\|y\|$. 
	Therefore, given $h\in\mathrm{Hom}(E,F)$, we have $\|h\wedge h\| = \max_{x,y}\min_{a\in\KK} \frac{\|h(x - y a)\|\|h(y)\|}{\|x\|\|y\|}$.
	
	\begin{Def}[Singular gap]\label{def:sqz}
		Let $E,F$ be Euclidean vector spaces and $h\in\mathrm{Hom}(E,F)\setminus\{0\}$. We define the first (logarithmic) singular gap, or squeeze coefficient of $h$ as:
		\begin{equation*}
			\sqz(h) := \log\left(\frac{\|h\|^2}{\| h \wedge h\|}\right) \in [0, + \infty].
		\end{equation*}
	\end{Def}
	
	\begin{Def}[Distance between projective classes]\label{def:distance}
		Let $E$ be a Euclidean space. We denote by $\mathbf{P}(E)$ the projective space of $E$ \ie the set of lines in $E$, endowed with the distance map $\dist$ which is characterized by:
		\begin{equation}\label{distanceprojective}
			\forall x, y \in E \setminus \{0\},\;\dist([x],[y]) = \frac{\|x\wedge y\|}{\|x\|\|y\|} = \min_{a \in \KK}\frac{\|x - y a\|}{\|x\|}.
		\end{equation}
	\end{Def}
	
	\begin{Lem}[Lipschitz property for the norm cocycle]\label{lem:product-is-lipschitz}
		Let $E$ and $F$ be Euclidean spaces and let $f \in\mathrm{Hom}(E,F)\setminus\{0\}$. Let $x, y \in E \setminus\{0\}$, we have:
		\begin{equation}
			\left|\frac{\|fx\|}{\|f\|\|x\|}- \frac{\|fy\|}{\|f\|\|y\|}\right| \le \dist([x],[y]).
		\end{equation}
	\end{Lem}
	
	\begin{proof}
		Let $f \in\mathrm{Hom}(E,F)$ and let $x,y \in E$ be unit \ie $\|f\| = \|x\| = \|y\| =1$. 
		We show that $\|fx\| \le \|fy\| + \dist([x],[y])$ and conclude by homogeneity and by symmetry.
		Let $c \in\KK$ be such that $\|x - y c\| = \min_{a \in \KK} \|x - y a\|$.
		Then by Definition \ref{def:distance}, we have $\dist([x],[y]) = \|x - y c\|$. 
		Moreover $|c| \le 1$ by property of the orthogonal projection.
		By triangular inequality and by definition of the norm, we have:
		\begin{equation*}
			\|fx\|  \le \|fyc\| + \|f(x-yc)\| 
			\le \|fy\||c| + \|f\|\|x - y c\| 
			\le \|fy\| + \dist([x],[y]).\qedhere
		\end{equation*}
	\end{proof}
	
	We remind that given $g$ and $h$ two matrices such that the product $gh$ is well defined and given $0< \eps \le 1$, we write $g \AA^\eps h$ when $\|gh\| \ge \eps \|g\|\|h\|$. We also remind that given $h \in\mathrm{Hom}(E, F)$, we write $h^*\in\mathrm{Hom}(F^*, E^*)$ for the map $f \mapsto fh$, we call it the transpose of $h$.

	\begin{Def}\label{def:dom}
		Let $E,F$ be Euclidean spaces and $h\in\mathrm{Hom}(E,F)\setminus\{0\}$. Let $ 0 < \eps \le 1$. We define $V^\eps(h) := \{x\in E;\|h x\| \ge \eps \|h\|\|x\|\}$ and $U^\eps(h):=h(V^\eps(h))$ and $W^\eps(h) := U^\eps(h^*)$.
	\end{Def}
	
	Note that the families $\left(V^\eps(h)\right)_{0 < \eps \le 1}$, $\left(U^\eps(h)\right)_{0 < \eps \le 1}$ and $\left(W^\eps(h)\right)_{0 < \eps \le 1}$, are decreasing for the inclusion order. 
	
	Note also that for $h$ an endomorphism of rank one, and for all $0 < \eps \le 1$, the cone $U^\eps(h)$ is the image of $h$ so it has diameter $0$ in the projective space. 
	
	The idea to have in mind is that given $h$ a matrix that has a large singular gap $U^\eps(h)$ will have a small diameter in the following sense.
	
	\begin{Lem}\label{lem:cont-prop}
		Let $E,F$ be Euclidean spaces, let $h\in\mathrm{Hom}(E,F)\setminus\{0\}$ and let $0 < \eps \le 1$. Let $u \in U^1(h)\setminus\{0\}$ and let $u' \in U^\eps(h)\setminus\{0\}$. Then we have:
		\begin{equation}\label{eq:cont-prop}
			\dist([u], [u']) \le \frac{\exp(-\sqz(h))}{\eps}.
		\end{equation}
	\end{Lem}
	
	\begin{proof}
		Let $v \in V^1(h)$ and let $v'\in V^\eps(h)$ be such that $u = hv$ and $u' = hv'$ Then we have $u\wedge u' = \bigwedge^2 h(v \wedge v')$ so:
		\begin{equation*}
			\|u\wedge u'\|\le \left\|\Wedge^2 h\right\| \|v \wedge v'\|.
		\end{equation*}
		Now saying that $v \in V^1(h)$ and $v'\in V^\eps(h)$, means that $\|u\| = \|h\| \|v\|$ and $\|u'\| \ge \eps \|h\| \|v'\|$. Hence:
		\begin{equation*}
			\|u\| \|u'\| \ge \eps \|h\|^2\|v\|\|v'\|.
		\end{equation*}
		Then by taking the quotient, we have:
		\begin{equation*}
			\frac{\|u\wedge u'\|}{\|u\| \|u'\|} \le \frac{\left\|\bigwedge^2 h\right\|}{\eps \|h\|^2}\frac{\|v\wedge v'\|}{\|v\| \|v'\|}\le \frac{\left\|\bigwedge^2 h\right\|}{\eps \|h\|^2}.
		\end{equation*}
		By definition, the term on the left is $\dist([u], [u'])$ and the term on the right is $\frac{\exp(-\sqz(h))}{\eps}$.
	\end{proof}
	
	Lemma \ref{lem:cont-prop} tells us that the projective image of $U^\eps(h)$ has diameter at most $\eps$ as long as $\sqz(h) \ge 2|\log(\eps)| + \log(2)$. With the toy model analogy, the condition $\sqz(h) \ge 2|\log(\eps)| + \log(2)$ will play the role of the condition for  word to be non-trivial. 
	We will extensively use the following simple remarks.
	
	\begin{Lem}
		Let $g$ and $h$ be non-zero matrices such that the product $gh$ is well defined and let $0 < \eps \le 1$. We have $g\AA^\eps h$ if and only if $h^*\AA^\eps g^*$. Moreover $\sqz(h^*)=\sqz(h)$.
	\end{Lem}
	
	\begin{proof}
		This is a consequence of three well known facts. One is that we have $\|h\| = \|h^*\|$ for all homomorphism $h$. One way of seeing that is to notice that the operator norm admits the following (obviously symmetric) characterization:
		\begin{equation*}
			\forall E,\,\forall F,\,\forall h\in\mathrm{Hom}(E,F),\,\|h\| = \max_{\substack{f \in F^* \setminus\{0\} \\ v \in E\setminus\{0\}}} \frac{|f h v|}{\|f\|\|v\|}.
		\end{equation*}
		The second fact is that $(gh)^* = h^* g^*$. It implies that for all non trivial $g,h$, we have $\frac{\|gh\|}{\|g\|\|h\|} = \frac{\|h^* g^*\|}{\|h^*\|\|g^*\|}$. The third fact is that $h^*\wedge h^* = (h\wedge h)^*$. It implies that $\sqz(h^*)=\sqz(h)$.
	\end{proof}

	\begin{Lem}\label{lem:sharp-ali}
		Let $g$ and $h$ be non-zero matrices such that the product $gh$ is well defined and let $0 < \eps \le 1$. 
		If there exist $u \in U^1(h)\setminus\{0\}$ and $w \in W^1(g)\setminus\{0\}$ such that $\frac{|wu|}{\|u\|\|w\|} \ge \eps$, then $g\AA^\eps h$.
		If $g\AA^\eps h$, then there exist $u \in U^\eps(h)\setminus\{0\}$ and $w \in W^\eps(g)\setminus\{0\}$ such that $\frac{|wu|}{\|w\|\|u\|} \ge \eps$ and $\frac{\|gu\|}{\|g\|\|u\|} \ge \eps$ and $\frac{\|wh\|}{\|w\|\|h\|} \ge \eps$.
	\end{Lem}
	
	\begin{proof}
		Let $u \in U^1(h)\setminus\{0\}$ and $w \in W^1(g)\setminus\{0\}$.
		Assume that $\frac{|wu|}{\|u\|\|w\|} \ge \eps$.
		Let $f \in V^1(g^*)$ and let $v \in V^1(h)$ be such that $w = fg$ and $u = hv$.
		We have $\frac{|fghv|}{\|fg\|\|hv\|} \ge \eps$, therefore $\frac{|fghv|}{\|f\|\|g\|\|h\|\|v\|} \ge \eps$, so $\frac{\|gh\|}{\|g\|\|h\|}\ge \eps$, which means that $g \AA^\eps h$. 
		This proves the first implication of Lemma \ref{lem:sharp-ali}.
		
		Now assume that $g \AA^\eps h$.
		Let $f \in E^*\setminus\{0\}$ and let $v \in E\setminus\{0\}$ be such that $|fghv| = \|f\|\|gh\|\|v\|$. 
		Then $\frac{|fghv|}{\|f\|\|g\|\|h\|\|v\|} \ge \eps$. 
		Let $u := hv$ and let $w := fg$. 
		Then we have:
		\begin{equation*}
			\eps \le \frac{|fghv|}{\|f\|\|g\|\|h\|\|v\|} = \frac{|wu|}{\|w\|\|u\|} \frac{\|fg\|}{\|f\|\|g\|} \frac{\|hv\|}{\|h\|\|v\|} = \frac{|fgu|}{\|f\|\|g\|\|u\|}\frac{\|hv\|}{\|h\|\|v\|}= \frac{\|fg\|}{\|f\|\|g\|} \frac{|whv|}{\|w\|\|h\|\|v\|}.
		\end{equation*}
		All factors are in $[0,1]$ so $\frac{|wu|}{\|w\|\|u\|} \ge \eps$ and $\frac{\|gu\|}{\|g\|\|u\|} \ge \frac{\|fgu\|}{\|f\|\|g\|\|u\|} \ge \eps$ and $\frac{\|wh\|}{\|w\|\|h\|} \ge  \frac{|whv|}{\|w\|\|h\|\|v\|}\ge \eps$.
		Moreover $\frac{\|hv\|}{\|h\|\|v\|} \ge \eps$ so $v \in V^\eps(h)$ and therefore $u \in U^\eps(h)$. 
		We also have $\frac{\|fg\|}{\|f\|\|g\|} \ge \eps$ so $w \in W^\eps(g)$. 
		This proves the second implication  of Lemma \ref{lem:sharp-ali}.
	\end{proof}

	\begin{Lem}\label{lem:c-prod}
		Let $g$ and $h$ be non-zero matrices such that the product $gh$ is well defined and let $0 < \eps \le 1$. Assume that $g\AA^\eps h$. Then one has:
		\begin{align}
			\sqz(gh) & \ge\sqz(g) + \sqz(h) - 2|\log(\eps)|.\label{eqn:lenali}
		\end{align}
		Moreover, for every non-zero vectors $u\in U^1(g)\setminus\{0\}$, and $u'\in U^1(gh)\setminus\{0\}$, we have:
		\begin{equation}\label{eqn:limbd}
			\dist([u],[u']) \le \frac{1}{\eps}\exp(-\sqz(g)).
		\end{equation}
	\end{Lem}
	
	\begin{proof}
		Note that the norm of the $\wedge$ product is sub-multiplicative because it is a norm so:
		\begin{equation}\label{eqn:m2}
			\|gh \wedge gh\|\le \|g \wedge g\|\|h \wedge h\|.
		\end{equation} 
		So if we do $2\log\eqref{eqn:def-ali} -\log\eqref{eqn:m2}$ we find \eqref{eqn:lenali}.
		
		Now to prove \eqref{eqn:limbd}, we only need to show that $U^1(gh) \subset U^\eps(g)$ and use \eqref{eq:cont-prop} from Lemma \ref{lem:cont-prop}. Indeed, consider $v \in V^1(gh)$, then one has $\|ghv\| \ge \eps \|g\|\|h\| \|v\| \ge \eps \|g\|\|h v\|$ which means that $hv \in V^\eps(g)$, therefore $ghv \in U^\eps(g)$ and we can apply Lemma \ref{lem:cont-prop}.
	\end{proof}

	\begin{Lem}\label{lem:triple-ali}
		Let $f$, $g$ and $h$ be non-zero matrices such that the product $fgh$ is well defined and let $0 < \eps \le 1$. Assume that $f \AA^\eps g \AA^\eps h$ and that $\sqz(g) \ge 2|\log(\eps)| + 2\log(2)$. Then $\sqz(fgh)\ge \sqz(g) - 4|\log(\eps)|-2\log(2)$.
	\end{Lem}
	
	\begin{proof}
		Let $u \in U^\eps(g)\setminus\{0\}$, let $u'\in U^1(gh)\setminus\{0\}$ and let $u'' \in U^1(g)\setminus\{0\}$ be non-trivial vectors. By Lemma \ref{lem:cont-prop}, we have $\dist([u],[u''])\le \frac{\eps}{4}$. By Lemma \ref{lem:c-prod}, we have $\dist([u''],[u'])\le \frac{\eps}{4}$. 
		So by triangular inequality, we have $\dist([u],[u'])\le \frac{\eps}{2}$. 
		
		Let $v \in V^{1}(fg)\setminus\{0\}$. 
		We have $\|fgv\| = \|fg\|\|v\| \ge \eps \|f\|\|g\|\|v\|$. 
		Hence $v\in V^\eps(g)\setminus\{0\}$ therefore $gv \in U^\eps(g)\setminus\{0\}$. 
		Let $v'\in V^1(gh)\setminus\{0\}$.
		Assume that $u = gv$ and that $u' = gh v'$. 
		Then by Lemma \ref{lem:product-is-lipschitz}, we have $\frac{\|fu'\|}{\|f\|\|u'\|} \ge \frac{\|fu\|}{\|f\|\|u\|} - \dist([u],[u']) \ge \frac{\eps}{2}$. Therefore, we have $\|fghv'\|\ge \|f\|\|gh\|\|v'\|\frac{\eps}{2}$ so $\|fgh\|\ge \frac{\eps}{2}\|f\|\|gh\|$. Moreover, we have $g \AA^\eps h$, therefore $\|fgh\| \ge \frac{\eps^2}{2}\|f\|\|g\|\|h\|$. Now using the formula $\sqz(\gamma) = \log\left(\frac{\|\gamma\|^2}{\|\gamma\wedge\gamma\|}\right)$, we get:
		\begin{equation*}
			\sqz(fgh) \ge \sqz(f)+\sqz(g)+\sqz(h) - 4 |\log(\eps)| - 2\log(2).\qedhere
		\end{equation*}
	\end{proof}
	
	\begin{Lem}[Heredity of the alignment]\label{lem:herali}
		Let $f,g,h$ be non-zero matrices such that the product $fgh$ is well defined and let $0 < \eps \le 1$. Assume that $\sqz(g) \ge 2 |\log(\eps)| + 3\log(2)$ and that $f \AA^{\eps} g \AA^{\frac{\eps}{2}} h$. Then $f \AA^{\frac{\eps}{2}} gh$.
	\end{Lem}
	
	\begin{proof}
		Let $u\in U^1(gh)\setminus\{0\}$, let $u' \in U^{\eps}(g)\setminus\{0\}$ and let $u''\in U^1(g)\setminus\{0\}$. By Lemma \ref{lem:cont-prop}, we have $\dist([u'],[u''])\le \frac{\eps}{8}$ and by Lemma \ref{lem:c-prod}, we have $\dist([u], [u''])\le \frac{\eps}{4}$. 
		Then by triangular inequality, we have $\dist([u],[u'])\le \frac{\eps}{2}$.
		
		Now let $v \in V^1(fg)\setminus\{0\}$. Then we have $\|fgv\|\ge \|f\|\|g\|\|v\|\eps$ so  $gv \in U^{\eps}(g)\setminus\{0\}$ and by the above argument, we have $\dist([u],[gv])\le \frac{\eps}{2}$. Then by Lemma \ref{lem:product-is-lipschitz}, we have $\frac{\|fu\|}{\|f\|\|u\|}\ge \frac{\|fgv\|}{\|f\|\|gv\|}- \frac{\eps}{2}$. Moreover $\|f\|\|gv\| \le \|f\|\|g\|\|v\| \le \|fgv\| / \eps$, therefore $\frac{\|fu\|}{\|f\|\|u\|}\ge \frac{\eps}{2}$ and $u \in U^1(gh)$ hence $f \AA^{\frac{\eps}{2}} gh$.
	\end{proof}
	
	\begin{Rem}[The ultra-metric case is easier]
		Let $0 < \eps \le 1$, let $\KK$ be a ultra-metric locally compact field and let $f,g,h$ be matrices with entries in $\KK$ such that th product $f,g,h$ is well defined. If we assume that $f \AA^\eps g \AA^\eps h$, and that $\sqz(g) > 2|\log(\eps)|$, then $f \AA^\eps gh$. Therefore, in the ultra-metric case, we get rid of all the $+ k \log(2)$ constants.
	\end{Rem}
	
	\begin{Lem}[Contraction property for aligned chains]\label{lem:c-chain}
		Let $E$ be a Euclidean vector space, let $0 < \eps \le 1$ let $n\in\NN$.
		Let $g_0, \dots, g_n$ be non-zero matrices such that the product $g_0\cdots g_n$ is well defined. Assume that for all $k \in\{0, \dots, n-1\}$, we have $\sqz(g_k) \ge 2|\log(\eps)| + 3\log(2)$ and $g_{k} \AA^\eps g_{k+1}$. Then one has:
		\begin{gather}
			\| g_0\cdots g_n \| \ge \left(\frac{\eps}{2}\right)^n\prod_{j = 0}^n \|g_j\| \label{eqn:alinorm-chaine}\\
			\sqz(g_0\cdots g_n) \ge \sum_{j = 0}^n \sqz(g_j) - 2n(|\log(\eps)| + \log(2)).\label{eqn:lenali-chaine}
		\end{gather}
		Moreover, for every non-zero vectors $u\in U^1(g_0)\setminus\{0\}$, and $u'\in U^1(g_0\cdots g_n)\setminus\{0\}$, we have:
		\begin{equation}\label{eqn:limbd-chaine}
			\dist([u],[u'])\le\frac{2}{\eps} \exp(-\sqz(g_0)).
		\end{equation}
	\end{Lem}
	
	\begin{proof}
		The lemma is trivial when $n = 0$. Assume $n \ge 1$.
		We claim that for all $0 \le k < n$, we have $g_k \AA^{\frac{\eps}{2}} g_{k+1}\cdots g_n$. For $k = n-1$, we assumed $g_{n-1}\AA^\eps g_n$ so $g_{n-1} \AA^{\frac{\eps}{2}} g_n$. Let $0 < k < n$ and assume that $g_k \AA^{\frac{\eps}{2}} g_{k+1}\cdots g_n$. Then by Lemma \ref{lem:herali} with $f = g_{k-1}$, $g := g_k$ and $h := g_{k+1}\cdots g_n$, we have $g_{k-1} \AA^{\frac{\eps}{2}} g_{k} \cdots g_n$. Hence, we have $g_0 \AA^{\frac{\eps}{2}} g_{1}\cdots g_n$ so by \eqref{eqn:limbd} in Lemma \ref{lem:c-prod}, we have \eqref{eqn:limbd-chaine}. 
		 
		For all $0< k <n$, we have $\|g_k\cdots g_n\|\ge \frac{\eps}{2}\|g_k\|\|g_{k+1}\cdots g_n\|$ by definition of $\AA^{\frac{\eps}{2}}$. 
		Then by induction on $k$, we have $\|g_k\cdots g_n\|\ge \left(\frac{\eps}{2}\right)^{n-k}\|g_k\|\|g_{k+1}\|\cdots \|g_n\|$, for $k =0$, we have \eqref{eqn:alinorm-chaine}.
		
		Now by \eqref{eqn:lenali}, we have $\sqz(g_k \cdots g_n)\ge \sqz(g_k) + \sqz(g_{k+1} \cdots g_n) - 2(|\log(\eps)| + \log(2))$ for all $0 \le k < n$.
		Then by induction, we have $\sqz(g_k \cdots g_n)\ge \sum_{j = k}^n \sqz(g_j)-2(n-k)(|\log(\eps)| + \log(2))$ for all $0\le k <n$, therefore we have \eqref{eqn:lenali-chaine}.
	\end{proof}

	\begin{Lem}[Alignment of partial products]\label{lem:alipart}
		Let $g_0,\dots,g_n$ be non-zero matrices such that the product $g_0\cdots g_n$ is well defined. 
		Let $0 < \eps \le 1$. 
		Assume that for every $k \in \{1,\dots,n-1\}$ we have $\sqz(g_i) \ge 2|\log(\eps)| + 4\log(2)$.
		Assume also that $g_0\AA^\eps g_1 \AA^\eps \cdots \AA^\eps g_n$ \ie for all $k \in \{0, \dots, n-1\}$, we have $g_k \AA^\eps g_{k+1}$. 
		Then for all $k \in \{1, \dots, n\}$, we have $(g_0 \cdots g_{k-1}) \AA^\frac{\eps}{2} (g_{k} \cdots g_{n})$.
	\end{Lem}
	
	\begin{proof}
		Let $k \in \{2, \dots, n-1\}$. Let $u \in U^1(g_{k} \cdots g_{n})\setminus\{0\}$, let $u' \in U^\eps(g_k)\setminus\{0\}$, let $w \in W^1(g_0\cdots g_{k-1})\setminus\{0\}$ and let $w'\in W^\eps(g_{k-1})\setminus\{0\}$. By Lemma \ref{lem:c-chain} applied to the sequence $g_k, \dots, g_n$, and by Lemma \ref{lem:cont-prop} applied to $g_k$ and by triangular inequality, we have $\dist([u],[u']) \le \frac{\eps}{8} + \frac{\eps}{16} \le \frac{\eps}{4}$. 
		By Lemma \ref{lem:c-chain} applied to the sequence $g_{k-1}^*, \dots, g_0^*$ and by the above argument, we have $\dist([w],[w']) \le \frac{\eps}{4}$. 
		
		Now since $g_{k-1} \AA^\eps g_{k}$ and by Lemma \ref{lem:sharp-ali}, there exist $w'\in W^\eps(g_{k-1})\setminus\{0\}$ and $u' \in U^\eps(g_k)\setminus\{0\}$ such that $\frac{|w'u'|}{\|w'\|\|u'\|}\ge \eps$. 
		Assume that $\frac{|w'u'|}{\|w'\|\|u'\|}\ge \eps$.
		Then by Lemma \ref{lem:product-is-lipschitz}, we have $\frac{|w'u|}{\|w'\|\|u\|}\ge \frac{3\eps}{4}$ and by duality, we have $\frac{|wu|}{\|w\|\|u\|}\ge \frac{\eps}{2}$, hence $(g_0 \cdots g_{k-1}) \AA^\frac{\eps}{2} (g_{k} \cdots g_{n})$. 
	\end{proof}

	\begin{Cor}\label{cor:limit-line}
		Let $0 < \eps \le 1$, let $E$ be a Euclidean space and let $(\gamma_n)_{n\in\NN}$ be a sequence in $\mathrm{End}(E)$.
		Assume that for all $n \in \NN$, one has $\gamma_n \AA^\eps \gamma_{n+1}$ and $\sqz(\gamma_{n+1}) \ge 2|\log(\eps)|+ 3\log(2)$.
		Then there is a limit line $l^\infty \in \mathbf{P}(E)$ such that:
		\begin{equation}\label{eqn:limit-line}
			\forall n\in\NN,\; \forall u_n\in U^1(\gamma_0\cdots\gamma_{n-1})\setminus\{0\},\; \dist([u_n],l^\infty) \le \frac{2}{\eps} \exp(-\sqz(\gamma_0\cdots\gamma_{n-1})).
		\end{equation}
	\end{Cor}
	
	\begin{proof}
		Let $m \le n$ be integers and let $u_n\in U^1(\gamma_0\cdots\gamma_{n-1}) \setminus\{0\}$ and $u_m\in U^1(\gamma_0\cdots\gamma_{m-1}) \setminus\{0\}$. By Lemma \ref{lem:alipart}, we have $(\gamma_0\cdots\gamma_{n-1}) \AA^\frac{\eps}{2} (\gamma_n \cdots \gamma_{m-1})$, then by Lemma \ref{lem:c-prod}, we have:
		\begin{equation}\label{eqn:limite-mn}
			\dist([u_n],[u_m]) \le \frac{2}{\eps}\exp(-\sqz(\gamma_0\cdots\gamma_{n-1})).
		\end{equation}  
		By Lemma \ref{lem:c-chain}, we have $\sqz(\gamma_1\cdots\gamma_{n-1}) \ge (n-1) \log(2)+ 2|\log(\eps)|+ 3\log(2)$ and by Lemma \ref{lem:c-prod}, we have $\sqz(\gamma_1\cdots\gamma_{n-1}) \ge n \log(2)$.
		So for any sequence $(u_n)_{n\in\NN} \in \prod_{n = 0}^{+ \infty} (U^1(\gamma_0\cdots\gamma_{n-1}) \setminus\{0\})$, the sequence $([u_n])$ is a Cauchy sequence in $\mathbf{P}(E)$, therefore it has a limit. 
		Moreover, the diameter of $\mathbf{P}U^1(\gamma_0\cdots\gamma_{n-1})$ goes to $0$ by the above argument so the limit $l^\infty$ does not depend on the choice of the $u_n$'s.
		
		Now we take the limit of \eqref{eqn:limite-mn} for $m\to + \infty$ and we get \eqref{eqn:limit-line}.
	\end{proof}

	\begin{Lem}\label{lem:zouli-alignemont}
		Let $0< \eps \le 1$ and let $n\in \NN$. 
		Let $h$ and $g_0, \dots, g_n$ be matrices such that the product $h g_0 \cdots g_n$ is well defined. 
		Assume that for all $i \in\{0, \dots, n\}$, we have $\sqz(g_i) \ge 2|\log(\eps)| + 4 \log(2)$. 
		Assume also that we have $h \AA^{\eps} g_0$ and that $g_i\AA^\frac{\eps}{2} g_{i+1}$ for all $i \in\{0, \dots, n-1\}$. Then we have $h \AA^\frac{\eps}{2} (g_0\cdots g_n)$.
	\end{Lem}
	
	\begin{proof}
		By Lemma \ref{lem:c-chain} \eqref{eqn:lenali-chaine}, we have $\sqz(g_0\cdots g_n) \ge  \sqz(g_0) + \sum_{j =1}^n (\sqz (g_j) - 2|\log(\eps)| - 2\log(2)) \ge \sqz(g_0)$. 
		Let $u \in U^1(g_0)\setminus\{0\}$ and let $u' \in U^1(g_0\cdots g_n)\setminus\{0\}$. By \eqref{eqn:limbd-chaine} in Lemma \ref{lem:c-chain}, we have $\dist([u],[u']) \le \frac{4\eps}{16}$. 
		Let $v \in V^\eps(h) \cap U^{\eps} (g_0)\setminus\{0\}$, which is not empty by Lemma \ref{lem:sharp-ali}. 
		Then by Lemma \ref{lem:cont-prop}, we have $\dist([v], [u]) \le \frac{\eps}{16}$. Hence $\dist([u'],[v]) \le \frac{5\eps}{16}$ so by Lemma \ref{lem:product-is-lipschitz}, we have $\frac{\|hu'\|}{\|h\|\|u'\|} \ge \frac{11\eps}{16} \ge \frac{\eps}{2}$. 
		Hence, we have $h \AA^\frac{\eps}{2} (g_0\cdots g_n)$.
	\end{proof}

	Now we prove a tricky lemma that is essential for the pivoting technique.
	
	\begin{Lem}\label{lem:Atilde}
		Let $0< \eps \le 1$ and let $n\in \NN_1$.
		Let $\gamma_{-1},\gamma_0, \gamma_1, \dots, \gamma_{2n}$ be non-zero matrices and assume that the product $\gamma_{-1}\cdots\gamma_{2n}$ is well defined. Assume that for all $i \in \{0,1,3,5, \dots, 2n-1\}$, we have $\sqz(\gamma_i) \ge 4|\log(\eps)|+ 7\log(2)$ and that for all $0 \le i < n$, we have:
		\begin{equation}\label{eq:hyp-lem-atilde}
			(\gamma_0\cdots\gamma_{2i})\AA^\eps \gamma_{2i + 1} \AA^\eps\gamma_{2i + 2}
		\end{equation}
		and that $\gamma_{-1}\AA^\eps\gamma_0$. Then $\gamma_{-1}\AA^\frac{\eps}{2}(\gamma_0\cdots\gamma_{2n})$.
	\end{Lem}
	
	\begin{proof}
		Let $i \in\{0, \dots, n\}$. By Lemma \ref{lem:herali} applied to $f = \gamma_0\cdots \gamma_{2i}$, $g = \gamma_{2i+1}$ and $h = \gamma_{2i+2}$, we have $(\gamma_0\cdots \gamma_{2i}) \AA^\frac{\eps}{2} (\gamma_{2i +1}\gamma_{2i +2})$ and by \eqref{eqn:lenali} in Lemma \ref{lem:c-prod}, we have $\sqz(\gamma_{2i +1}\gamma_{2i +2}) \ge 2|\log(\eps)|+ 7\log(2)$. 
		We moreover claim that for all $i \in\{1, \dots, n-1\}$, we have $(\gamma_{2i-1}\gamma_{2i}) \AA^\frac{\eps}{4} (\gamma_{2i +1}\gamma_{2i +2})$.
		Let $i \in \{2, \dots, n-1\}$, let $w \in W^{1}(\gamma_{2i-1}\gamma_{2i}) \setminus\{0\}$ and let $w' \in W^{\frac{\eps}{2}}(\gamma_{0}\cdots\gamma_{2i}) \setminus\{0\}$ and let $w'' \in W^{1}(\gamma_{0}\cdots\gamma_{2i}) \setminus\{0\}$. 
		We have $\sqz(\gamma_{2i +1}\gamma_{2i +2}) \ge 2 |\log(\eps)| + 7 \log(2)$ so by Lemma \ref{lem:cont-prop}, we have $\dist([w'],[w'']) \le \frac{\eps}{64}$.
		Moreover, $\gamma_{0}\cdots\gamma_{2i-2}\AA^{\eps/2}(\gamma_{2i-1}\gamma_{2i})$ so by \ref{lem:c-prod}, we have $\dist([w],[w'']) \le \frac{\eps}{64}$.
		Then by triangular inequality, we have $\dist([w],[w']) \le \frac{\eps}{4}$ so by Lemma \ref{lem:product-is-lipschitz}, we have:
		\begin{equation*}
			\frac{\|w \gamma_{2i +1}\gamma_{2i +2}\|}{\|w\|\|\gamma_{2i +1}\gamma_{2i +2}\|} \ge \frac{\|w' \gamma_{2i +1}\gamma_{2i +2}\|}{\|w'\|\|\gamma_{2i +1}\gamma_{2i +2}\|} -\frac{\eps}{4}
		\end{equation*}
		Moreover, we have $(\gamma_0\cdots \gamma_{2i}) \AA^\frac{\eps}{2} (\gamma_{2i +1}\gamma_{2i +2})$ so there exists a linear form $w' \in W^{\frac{\eps}{2}}(\gamma_{0}\cdots\gamma_{2i})\setminus\{0\}$ such that $\frac{\|w' \gamma_{2i +1}\gamma_{2i +2}\|}{\|w'\|\|\gamma_{2i +1}\gamma_{2i +2}\|} \ge\frac{\eps}{2}$. 
		Hence we have $\frac{\|w \gamma_{2i +1}\gamma_{2i +2}\|}{\|w\|\|\gamma_{2i +1}\gamma_{2i +2}\|} \ge \frac{\eps}{4}$, which proves the claim.
		
		Now we have $\gamma_0 \AA^\frac{\eps}{4} (\gamma_1\gamma_2) \AA^\frac{\eps}{4} \cdots \AA^\frac{\eps}{4} (\gamma_{2i-n}\gamma_{2n})$. Let $u \in U^{1}(\gamma_0\cdots\gamma_{2n})\setminus\{0\}$ and let $u' \in U^{\eps}(\gamma_0)\setminus\{0\}$ and let $u'' \in U^{1}(\gamma_0)\setminus\{0\}$. By Lemma \ref{lem:c-chain} applied to $g_0 = \gamma_0$ and $g_i = \gamma_{2i-1}\gamma_{2i}$ for all $i \in\{1, \dots, n\}$ and $\eps' = \frac{\eps}{4}$, we have $\dist([u],[u'']) \le \frac{\eps}{16}$. Moreover, by Lemma \ref{lem:cont-prop}, we have $\dist([u'],[u'']) \le \frac{\eps^3}{128}$. Then by triangular inequality, we have $\dist([u],[u']) \le \frac{\eps}{2}$. Now we may assume that $\frac{\|\gamma_{-1} u'\|}{\|\gamma_{-1}\|\| u'\|} \ge \eps$ because $\gamma_{-1}\AA^\eps\gamma_0$. Then by Lemma \ref{lem:product-is-lipschitz}, we have $\frac{\|\gamma_{-1} u\|}{\|\gamma_{-1}\|\| u\|} \ge \eps$ so $\gamma_{-1}\AA^\frac{\eps}{2}(\gamma_0\cdots\gamma_{2n})$. 
	\end{proof}

	\subsection{Link between singular values and eigenvalues}
	
	In this short section, we prove the following lemma. 
	We will use the following notations. Let $g$ be an endomorphism such that $\prox(g) > 0$. 
	We write $E^+(g)$ for the eigenspace associated to the maximal eigenvalue of $g$, it is a line because $\prox(g) > 0$.
	A basic fact is that $g(E^+(g)) = E^+(g)$ and there is a $g$-stable supplementary $E^-(g)$ \ie such that $g(E^-(g)) \subset E^-(g)$ and $E = E^+(g) \oplus E^-(g)$. 
	
	\begin{Lem}\label{lem:eigen-align}
		Let $E$ be a Euclidean space and let $0 < \eps \le 1$. Let $g$ be an endomorphism such that $\sqz(g) \ge 2|\log(\eps)|+ 4\log(2)$ and $g \AA^\eps g$. Then $g$ is proximal and we have the following:
		\begin{gather}
			\rho_1(g) \ge \frac{\eps}{2} \|h\|\label{eq:ev-sv} \\
			\prox(g) \ge \sqz(g) - 2|\log(\eps)| - 2 \log(2)\label{eq:sqz-prox} \\
			\forall u \in U^1(g), \dist([u], E^+(g)) \le \frac{2}{\eps} \exp(-\sqz(g))\label{eq:dom-eigen}
		\end{gather}
	\end{Lem}
	
	\begin{proof}
		Consider $(g_k)_{k\ge 0}$ to be the sequence of copies of $g$.
		First, we apply Lemma \ref{lem:alipart} and we get that $\sqz(g^n) \ge n (\sqz(g) - 2|\log(\eps)| - 2 \log(2))$, then going to the limit $n \to +\infty$, we get that $\liminf\frac{\sqz(g^n)}{n}\ge \sqz(g) - 2|\log(\eps)| - 2 \log(2)$. Moreover, we know that this inferior limit is in fact an honest limit and it is $\prox(g)$, which proves \eqref{eq:sqz-prox}. The proof of \eqref{eq:ev-sv} goes the same but using \eqref{eqn:alinorm-chaine}. Note that \eqref{eq:sqz-prox} implies that $\prox(g) > 0$ so $E^+(g)$ is a line.
		
		To get \eqref{eq:dom-eigen}, we apply Corollary \ref{cor:limit-line}. We get a line $l^\infty$ such that for any $u_n \in U^1(g^n)\setminus\{0\}$, we have $[u_n] \to l^\infty$. 
		Moreover, we have $\dist([u], l^\infty) \le \frac{2 \exp(-\sqz(g))}{\eps}$. Now we only need to show that $l^\infty = E^+(g)$. 
		Let $e \in E^+(g)$. 
		Then we have $\|ge\| = \rho_1(g) \|e\|$ so by \eqref{eq:ev-sv}, $e \in V^\frac{\eps}{2}(g)$. 
		Moreover, $e$ is an eigenvector associated to a simple eigenvalue so $e \in\KK ge$ and as a consequence $e \in U^\frac{\eps}{2}(g)$. Now this reasoning holds for all positive power of $g$ so we have $e \in U^\frac{\eps}{2}(g^n)$. Moreover, $\sqz(g^n) \to +\infty$ by\eqref{eq:sqz-prox} so by Lemma \ref{lem:cont-prop}, the projective diameter of $U^\frac{\eps}{2}(g^n)$ goes to zero so $[u_n] \to [e]$, so $l^\infty = [e]$.
	\end{proof}

	\subsection{Finite description of the alignment}
	
	In this section, on construct a finitely described alignment relation that allows us to use the tools described in the toy model of paragraph~\ref{toy-model}, even though we are not working on locally finite groups. 
	
	\begin{Def}\label{def:discrete}
		Let $\AA$ be a measurable binary relation on a measurable set $\Gamma$ \ie an $\mathcal{A}_\Gamma\otimes\mathcal{A}_\Gamma$ measurable subset of $\Gamma\times\Gamma$. We say that $\AA$ is finitely described if there exist an integer $M \in\NN_{\ge 1}$ and two families of measurable subsets $(L_i)_{1\le i \le M}\in\mathcal{A}_\Gamma^M$ and $(R_j)_{1\le j \le M}\in\mathcal{A}_\Gamma^M$ such that:
		\begin{equation*}
			\Gamma=\bigsqcup_{i = 1}^{M} L_i=\bigsqcup_{j=1}^{M} R_j
		\end{equation*}
		and a subset $A\subset\{1,\dots,M\}^2$ such that:
		\begin{equation*}
			\AA = \bigsqcup_{(i,j)\in A} L_i\times R_j.
		\end{equation*}
	\end{Def}

	\begin{Lem}[Discrete descriptions of alignment relations]\label{lem:discrete}
		Let $E$ be an Euclidean vector space and $0 < \eps_1 < \eps_2 \le 1$. There exists a discrete binary relation $\AA$ on $\mathrm{End}(E)$ that satisfies the inclusions $\AA^{\eps_2}\subset\AA\subset\AA^{\eps_1}$ \ie for any given $g,h \in \mathrm{End}(E)$, we have $g\AA^{\eps_2}h\Rightarrow g\AA h\Rightarrow g\AA^{\eps_1}h$.
	\end{Lem}
	
	\begin{proof}
		Let $\eps := \frac{\eps_2-\eps_1}{4}$. Let $N := \left\lfloor\frac{1}{\eps}\right\rfloor$. Let $k \in\NN$ and let $(u_1, \dots u_k)\subset E\setminus\{0\}$ be such that:
		\begin{equation*}
			\forall v \in E \setminus\{0\},\; \exists i \in \{1, \dots, k\},\;\dist([v],[u_i])\le \eps.
 		\end{equation*}
 		Such a family exists because $\mathbf{P}(E)$ is compact. Now let $(w_1, \dots w_k)\subset E^*\setminus\{0\}$ be such that:
 		\begin{equation*}
 			\forall f \in E^* \setminus\{0\},\; \exists i \in \{1, \dots, k\},\;\dist([f],[w_i])\le \eps.
 		\end{equation*}
 		Such a family exists because $E^*$ is isometric to $E$.
 		Now let $h \in \mathrm{End}(E) \setminus\{0\}$ and let $n \in \{1,\dots, N\}$, we define:
 		\begin{equation}
 			\phi_n(h) := \{i \in \{1, \dots,k\}\,|\, w_i \in W^{n\eps}(h)\} \quad \text{and}\quad \psi_n(h) := \{i \in \{1, \dots,k\}\,|\,u_i \in U^{n\eps}(h)\}.
 		\end{equation}
 		Now let:
 		\begin{multline*}
 			\AA := \left\{(g,h) \in \mathrm{End}(E) \setminus\{0\}\,\middle|\,\exists n_1, n_2\in\{1,\dots, N\}, \exists i \in \phi_{n_1}(g), \exists j \in \psi_{n_2}(h), \frac{|w_i u_j|}{\|w_i\|\|u_j\|} n_1n_2\eps^2 \ge \eps_1\right\} \\
 			\sqcup \left((\mathrm{End}(E) \setminus\{0\})\times \{0\}\right) \sqcup \left(\{0\} \times(\mathrm{End}(E) \setminus\{0\})\right) \sqcup \left(\{0\}\times\{0\}\right).
 		\end{multline*}
 		First we claim that $\AA \subset \AA^{\eps_1}$. Let $g \AA h$. If $g = 0$ or $h = 0$, then we have $g \AA^{\eps_1} h$ trivially. 
 		Assume that $g \neq 0$ and $h \neq 0$. 
 		Let $n_1, n_2\in\{1,\dots, N\}$, let $i \in \phi_{n_1}(g)$ and let $j \in \psi_{n_2}(h)$ be such that $\frac{|w_i u_j|}{\|w_i\|\|u_j\|} n_1 n_2\eps^2 \ge \eps_1$. 
 		Let $f \in V^{n_1 \eps}(g^*)$ be such that $f g = w_i$ and let $v \in V^{n_2 \eps}(h)$ be such that $hv = u_j$. Such $f, v$ exist because $w_i \in W^{n_1\eps}(g)$ and $u_j \in U^{n_2 \eps}(h)$, moreover, they are not trivial. 
 		Then, we have $\frac{|f g h v|}{\|fg\|\|hv\|} = \frac{|w_i u_j|}{\|w_i\|\|u_j\|}$ and $\frac{\|fg\|}{\|f\|\|g\|}\ge n_1 \eps$ and $\frac{\|hv\|}{\|h\|\|v\|}\ge n_2 \eps$. Hence, we have $\frac{|f g h v|}{\|f\|\|g\|\|h\|\|v\|} \ge \eps_1$ so $\|gh\|\ge \eps_1\|g\|\|h\|$, which proves the claim.
 		
 		Now we claim that $\AA^{\eps_2} \subset \AA$. Let $g \AA^{\eps_2} h$. Assume that $g \neq 0$ and $h \neq 0$. Let $f \in E^*\setminus\{0\}$ and $v \in E \setminus\{0\}$ be such that $|fghv| = \|f\|\|gh\|\|v\| \ge \eps_2 \|f\|\|g\|\|h\|\|v\|$. let $n_1 := \left\lfloor\frac{\|fg\|}{\|f\|\|g\|\eps}\right\rfloor$ and let $n_2 := \left\lfloor\frac{\|hv\|}{\|h\|\|v\|\eps}\right\rfloor$. Then we have $n_1 \eps \ge \frac{\|fg\|}{\|f\|\|g\|}-\eps$ and $n_2\eps \ge \frac{\|hv\|}{\|h\|\|v\|}-{\eps}$. Let $i,j\in\{1, \dots, k\}$ be such that $\dist([w_i],[fg])\le \eps$ and $\dist([u_j],[hv]) \le \eps$. Then by Lemma \ref{lem:product-is-lipschitz}, we have $\frac{|w_i u_j|}{\|w_i\|\|u_j\|} \ge \frac{|fghv|}{\|fg\|\|hv\|} - 2 \eps$. Hence:
 		\begin{equation*}
 			\frac{|w_i u_j|}{\|w_i\|\|u_j\|} n_1 n_2 \eps^2 \ge \left(\frac{|fghv|}{\|fg\|\|hv\|} - 2 \eps\right) \left(\frac{\|fg\|}{\|f\|\|g\|}-\eps\right) \left(\frac{\|hv\|}{\|h\|\|v\|}-{\eps}\right) \\
 		\end{equation*}
 		Moreover all three factors are in $[0,1]$ so if we develop, we get:
 		\begin{align*}
 			\frac{|w_i u_j|}{\|w_i\|\|u_j\|} n_1 n_2 \eps^2 & \ge \left(\frac{|fghv|}{\|fg\|\|hv\|}\right) \left(\frac{\|fg\|}{\|f\|\|g\|}- \eps \right) \left(\frac{\|hv\|}{\|h\|\|v\|} - \eps\right) - 2\eps \\
 			& \ge \left(\frac{|fghv|}{\|fg\|\|hv\|}\right) \left(\frac{\|fg\|}{\|f\|\|g\|}\right) \left(\frac{\|hv\|}{\|h\|\|v\|}\right) - 4\eps \\
 			 & \ge \frac{|fghv|}{\|f\|\|g\|\|h\|\|v\|}-4\eps \ge \eps_2 - 4 \eps = \eps_1.
 		\end{align*}
 		Therefore, we have $g \AA h$, which proves the claim.
 		
 		Now we claim that $\AA$ is discrete. 
 		This follows directly from the fact that given $g$ and $h$ two matrices, the condition $g \AA h$ is expressed in terms of $(\phi_1(g), \dots, \phi_N(g))$ and $(\psi_1(h), \dots, \psi_N(h))$, which take only finitely many values.
	\end{proof}
	
	Note that the same proof may be used to construct discrete alignment relations on $\mathrm{Hom}(E,F) \times \mathrm{Hom}(H,E)$ for $E$, $F$ and $H$, three given Euclidean spaces.

	\section{Random products and extractions}\label{markov}

	\subsection{Notations for extractions}
	
	In the two following sections, we will denote by $\Gamma$ an abstract measurable semi-group \ie a second countable measurable space endowed with an associative and measurable composition map $\cdot :\Gamma \times \Gamma \to \Gamma$. 
	We will assume that $\Gamma$ has an identity element that we denote by $\mathbf{1}_\Gamma$.
	The measurable semi-group $\Gamma$ can be $(\NN, +)$, $\mathrm{End}(E)$ or a semi-group of words.
	The semi-group $\NN$ will always be endowed with the addition map.
	
	Let us recall the notations introduced in Paragraph \ref{sec:intro-pivot}.
	We write $\widetilde\Gamma$ for the semi-group of words with letters in $\Gamma$ \ie the set of all tuples $\bigsqcup_{l\in\NN}\Gamma^l$, (where $\Gamma^l$ is identified with $\Gamma^{\{0, \dots, l-1\}}$ and endowed with the product $\sigma$-algebra for all $l \in\NN$) that we endow with the concatenation product
	\begin{equation*}
		\begin{array}{crcl}
			\odot : & \widetilde\Gamma \times\widetilde\Gamma & \longrightarrow & \widetilde\Gamma \\
			&((\gamma_0, \dots,\gamma_{k-1}),(\gamma'_0,\dots,\gamma'_{l-1})) \in\Gamma^k \times\Gamma^l & \longmapsto & (\gamma_0, \dots,\gamma_{k-1},\gamma'_0,\dots,\gamma'_{l-1})\in\Gamma^{k+l}.
		\end{array}
	\end{equation*}
	We also define the length functor:
	\begin{equation*}
		L :\; \widetilde\Gamma \longrightarrow \NN ; \; (\gamma_0, \dots,\gamma_{k-1}) \longmapsto k,
	\end{equation*}
	and the product functor:
	\begin{equation*}
		\Pi : \; \widetilde\Gamma \longrightarrow \Gamma ;\; (\gamma_0, \dots,\gamma_{k-1}) \longmapsto \gamma_0\cdots\gamma_{k-1}.
	\end{equation*}

	Given $(\tilde\gamma_n)_{n\in\NN}\in\widetilde\Gamma^{\underline{\NN}}$, we write $\bigodot_{n = 0}^{+\infty} \tilde\gamma_n \in\Gamma^\NN$ for the left to right concatenation of all the $\tilde\gamma_k$'s and we write $\bigodot^\infty : \widetilde\Gamma^\NN \to \Gamma^\NN ; (\tilde\gamma_n) \mapsto \bigodot_{n = 0}^{+\infty} \tilde\gamma_n$. 
	In other words, for all $n \in\NN$ and all $0 \le k < L(\tilde\gamma_n)$, and for $m := L(\tilde\gamma_0\odot \cdots\odot\tilde\gamma_{n-1}) + k$, the $m$-th element, \ie the projection on the $m$-indexed coordinate, of the sequence $\bigodot_{n = 0}^{+\infty} \tilde\gamma_n$ is the $k$-th letter of $\tilde\gamma_n$, \ie the projection on the $k$-indexed coordinate. Note that all the above defined maps $\odot$, $L$, $\Pi$ and $\bigodot^\infty$ are measurable.
	
	\begin{Def}[Grouping of factors]\label{def:grouping}
		Let $\Gamma$ be a semi-group. Let $\gamma := (\gamma_n)_{n\in\NN}\in\Gamma^\NN$ and let $w := (w_n)_{n\in\NN}\in\NN^{\NN}$ be non-random sequences. For all $n \in\NN$, define $\overline{w}_n := w_0 + \dots + w_{n-1}$. We denote by $\widetilde\gamma^w \in\widetilde{\Gamma}^\NN$ the sequence of $w$-groups of $\gamma$ which we define by:
		\begin{equation*}
			\forall n\in\NN,\;\widetilde\gamma^w_n := (\gamma_{\overline{w}_n +k})_{0 \le k < w_n} = (\gamma_{\overline{w}_n},\dots, \gamma_{\overline{w}_{n+1}-1}).
		\end{equation*}
		We denote by $\gamma^w$ the sequence of $w$-products of $\gamma$ defined as $\gamma^w := \Pi \circ \widetilde\gamma^w \in\Gamma^\NN$ \ie $\gamma^w_n = \gamma_{\overline{w}_n} \cdots \gamma_{\overline{w}_{n+1}-1}$ for all $n\in\NN$. 
		
		We denote by $\overline{\gamma} \in\Gamma^\NN$ the left to right product associated to $\gamma$, defined as $\overline{\gamma}_n := \gamma_0 \cdots \gamma_{n-1}$ for all $n\in\NN$ and we denote by $\overline{\gamma}^w\in\Gamma^\NN$ the left to-right product associated to $\gamma^w$ defined as $\overline{\gamma}^w_n := \gamma^w_0 \cdots \gamma^w_{n-1} = \overline\gamma_{\overline{w}_n}$ for all $n \in \NN$.
	\end{Def}

	Let $(\tilde{g}_n) \in \widetilde\Gamma^{\NN}$ be a sequence which is not stationary to the trivial word, note that the map $\left(\bigodot^\infty, L^{\otimes\NN}\right)$, that sends $(\tilde{g}_n)$ to the pair $\left(\bigodot_{n = 0}^{+\infty}\tilde{g}_n,(L(\tilde{g}_n))_{n\in\NN}\right)\in \Gamma^{\NN}\times \NN^{\NN}$, is one-to-one. 
	Indeed, to get back to the sequence $\tilde{g}$, write $\gamma := \bigodot_{n = 0}^{+\infty}\tilde{g}_n$ and $w_n := (L(\tilde{g}_n))_{n\in\NN}$, then $\tilde{g} = \widetilde{\gamma}^w$.
	Given $\tilde\mu$ a probability distribution on $\widetilde\Gamma^\NN$, we will write $(\widetilde\gamma^w_n) \sim \tilde\mu$ to introduce a random sequence $(\gamma_n)\in\Gamma^\NN$ and a random sequence $(w_n)\in\NN^\NN$, defined on the same probability space and such that $(\widetilde\gamma^w_n) \sim \tilde\mu$.

	Given $\tilde\eta$ and $\tilde\kappa$ two probability measures on $\widetilde\Gamma$, we write $\tilde\eta \odot \tilde\kappa := \odot_*(\tilde\eta \otimes \tilde\kappa)$ for the convolution of $\tilde\eta$ and $\tilde\kappa$ \ie the push-forward by $\odot$ of the product measure $\tilde\eta \otimes \tilde\kappa$. 
	Then $\tilde\eta \odot \tilde\kappa$ is the distribution of the concatenation of two independent random words of respective distribution $\tilde\eta$ and $\tilde\kappa$.
	Given $(\tilde\eta_n)_{n\in\NN}$ a sequence of probability measures, we write $\bigodot_{n=0}^{+\infty} \tilde\eta_n$ for the push forward of $\bigodot^\infty_*\left(\bigotimes_{n=0}^{+\infty} \tilde\eta_n\right)$. Given $\tilde\eta$ a probability measure on $\widetilde\Gamma$, we write $\tilde\eta^{\odot\NN}$ for $\bigodot_{n=0}^{+\infty} \tilde\eta$.

	\begin{Def}[Extraction]\label{def:extract}
		Let $\Gamma$ be a semi-group, let $\mu$ be a probability measure on $\Gamma^\NN$ and let $\tilde\mu$ be a probability measure on $\widetilde\Gamma^\NN$. 
		We say that $\tilde\mu$ is an extraction of $\mu$ if $\mu = \bigodot^\infty_* \tilde\mu$ and if there exist constants $C, \beta >0$ such that for $(\tilde{g}_n)_{n\in\NN} \sim \tilde\mu$ and for all $n \in\NN$, we have almost surely:
		\begin{equation*}
			\EE\left(\exp(\beta L(\tilde g_n))\,\middle|\, (\tilde{g}_k)_{k < n}\right) \le C.
		\end{equation*}
	\end{Def}
	
	\subsection{Rank and essential kernel of a probability distribution}\label{sec:kernel}
	In this section, we describe the probabilistic behaviour of the kernel of a product of i.i.d. random matrices.
	For that we do not use the tools from Section \ref{sec:ali-sqz}, in particular, we do not care about the Euclidean structure of the space.
	Given $h$ a linear map, we denote by $\rank(h)$ the rank of $h$ \ie the dimension of the image of $h$. 
	We say that a probability measure $\nu$ is supported on a set $S$ if $\nu(S) =1$, it is weaker than saying that $S$ is the support of $\nu$ because we do not assume that $S$ is closed nor minimal. 
	Note that given $E$ a vector space a measure $\nu$ may be supported on $\mathrm{GL}(E)$

	\begin{Def}[Rank of a distribution]\label{def:rank}
		Let $E$ be a vector space and let $\nu$ be a step distribution on $\mathrm{End}(E)$. We define the eventual rank of $\nu$ as the largest integer $\underline{\rank}(\nu)$ such that:
		\begin{equation}
			\forall n\ge 0,\,\nu^{*n}\left\{\gamma \in \mathrm{End}(E) \;\middle|\; \rank(\gamma) < \underline{\rank}(\nu)\right\}=0.
		\end{equation}
	\end{Def}

	\begin{Lem}[Eventual rank of a distribution]\label{lem:rank}
		Let $E$ be a Euclidean space and let $\Gamma := \mathrm{End}(E)$. Let $\nu$ be a probability measure on $\Gamma$. There is a probability measure $\tilde\kappa$ on $\widetilde\Gamma$ such that $\tilde\kappa^{\otimes\NN}$ is an extraction of $\nu^{\otimes\NN}$ and $\Pi_*\tilde\kappa$ is supported on the set of endomorphisms of rank $\underline{\rank}(\nu)$.
	\end{Lem}
	
	\begin{proof} 
		Given a non-random sequence $\gamma = (\gamma_n)_{n\in\NN}$, the sequence $(\rank(\overline\gamma_n))_{n\in\NN}$ is a non-increasing sequence of non-negative integers so it is stationary.
		Write $r_\gamma$ for the limit of $(\rank(\overline\gamma_n))_{n\in\NN}$.
		Then there is an integer $n' \ge 1$ such that $\rank(\overline\gamma_n) = r_\gamma$ for all $n \ge n'$.
		Write $n_\gamma$ for the minimal such $n'$. 
		Note that $\gamma\mapsto r_\gamma$ and $\gamma \mapsto n_\gamma$ are measurable maps.
		
		Now let $(\gamma_n)_{n\in\NN}\sim\nu^{\otimes\NN}$ be a random sequence. 
		We define $\tilde\kappa$ to be the distribution of $(\gamma_0, \dots, \gamma_{n_\gamma - 1})$. 
		Note that $n_\gamma$ is a stopping time for $(\overline\gamma_n)_{n\in\NN}$ so the conditional distribution of $(\gamma_{n + n_{\gamma}})_{n\in\NN}$ relatively to $(\gamma_0, \dots, \gamma_{n_\gamma - 1})$ is $\nu^\otimes\NN$.
		Hence, we have $\tilde\kappa \odot \nu^{\otimes\NN} = \nu^\otimes\NN$, so $\tilde\kappa^{\odot k} \odot \nu^{\otimes\NN} = \nu^\otimes\NN$ for all $k \in\NN$ and by construction $\tilde\kappa$ is non-trivial so $\tilde\kappa^{\odot \NN} = \nu^\otimes\NN$.
		Therefore, the measure $\tilde\kappa^{\otimes\NN}$ is an extraction of $\nu^{\otimes\NN}$.
		
		Moreover $\Pi_*\tilde\kappa$ is the distribution of $\overline\gamma_{n_\gamma}$ which has rank $r_\gamma$ and $L_*\tilde\kappa$ is the distribution of $n_\gamma$.	
		Therefore, we only need to show that $r_\gamma$ is almost surely constant and that $n_\gamma$ has finite exponential moment. 
		
		Let $r_0$ be the essential lower bound of $r_\gamma$ \ie the largest integer such that $\PP(r_\gamma \ge r_0) =1$. 
		Let $n_0 \in\NN$ be such that $\PP(\rank(\overline\gamma_{n_0}) = r_0) > 0$ and write $\alpha := \PP(\rank(\overline\gamma_{n_0}) = r_0)$. 
		We claim that such an integer $n_0$ exists. 
		Indeed, by minimality, we have $\PP(r_\gamma > r_0) < 1$ so $\PP(r_\gamma = r_0) > 0$, which means that $\PP(\rank(\overline\gamma_{n_\gamma}) = r_0) > 0$.
		Let $n_0$ to be such that $\PP(n_\gamma \le n_0 \cap r_\gamma = r_0) > 0$. 
		Such an $n_0$ exists, otherwise $n_0$ would be almost surely infinite, which is absurd.
		
		Now since the sequence $(\gamma_n)$ is i.i.d, we have $\PP(\rank(\gamma_{kn_0}\cdots \gamma_{(k+1)n_0-1}) = r_0) > 0$ for all $k \in\NN$. 
		Moreover these events are independents so for all $k \in\NN$, we have:
		\begin{align*}
			\PP(\forall k' < k, \rank(\gamma_{k'n_0}\cdots\gamma_{(k'+1)n_0-1}) > r_0 ) & = (1-\alpha)^k.
		\end{align*}
		Now note that the rank of a product is bounded above by the rank on each of its factor so:
		\begin{align*}
			\forall k\in\NN,\;\PP(\rank(\overline\gamma_{kn_0}) > r_0) & \le (1-\alpha)^k \\
			\forall n\in\NN, \; \PP\left(\rank(\overline\gamma_{\lfloor\frac{n}{n_0}\rfloor n_0}) > r_0\right) & \le (1-\alpha)^{\lfloor\frac{n}{n_0}\rfloor} 
		\end{align*}
		Now note that for all $n\in\NN$, we have $\lfloor\frac{n}{n_0}\rfloor \ge \frac{n}{n_0}-1$ and $\lfloor\frac{n}{n_0}\rfloor n_0 \le n$ so:
		\begin{align*}
			\forall n\in\NN, \; \PP\left(\rank(\overline\gamma_{n}) > r_0\right) & \le (1-\alpha)^{\frac{n}{n_0}-1}.
		\end{align*}
		Let $C = \frac{1}{1-\alpha}$ and $\beta = \frac{-\log{(1-\alpha)}}{n_0 }> 0$. 
		Then we have $\PP\left(\rank(\overline\gamma_{n}) > r_0\right) \le C \exp(-\beta n)$ for all $n \in\NN$ and $\beta >0$. 
		Note that for all $n \in\NN$, we have $ \PP\left(\rank(\overline\gamma_{n}) > r_0\right) \ge \PP(r_\gamma > r_0)$ so $\PP(r_\gamma > r_0) =0$, which means that $r_0 = r_\gamma$. 
		Hence $\PP\left(\rank(\overline\gamma_{n}) > r_0\right) = \PP( n < n_\gamma)$ for all $n\in\NN$, so $\PP(n_\gamma > n)\le C \exp(-\beta n)$, which means that $n_\gamma$ has finite exponential moment.
		
	\end{proof}

	\begin{Def}[Essential kernel]\label{def:essker}
		Let $E$ be a Euclidean space of dimension $d\ge 2$ and let $\nu$ be a probability distribution on $\mathrm{End}(E)$. We define the essential kernel of $\nu$ as:
		\begin{equation}\label{essker}
			\underline{\ker}(\nu) := \left\{v \in E \,\middle|\,\exists n\in \NN, \nu^{*n}\{h\in\mathrm{End}(E)\,|\,hv =0\} > 0 \right\}.
		\end{equation}
	\end{Def}
	
	\begin{Lem}\label{lem:descri-ker}
		Let $E$ be a Euclidean space of dimension $d\ge 2$ and let $\nu$ be a probability distribution on $\mathrm{End}(E)$. There is a probability distribution $\kappa$ on $\mathrm{End}(E)$ which is supported on the set of rank $\underline{\rank}(\nu)$ endomorphisms and such that:
		\begin{gather}\label{eq:essker}
			\underline{\ker}(\nu) = \underline{\ker}(\kappa) = \left\{v \in E \,\middle|\, \kappa\{h\in\mathrm{End}(E)\,|\,hv = 0\} > 0 \right\} \\
			\forall v\in E, \; \lim_{n \to +\infty}\nu^{*n}\{h \,|\,hv = 0\} = \sup_{n \in \NN}\nu^{*n}\{h\,|\,hv = 0\} = \kappa\{h\,|\,hv = 0\}.\label{eq:limker}
		\end{gather}
		Moreover, there exists a constant $\alpha < 1$ such that:
		\begin{equation}\label{eq:max-ker}
			\forall v\in E, \,\sup_{n\in\NN}\nu^{*n}\{h\in\mathrm{End}(E)\,|\,hv = 0\} \in [0, \alpha]\cup\{1\}.
		\end{equation}
		Moreover, the set:
		\begin{equation}\label{eq:ker-is-inv}
			\ker(\nu) := \left\{v \in E\,\middle|\, \sup_{n\in\NN}\nu^{*n}\{h\in\mathrm{End}(E)\,|\,hv = 0\}  = 1\right\}
		\end{equation}
		is a subspace of $E$ which is $\nu$-almost surely invariant.
	\end{Lem}
	
	\begin{proof}
		Let $(\gamma_n) \sim \nu^{\otimes\NN}$. We define the random integer:
		\begin{equation*}
			n_0 := \min \left\{n \in\NN\,\middle|\,\rank(\gamma_{n-1} \cdots \gamma_0) = \underline{\rank}(\nu) \right\}.
		\end{equation*} 
		Let $g := \gamma_{n_0-1} \cdots \gamma_0$ and let $\kappa$ be the distribution of $g$.
		Then by Lemma \ref{lem:rank} applied to the transpose of $\nu$, the random integer $n_0$ has finite exponential moment. Now let $v \in E$, and let $n \in\NN$, one has:
		\begin{equation}\label{eq:ker-n}
			\nu^{*n}\{h\in\mathrm{End}(E)\,|\,hv = 0\} = \PP(\gamma_{n-1} \cdots \gamma_0 v = 0) \le \PP(g v = 0)
		\end{equation}
		Indeed if $n \le n_0$ and $\gamma_{n-1} \cdots \gamma_0 = 0$, then $gv = 0$ and if $n \ge n_0$ then $\rank(\gamma_{n-1} \cdots \gamma_0) = \rank(g)$ therefore $\ker (\gamma_{n-1} \cdots \gamma_{n_0}) \cap \mathrm{im}(g) = \{0\}$.
		So $\gamma_{n-1} \cdots \gamma_0 = 0 \Rightarrow gv = 0$. The inequality \eqref{eq:ker-n} is true for all $n$, this implies that:
		\begin{equation}\label{eq:ker-m}
			\underline{\ker}(\nu) \subset \left\{v \in E \,\middle|\, \kappa\{h\in\mathrm{End}(E)\,|\,hv = 0\} > 0 \right\}.
		\end{equation}
		Now let $v$ be such that $\PP(g v = 0) > 0$, then for all $n \in\NN$, we have:
		\begin{equation*}
			\PP(\gamma_{n-1} \cdots \gamma_0 v = 0) \ge \PP(g v = 0) - \PP(n_0 > n).
		\end{equation*}
		Moreover $\PP(n_0 > n) \to 0$, so we have \eqref{eq:limker} by \eqref{eq:ker-n} and \eqref{eq:ker-m}.
		Therefore, there exists an integer $n \in\NN$ such that $\PP(\gamma_{n-1} \cdots \gamma_0 v = 0) > 0$. This proves \eqref{eq:essker} by double inclusion.
		
		Let us prove \eqref{eq:max-ker}.
		Let $V$ be the largest subspace of $E$ such that  $g(V) =\{0\}$ almost surely. Let
		\begin{equation}
			\alpha := \sup_{n\in\NN, v \in E \setminus V}\nu^{*n}\{h\in\mathrm{End}(E)\,|\,hv = 0\}.
		\end{equation}
		Assume by contradiction that $\alpha = 1$.
		Let $(v_n)$ be a non-random sequence in $E \setminus V$ such that $\PP(g v_n = 0)\ge 1 - 2^{-n}$. 
		Then we have $\sum_{n = 0}^{+\infty} \PP\left(g v_n \neq 0\right) < +\infty$. 
		Therefore, by Borel-Cantelli's Lemma, the set $\{n \in\NN\,|\,g v_n \neq 0\}$ is almost surely finite.
		Let $V' := \bigcap_{m \to +\infty} \langle(v_n)_{n\ge m}\rangle$. 
		Then since $E$ is finite dimensional, there is an integer $m \in\NN$ such that $v_n \in V'$ for all $n \ge m$.
		Moreover $g(V') = \{0\}$ almost surely, so $V' \subset V$, which is absurd. This proves \eqref{eq:max-ker} by contradiction.
		
		Let us prove \eqref{eq:ker-is-inv}.
		Assume by contradiction that $\PP(\gamma_0(V) = V) \neq 1$. 
		Let $v \in V$ be such that $\PP(\gamma_0 v \notin V) > 0$. Then for all $n\in\NN$, we have:
		\begin{equation*}
			\PP(\gamma_{n}\cdots \gamma_0 v \neq 0) \ge \PP(\gamma_0 v \notin V) \PP(\gamma_n\cdots \gamma_1 \gamma_0 v \neq 0\,|\,\gamma_0 v \notin V) \ge \PP(\gamma_0 v \notin V) (1-\alpha)> 0,
		\end{equation*}
		which is absurd because $\PP(\gamma_{n} \cdots \gamma_0 v \neq 0) \to 0$.
	\end{proof}

	\begin{Prop}\label{prop:essker}
		Let $\nu$ be a probability distribution on $\mathrm{End}(E)$. The set $\underline{\ker}(\nu)$ is a countable union of subspaces of $E$ that each have dimension at most $\dim(E) - \underline{\rank}(\nu)$. 
	\end{Prop}
	
	\begin{proof}
		Let $d' := \dim(E) - \underline{\rank}(\nu)$. For all $k \in\{0, \dots, \dim(E)\}$, we denote by $\mathrm{Gr}_{k}(E)$ the set of subspaces of $E$ of dimension $k$. Let $\kappa$ be as in Lemma \ref{lem:descri-ker}
		First we show that $\underline{\ker}(\nu)$ is included in a countable union of subspaces of dimension exactly $d'$.
		Given $n \in \NN$ and $\alpha > 0$, we define:
		\begin{equation*}
			K_\alpha := \{x\in E\,|\,\kappa\{h\in\mathrm{End}(E)\,|\,h x =0\} \ge \alpha\}
		\end{equation*}
		Note that we have:
		\begin{equation*}
			\underline{\ker}(\nu) = \bigcup_{m \in\NN} K_{2^{-m}},
		\end{equation*}
		so we only need to show that $K_{2^{-m}}$ is included in a countable union of subspaces for all $m \in \NN$.
		Let $m \in\NN$, we claim that $K_{2^{-m}}$ is included in a union of at most $\binom{d'2^m}{d'}$ subspaces of $E$ of dimension $d'$. Let $g \sim \kappa$, write $\alpha := 2^{-m}$ and assume that $K_\alpha \neq \{0\}$.
		
		Let $N$ be an integer and let $(x_1, \dots, x_N) \in K_\alpha$. 
		Assume that for all $1 \le i_1 < \dots < i_k \le N$ with $k \le d'+1$, the space $\langle x_{i_1}, \dots, x_{i_k} \rangle$ has dimension exactly $k$.
		In this case, we say that the family $(x_i)_{1\le i\le N}$ is in general position up to $d'$. 
		We claim that in this case:
		\begin{equation}
			N \le \frac{d'}{\alpha}. \label{eq:n-le-d'/alpha}
		\end{equation}
		To all index $i \in \{1, \dots, N\}$, we associate a random integer variable $a_i := \mathds{1}_{g x_i = 0} \in \{0,1\}$ \ie such that $a_i = 1$ when $g(x_i) = 0$ and $a_i = 0$ otherwise.
		Note that $\ker(g)$ has dimension at most $d'$ almost surely. 
		As a consequence, for all $\le i_1 < \dots < i_{d'+1} \le N$, we have $\dim\langle x_{i_j} \rangle_{1\le j \le d'+1} > \dim(\ker(g))$ almost surely. Hence $\PP\left(\langle x_{i_j} \rangle \subset \ker(g)\right) = 0$ so $g x_{i_j} \neq 0$ for at least one index $j$.
		This means that, with probability $1$, the random set of indices $\{1\le i \le N\,|\,g x_i =0\}$ does not admit any subset of size $d'+1$ so it has cardinal at most $d'$. 
		In other words, $\sum_{i=1}^N a_i\le d'$ almost surely.
		Now, note that by definition of $\KK_\alpha$, we have $\EE(a_i) = \PP(g x_i = 0) \ge \alpha$ for all $i \in\{1,\dots, N\}$. Hence $ N \alpha \le \sum_{i=1}^N E(a_i) \le d'$, which proves \eqref{eq:n-le-d'/alpha}.
		
		Now we want to construct a family $(x_1, \dots, x_N) \in K_\alpha$ that is in general position up to $d'$ and such that:
		\begin{equation}\label{eq:k-alpha}
			K_\alpha \subset \bigcup_{1\le i_1 < \dots < i_{d'} \le N} \langle x_{i_j} \rangle_{1\le j \le d'}.
		\end{equation}
		We do it by induction. Since we assumed that $K_\alpha \neq \{0\}$, there is a non-zero vector $x_1 \in K_\alpha$. Now let $j \in\NN_{\ge 1}$. Assume that we have constructed a sequence $(x_1, \dots, x_j) \in K_\alpha$ that is in general position up to $d'$. If we have:
		\begin{equation*}
			K_\alpha \subset \bigcup_{1\le i_1 < \dots < i_{d'} \le j} \langle x_{i_j} \rangle_{1\le j \le d'},
		\end{equation*}
		then we write $N := j$ and the algorithm ends as \eqref{eq:k-alpha} is satisfied.
		Otherwise, we take:
		\begin{equation*}
			x_{j+1} \in K_\alpha \setminus\left(\bigcup_{1\le i_1 < \dots < i_{d'} \le j} \langle x_{i_1},\dots, x_{i_{d'}} \rangle\right).
		\end{equation*}
		Then we have constructed a family $(x_1, \dots, x_{j+1}) \in K_\alpha$ that is in general position up to $d'$. 
		This process terminates after at most $\left\lfloor\frac{d'}{\alpha}\right\rfloor$ steps by \eqref{eq:n-le-d'/alpha}.
		Then we conclude by noting that for all $N \in\NN$, the set of multi-indices $\{1\le i_1 < \dots < i_{d'} \le N\}$ has cardinality $\binom{N}{d'}$. 
		
		Now for all $m$ we choose a family $\left(V^{m}_1, \dots, V^{m}_{\binom{d'2^m}{d'}}\right) \in \mathrm{Gr}_{d'}(E)^{\binom{d'2^m}{d'}}$ such that $K_{2^{-m}} \subset \bigcup_{j=1}^{\binom{d'2^m}{d'}} V^{m,n}_j$ and we have:
		\begin{equation*}
			\underline{\ker}(\nu) \subset \bigcup_{\substack{m \in\NN, j\in\NN, \\ 1 \le j \le \binom{d'2^m}{d'}}} V^{m}_j.
		\end{equation*}
		This proves that $\underline{\ker}(\nu)$ is included in a countable union of subspaces of dimension exactly $d'$.
		
		Now we will show that $\underline{\ker}(\nu)$ is in fact equal to a countable union of subspaces. 	
		Let $g \sim \kappa$ and let $K := \underline{\ker}(\nu)$. Let $(V_k)_{k\in\NN} \in \mathrm{Gr}_{d'}(E)^{\NN}$ be such that $K \subset \bigcup V_k$. 
		We will construct a family $\left(V_{k_0, \dots, k_{j}}\right)_{\substack{0 \le j \le d', \\ (k_0, \dots, k_j) \in \NN^{j+1}}}$ such that $(V_{k_0})_{k_0 \in\NN} = (V_k)_{k\in\NN}$, and such that for all $j \in \{1, \dots,d'\}$, we have:
		\begin{equation}
			K \subset \bigcup_{(k_0, \dots, k_j) \in \NN^{j+1}} V_{k_0, \dots, k_{j}},
		\end{equation}
		and for all multi-index $(k_0, \dots, k_{j}) \in \NN^{j+1}$, we have $V_{k_0, \dots, k_{j}} \subset V_{k_0, \dots, k_{j-1}}$, with equality if and only if $V_{k_0, \dots, k_{j-1}} \subset K$. 
		Then we have:
		\begin{equation}
			K = \bigcup_{(k_0, \dots, k_{d'}) \in \NN^{d'+1}} V_{k_0, \dots, k_{d'+1}}.
		\end{equation}
		Indeed, for all $(k_0, \dots, k_{d'}) \in \NN^{d+1}$, we either have $V_{k_0, \dots, k_{d'}} \subset K$ or $V_{k_0} \supsetneq V_{k_0, k_1} \supsetneq \dots \supsetneq V_{k_0, \dots, k_{d'}}$. 
		In the second case, we have $\dim(V_{k_0, \dots, k_{d'}}) \le \dim(V_{k_0}) - d'$ so $V_{k_1, \dots, k_{d'}} = \{0\}$, which is a contradiction because $0 \in K$ by definition.
		
		We do it by induction. Let $0 \le c \le d'$. Assume that we have constructed a family $\left(V_{k_0, \dots, k_{j}}\right)_{\substack{0 \le j \le c, \\ (k_0, \dots, k_j) \in \NN^{j+1}}}$, such that for all $j \in \{0,\dots, c\}$, we have:
		\begin{equation}
			K \subset \bigcup_{(k_0, \dots, k_j)\in\NN^j} V_{k_1, \dots, k_{j}},
		\end{equation}
		and such that for all $j\in\{1, c\}$ and all $(k_0, \dots, k_j)\in\NN^{j+1}$, we have $V_{k_0, \dots, k_{j}} \subset V_{k_0, \dots, k_{j-1}}$, with equality if and only if $V_{k_0, \dots, k_{j-1}} \subset K$.
		
		Let $(k_0, \dots, k_c) \in \NN^{c+1}$ be a multi-index such that $V_{k_1, \dots, k_{c}} \not \subset K$. 
		Then we have almost surely $g(V_{k_0, \dots, k_{c}}) \neq \{0\}$ so the restriction of $h$ to $V_{k_0, \dots, k_{c}}$ has rank at least $1$ almost surely. 
		By the previous argument, the set:
		\begin{equation*}
			K \cap V_{k_0, \dots, k_{c}} = \{x\in V_{k_0, \dots, k_{c}}\,|\, \PP(h x = 0) > 0\}
		\end{equation*}
		is included in the union of a countable family of subspaces of $V_{k_0, \dots, k_{c}}$ that have dimension $\dim\left(V_{k_0, \dots, k_{c}}\right) - 1$. 
		For all multi-index $(k_0, \dots, k_c)\in\NN^{c+1}$ such that $V_{k_1, \dots, k_{c}} \not\subset K$, we define $\left( V_{k_0, \dots, k_{c+1}} \right)_{k_{c+1} \in \NN}$ to be such a family.
		For every other multi-index $(k_0, \dots, k_c) \in \NN^{c+1}$, we define $V_{k_0, \dots, k_{c+1}} := V_{k_0, \dots, k_{c}}$ for all $k_c \in \NN$.
	\end{proof}

	\begin{Rem}\label{rem:ker}
		Let $\nu$ be a probability distribution over $\mathrm{End}(E)$ and $(\gamma_n)$ be a random sequence of distribution $\nu^{\otimes\NN}$. Then for every $x\in E$, the sequence $\PP(\overline{\gamma}_n x = 0) $ is non-decreasing and its limit is positive if and only if $x \in \underline\ker(\nu)$.
	\end{Rem}

	\subsection{Rank one boundary of a semi-group}
	
	Given a subset $A$ of a topological space $X$, we denote by $\mathbf{cl}_X(A)$ the closure of $A$ in $X$. Note that saying that an endomorphism has rank one is equivalent to saying that it is the product of a non-trivial vector on the left by a non-trivial linear form on the right.
	
	Given a probability measure $\nu$ on a topological space $X$, we denote by $\mathbf{supp}_X(\nu)$ or simply $\mathbf{supp}(\nu)$ the smallest closed subset of $X$ on which $\nu$ is supported. Then $\mathbf{supp}(\nu)$ is characterized by the fact that it is closed and for all open $\mathcal{U}\subset X$, we have $\nu(\mathcal{U}) > 0$ if and only if $\mathcal{U} \cap \mathbf{supp}(\nu) \neq \emptyset$. 
	
	We remind that given $E$ a vector space and $u\in E\setminus\{0\}$, we denote by $[u]$ the projective class of $u$ and we denote by $\mathbf{P}(E)$ the projective space of $E$. Given $X \subset  \mathbf{P}(E)$, we will write "Let $[u] \in X$" for "Let $u$ be a non-zero vector such that $[u] \in X$".

	\begin{Def}[Rank one boundary]
		Let $E$ be a Euclidean space and let $\Gamma < \mathrm{End}(E)$ be a sub-semi-group. Let $\overline\Gamma := \mathbf{cl}_{\mathrm{End}(E)}(\KK \Gamma)$. We denote by $\partial\Gamma$ the rank-one boundary of $\Gamma$, defined as:
		\begin{equation}
			\partial\Gamma :=  \{[\gamma]\,|\,\gamma\in\overline\Gamma,\; \rank(\gamma) = 1\} 
		\end{equation}
		We define the left and right boundaries of $\Gamma$ as:
		\begin{align*}
			\partial_u\Gamma & := \left\{[hv]\,\middle|\,[h] \in\partial \Gamma,\; v \in E \setminus \ker(h)\right\} \subset \mathbf{P}(E) \\
			\partial_w\Gamma & := \left\{[fh]\,\middle|\,[h] \in\partial \Gamma,\; f \in E^* \setminus \ker(h^*)\right\} \subset \mathbf{P}(E^*).
		\end{align*}
	\end{Def}

	\begin{Def}[Range and boundary of a distribution]
		Let $E$ be a Euclidean space and let $\nu$ be a probability measure on $\mathrm{End}(E)$. We denote by $\Gamma_\nu$ the range of $\nu$ defined as the smallest closed sub-semi-group of $\mathrm{End}(E)$ that has measure $1$ for $\nu$. We define $\partial\nu := \partial\Gamma_\nu$ and $\partial_u\nu := \partial_u\Gamma_\nu$ and $\partial_w\nu := \partial_w\Gamma_\nu$.
	\end{Def}
	
	\begin{Lem}[Invariant subspaces and irreducibility]\label{lem:carac-irred}
		Let $E$ be a Euclidean space.
		Let $\nu$ be a probability measure on $\mathrm{End}(E)$.
		Let $S := \mathbf{supp}(\nu)$. 
		Let $V \subset E$ be a proper non-trivial subspace.
		If $S V \subset V$ then $\nu$ is not irreducible.
		
		Let $N \ge 1$ and let $V_1, \cdots, V_n \subset E$ be proper non-trivial subspaces.
		If $\bigcup_{i =1}^N S \cdot V_i \subset \bigcup_{i =1}^N V_i$ then $\Gamma$ is not strongly irreducible.
	\end{Lem}
	
	\begin{proof}
		Let $\Gamma = \bigcup_{n\in\NN} S^{\cdot n} = \bigcup_{n\in\NN} \Pi(S^{n})$ be the semi-group generated by $S$.
		Note that for all $n\in\NN$, one has $\nu^{*n}(\Gamma) = 1$.
		Let $V \subset E$ be a proper non-trivial subspace such that $SV \subset V$. 
		The fact that $V$ is a proper subspace implies that there is a linear form $f \in E^*\setminus \{0\}$ such that $V \subset \ker(f)$.
		The fact that $V$ is not trivial implies that there is a vector $v \in V \setminus\{0\}$.
		Let $f$ and $v$ be as above.
		We have $\Gamma \cdot v \subset V$, hence $f \gamma v = 0$ for all $\gamma \in\Gamma$ so $\nu$ is not irreducible.
		
		Let $N \ge 1$ and let $V_1, \cdots, V_n \subset E$ be a family of proper non-trivial subspaces such that $\bigcup_{i =1}^N S \cdot V_i \subset \bigcup_{i =1}^N V_i$.
		Let $f_1, \dots, f_n \in E^{*}\setminus \{0\}$ be such that $V_i \subset \ker(f_i)$ for all $i\in\{1, \dots, N\}$.
		Let $v \in V_1 \setminus \{0\}$.
		Then one has $\Gamma \cdot v \subset \bigcup V_i$, hence $\prod_{i = 1}^N f_i \gamma v = 0$ for all $\gamma \in\Gamma$ so $\nu$ is not strongly irreducible.
	\end{proof}
	
	We call irreducible semi-group a semi-group $\Gamma \subset \mathrm{End}(E)$ such that for all $v \in E\setminus\{0\}$ and all $f \in E^*\setminus\{0\}$, there is an element $\gamma \in \Gamma$ such that $f \gamma v \neq 0$.

	\begin{Lem}\label{lem:boundary}
		Let $E$ be a Euclidean space and let $\Gamma < \mathrm{End}(E)$ be an irreducible semi-group. Then we have a factorization:
		\begin{equation}
			\partial\Gamma = \left\{ [uw] \,\middle|\, [u] \in\partial_u\Gamma, [w] \in\partial_w\Gamma\right\}.   
		\end{equation}
	\end{Lem}

	\begin{proof}
		Note that the space $\mathbf{P}(\mathrm{End}(E))$ is metrizable so the closure is characterized by sequences. Let $\pi$ be a rank one endomorphism. Saying that $[\pi] \in \partial \Gamma$ is equivalent to saying that there is a sequence $(\gamma_n) \in\left(\Gamma \setminus \{0\}\right)^\NN$ such that $[\gamma_n] \to [\pi]$. Note also that the product map is continuous so $\overline{\Gamma}$ is a semi-group. Therefore, for all $[\pi_1], [\pi_2] \in \partial \Gamma$, and for all $\gamma \in \Gamma$ such that $\pi_1 \gamma \pi_2 \neq 0$, we have $[\pi_1 \gamma \pi_2] \in\partial \Gamma$.
		
		Let $v_1 \in\partial_u\Gamma$ and let $f_2\in\partial_v \Gamma$. Let $v_2 \in E$ and $f_1 \in E^*$ be such that $[v_1 f_1] \in \partial\Gamma$ and $[v_2 f_2] \in \partial \Gamma$. By definition of the irreducibility, there is an element $\gamma \in \Gamma$ such that $f_1 \gamma v_2 \neq 0$. 
		Let $\gamma$ be such an element.
		Then we have $v_1 f_1 \gamma v_2 f_2 \neq 0$ hence $[v_1 f_1 \gamma v_2 f_2] \in \partial \Gamma$.
		Moreover $[v_1 f_1 \gamma v_2 f_2 ] = [v_1 f_2]$ because $f_1 \gamma v_2 $ is a scalar, therefore $[v_1 f_2] \in \partial \Gamma$. 
	\end{proof}
	
	\begin{Lem}[Characterisation of proximality]\label{lem:exiprox}
		Let $E$ be a Euclidean space and let $\nu$ be a probability measure on $\mathrm{End}(E)$. 
		Let $(\gamma_n)\sim \nu^{\otimes\NN}$. 
		Assume that $\nu$ is irreducible and that $\underline{\rank}(\nu) \ge 1$. 
		Then the following assertions are equivalent:
		\begin{enumerate}
			\item\label{exiprox:noprox} $\nu$ is not proximal (in the sense of Definition \ref{def:prox}).
			\item\label{exiprox:sqz} There is a constant $B$ such that $\sqz(\overline\gamma_n)\le B$ almost surely and for all $n\in\NN$.
			\item\label{exiprox:empty} $\partial\nu = \emptyset$.
		\end{enumerate}
	\end{Lem}
	
	\begin{proof}
		We assume \eqref{exiprox:sqz} and we claim that we have \eqref{exiprox:empty}. Note that $\sqz$ is a continuous map over $\mathrm{End}(E) \setminus\{0\}$ and it only depends on the projective class so for $B$ as in \eqref{exiprox:sqz}, we have $\sqz(h) \le B$ for all $h \in\overline{\Gamma}_\nu\setminus\{0\}$. It implies that there is no rank one endomorphism in $\overline{\Gamma}_\nu$. Therefore $\partial\nu = \emptyset$.
		
		Now we prove the converse by contraposition. Assume that $(\sqz(\overline\gamma_n))_{n\in\NN}$ is not almost surely bounded. It means that for all $B$, there is an integer $n$ such that $\PP(\sqz(\overline\gamma_n) > B) > 0$ or equivalently, there is a matrix $h \in \mathbf{supp}(\nu^{*n})$ such that $\sqz(h) > B$. Then there is a sequence $(h_k)\in\Gamma_\nu^\NN$ such that $\sqz(h_k) \to +\infty$. The space $\mathbf{P}(\mathrm{End}(E))$ is compact so the sequence $[h_k]$ has a limit point. Let $[\pi]$ be such a limit point, then we have $\sqz(\pi) = +\infty$ so $\pi$ has rank one, hence $[\pi]\in\partial \nu$.
		
		We assume that $\nu$ is proximal and show that \eqref{exiprox:sqz} is false. 
		The map $\prox$ is not continuous on $\mathrm{End}(E) \setminus \{0\}$ but it is on the set of proximal matrices.
		That means that given a matrix $h$ such that $\prox(h) > 0$, there is a neighbourhood $\mathcal{N}$ of $h$ such that $\prox(h') \ge \frac{1}{2}\prox(h)$ for all $h'\in \mathcal{N}$. 
		Let $n$ be an integer such that $\PP(\prox(\overline{\gamma}_n) > 0) > 0$, then there is a matrix $h \in\mathbf{supp}(\nu^{*n})$ such that $\prox(h) > 0$, it means that $\sqz(h^m) \underset{m\to \infty}{\longrightarrow} +\infty$. 
		Moreover, for all $m \in \NN$, we have $h^m \in \mathbf{supp}(\nu^{*nm})$, hence $\PP(\sqz(\overline{\gamma}_{nm}) \ge \sqz(h^m)-1) > 0$, which contradicts \eqref{exiprox:sqz}.
		
		Now we assume that $\partial \nu \neq \emptyset$ and show that $\nu$ is proximal.
		First we prove that there is $[\pi] \in \partial \nu$ such that $\prox(\pi) = +\infty$, which simply means that $\pi^2 \neq 0$. Let $[w] \in\partial_w\nu$. Let $[u] \in\partial_u\nu$ be such that $wu \neq 0$. 
		Such a $u$ exists because $\overline{\Gamma}_\nu$ is invariant by left multiplication by $\Gamma_\nu$.
		Therefore the set $\{0\}\cup\{u\in E\,|\,[u]\in\partial_u\nu\}$ is invariant by the action of $\Gamma_\nu$, which is irreducible, hence it is not included in $\ker(w)$ by Lemma \ref{lem:carac-irred}. 
		Then by Lemma \ref{lem:boundary}, we have $[uw] \in \partial \nu$ and since $wu\neq 0$, we have $(uw)^2\neq 0$. 
		Let $\mathcal{N}$ be an open neighbourhood of $uw$ such that $\prox(h') \ge 1$ for all $h'\in \mathcal{N}$. 
		This neighbourhood intersects $\overline\Gamma_\nu$ so $\KK\mathcal{N}$ intersects $\mathbf{supp}(\nu^{*n})$ for some $n$, which means that $\nu$ is proximal.
	\end{proof}
	
\subsection{Construction of the Schottky measure}

Remind that a measurable binary relation over a measurable space $\Gamma$ is a subset $\AA \subset \Gamma\times\Gamma$ that is measurable for the product $\sigma$-algebra. Given $g,h \in\Gamma$, we write $g \AA h$ to say that $(g,h) \in \AA$. 
Given $S, T \subset \Gamma$, we write $S \AA T$ to say that $S \times T \subset \AA$.
Given $n \in\NN$ and $(g_0, \dots, g_n )\in\Gamma^{n+1}$, we write $g_0 \AA \dots\AA g_n$ to say that $g_i \AA g_{i+1}$ for all $i \in\{0, \dots, n-1\}$.

	\begin{Def}
		Let $\Gamma$ be a measurable space, let $\AA$ be a measurable binary relation on $\Gamma$ and let $0 \le \rho < 1$. 
		Let $\nu_s$ be a probability measure on $\Gamma$. 
		We say that $\nu_s$ is $\rho$-Schottky for $\AA$ if:
		\begin{gather*}
			\forall h \in \Gamma, \nu_s\{\gamma \in \Gamma \,|\, h \AA \gamma\} \ge 1 - \rho \quad\text{and}\quad \nu_s\{\gamma \in \Gamma \,|\, \gamma \AA h\} \ge 1 - \rho.
		\end{gather*}
	\end{Def}
	
	\begin{Rem}
		Let $\Gamma$ be a measurable space, let $\AA \subset \AA'$ be measurable binary relation on $\Gamma$ and let $0 \le \rho \le \rho' < 1$. Let $\nu$ be a probability measure on $\Gamma$ that is $\rho$-Schottky for $\AA$. Then $\nu_s$ is also $\rho'$-Schottky for $\AA'$.
	\end{Rem}
	
	We recall that the alignment $\AA^\eps$ has been defined in Definition \ref{def:ali}.
	Given $0 < \eps \le 1$ and $g,h$ two matrices, we write $g\AA^\eps h$ when $\|gh\|\ge \eps \|g\|\|h\|$.
	
	\begin{Lem}\label{lem:pi-schottky}
		Let $E$ be a Euclidean space, let $\Gamma < \mathrm{End}(E)$ be a strongly irreducible semi-group and let $0 < \rho \le 1$. There exist an integer $N\in\NN$, a constant $\eps > 0$ and a family $([\pi_1], \dots, [\pi_N])\in\partial\Gamma^N$ such that:
		\begin{equation*}
			\forall h \in\mathrm{End}(E)\setminus\{0\},\; \#\{k\,|\,\pi_k\AA^\eps h\} \ge (1-\rho) N
		\text{ and }\#\{k\,|\,h\AA^\eps \pi_k\} \ge (1-\rho) N.
		\end{equation*}
	\end{Lem}
	
	\begin{proof}
		Let $d := \dim(E)$. Let $m \in \NN$. Assume that we have constructed a family $([u_1], \dots, [u_m]) \in \partial_u\Gamma^{m}$ that is in general position in the sense that for all $k \le d$, and for all $1 \le i_1 < \dots < i_k \le M$, we have $\dim\left(\langle u_{i_1}, \dots, u_{i_k}\rangle \right) = k$. Let:
		\begin{equation*}
			u_{m+1} \in \{u\in E\,|\,[u]\in\partial_u\Gamma\}\setminus \bigcup_{\substack{k \le d-1,\\ 1\le i_1<\dots <i_k\le m}} \langle u_{i_j}\rangle_{1 \le j\le k}.
		\end{equation*}
		Such a $u_{m+1}$ exists because $\{u\in E\,|\,[u]\in\partial_u\Gamma\}\cup\{0\}$ is $\Gamma$-invariant and it is not $\{0\}$ by Lemma \ref{lem:exiprox}. 
		Hence $\{u\in E\,|\,[u] \in \partial_u \Gamma\}$ can not be included in $\bigcup_{\substack{i_1<\dots <i_k}} \langle u_{i_j}\rangle_{1 \le j\le k}$, which is a finite union of hyperplanes. 
		Then $[u_{m+1}] \in\partial \Gamma$ and we can easily check that $([u_1], \dots, [u_{m+1}])$ is in general position. 
		
		let $M := \left\lceil \frac{d-1}{\rho} \right\rceil$.
		Let $([u_1], \dots, [u_{M}]) \in \partial_u\Gamma^{M}$ be in general position. We can construct such a family by induction using the above argument. Let $([w_1], \dots, [w_{M}]) \in \partial_w\Gamma^{M}$ be in general position. By the above argument applied to $\Gamma^* := \{\gamma^*\,|\,\gamma\in\Gamma\} \subset \mathrm{End}(E^*)$, which is also strongly irreducible, we can construct such a family.
		
		Let $N := M^2$. Given $i,j \in\{1, \dots,M\}$, we define $\pi_{Mi + j} := {u_i w_j}$. Then $([\pi_1], \dots, [\pi_N])\in\partial\Gamma^N$ by Lemma \ref{lem:boundary}. Note also that given $g,h \in\mathrm{End}(E)$, saying that $g\AA^\eps h$ is equivalent to saying that $\frac{\|gh\|}{\|g\|\|h\|}\ge \eps$. 
		
		Let $h \in\mathrm{End}(E) \setminus\{0\}$. Let $I^h := \{i\,|\, u_i \in\ker(h)\}$. Since the family $([u_1], \dots, [u_{M}])$ is in general position and $\langle u_i\rangle_{i\in I} \subset \ker(h)$, he have $\# I \le d-1$. Let $J^h := \{j\,|\, w_j \in\ker(h^*)\}$. By the same argument, we have $\#J^h \le d-1$. Now let:
		\begin{equation*}
			\psi(h) := \max_{\substack{I, J \subset \{1, \dots, M\},\\\#I\le d-1, \#J \le d-1}} \min\left\{\min_{i \notin I} \frac{\|h u_i\|}{\|h\|\|u_i\|},\, \min_{j \notin J} \frac{\|w_j h\|}{\|w_j\|\|h\|}\right\}.
		\end{equation*}
		By the previous argument, one has $\psi(h) > 0$ for all $h$. Moreover the maps $h\mapsto \frac{\|h v\|}{\|h\|\|v\|}$ and $\frac{\|f h\|}{\|f\|\|h\|}$ are continuous for all $v \in E\setminus\{0\}$ and all $f\in E^* \setminus\{0\}$. Hence $\psi$ is continuous. Moreover, $\psi$ is invariant by scalar multiplication so there is a continuous map $\phi : \mathbf{P}(\mathrm{End}(E)) \to (0,1]$ such that $\phi([h]) = \psi(h)$ for all $h \in\mathrm{End}(E) \setminus\{0\}$. The projective space $\mathbf{P}(\mathrm{End}(E))$ is compact so $\phi (\mathbf{P}(\mathrm{End}(E)))$ is compact. Let $\eps$ be its lower bound.
		
		Now let $h \in\mathrm{End}(E)\setminus\{0\}$. The set of indices $\{k \,|\, \pi_k\AA^{\eps} h\}$ is $\{M i + j\,|\, i,j \in\{1, \dots,M\},\,w_j \AA^{\eps} h\}$, which has cardinality at least $M (M-d+1)$. Indeed, since $\psi(h) \ge \eps$, a sufficient condition to have $w_j \AA^{\eps} h$ is for $j$ not to be included in a set $J$ that realises the maximum in the definition of $\psi(h)$. Moreover we have $M \rho \ge d - 1$ so $M (M-d+1) \ge (1 - \rho) N$. By the same argument $\#\{k \,|\, h\AA^{\eps} \pi_k\} \ge (1 - \rho) N$.
	\end{proof}

	\begin{Cor}\label{cor:S'-schottky} 
		Let $E$ be a Euclidean space, let $\nu$ be a strongly irreducible and proximal probability measure on $\mathrm{End}(E)$, let $0 <\rho \le 1$ and let $K \ge 8$. There exist an integer $N$, two constants $\alpha', \eps \in (0,1)$, a family $(n_{k})_{1\le k \le N} \in\NN^{N}$ and a family $(S'_{k})_{1\le k \le N}$ of measurable subsets of $\Gamma$ such that $\nu^{*n_{k}} (S'_{k}) \ge \alpha'$ for all $1\le k \le N$ and such that:
		\begin{gather}
			\forall h \in\mathrm{End}(E),\; \#\{k\,|\,S'_k\AA^{2\eps} h\} \ge (1-\rho) N
			\text{ and }\#\{k\,|\,h\AA^{2\eps} S'_k\} \ge (1-\rho) N,\label{item:align}\\
			\forall j \in\{1, \dots, N\},\; \#\{k\,|\,S'_k\AA^{\eps} S'_j\} \ge (1-\rho) N
			\text{ and }\#\{k\,|\,S'_j\AA^{\eps} S'_k\} \ge (1-\rho) N,\label{item:self-align}\\
			\forall h \in \bigcup_{k = 1}^N S'_k, \; \sqz(h) \ge  K |\log(\eps)| + K \log(2).
		\end{gather}
	\end{Cor}

	\begin{proof}
		Let $N \in \NN$, let $\eps > 0$ and let $([\pi_1], \dots, [\pi_N])\in\partial\nu^N$ be such that:
		\begin{equation*}
			\forall h \in\mathrm{End}(E)\setminus\{0\},\; \#\{k\,|\,\pi_k\AA^{3\eps} h\} \ge (1-\rho) N
			\text{ and }\#\{k\,|\,h\AA^{3\eps} \pi_k\} \ge (1-\rho) N.
		\end{equation*}
		Such a family exists by Lemma \ref{lem:pi-schottky}. To all $k \in\{1, \dots, N\}$, we associate the open set:
		\begin{equation*}
			S'_k := \left\{ h\in\mathrm{End}(E)\setminus\{0\}\,\middle|\,\eps > \left\|\frac{h}{\|h\|}-\frac{\pi_k}{\|\pi_k\|}\right\|,\;\sqz\left(h\right) > K |\log(\eps)| + K \log(2)\right\}.
		\end{equation*}
		Now let $k \in\{1, \dots, N\}$ and $h\in\mathrm{End}(E) \setminus\{0\}$. Assume that $h \AA^{3\eps} \pi_k$, then we claim that for all $h' \in S'_k$, we have $h \AA^{2\eps} h'$. Note that the right multiplication by $h'$ is $\|h'\|$-Lipschitz so $\left\|\frac{hh'}{\|h\|}- \frac{\pi_kh'}{\|\pi_k\|}\right\|\le \eps\|h'\|$. We assumed $h \AA^{3\eps} \pi_k$, which means that $\left\|\frac{\pi_kh'}{\|\pi_k\|}\right\|\ge 3\eps\|h'\|$, hence by triangular inequality, we have $\left\|\frac{hh'}{\|h\|}\right\|\ge 2 \eps \|h'\|$. The same reasoning works the same for the left alignment and we have:
		\begin{equation*}
			\forall h \in\mathrm{End}(E)\setminus\{0\},\; \#\{k\,|\,S'_k\AA^{2\eps} h\} \ge (1-\rho) N
			\text{ and }\#\{k\,|\,h\AA^{2\eps} S'_k\} \ge (1-\rho) N.
		\end{equation*}
		
		Now let $j,k \in\{1, \dots, N\}$ be such that $\pi_j \AA^{3\eps}\pi_k$ and let $h \in S'_j$ and $h' \in S'_k$. Then by the above argument, we have $h \AA^{2\eps} \pi_k$ and by the same argument, we have $h\AA^\eps h'$. Hence:
		\begin{equation*}
			\forall j \in\{1, \dots, N\},\; \#\{k\,|\,S'_k \AA^{\eps} S'_j\} \ge (1-\rho) N
			\quad\text{and}\quad\#\{k\,|\,S'_j \AA^{\eps} S'_k\} \ge (1-\rho) N.
		\end{equation*}
		
		Let $k\in\{1, \dots, N\}$. The interior of $S'_k$ contain $\pi_k$. It means that the interior of $S'_k$ intersects $\overline\Gamma_\nu = \mathbf{cl}\left(\bigcup_{n=0}^\infty \KK\mathbf{supp}(\nu^{*n})\right)$. Hence, there is an integer $n_k$ such that the interior of $S'_k$ intersects $\RR \mathbf{supp}(\nu^{*n_k})$ and then $S'_k$ intersects $\mathbf{supp}(\nu^{*n_k})$ because it is a cone and $\nu^{n_k}(S'_k) > 0$ by characterization of the support. 
		Let $(n_k)\in\NN^N$ be such that $\nu^{*n_k}(S'_k) > 0$ for all $k\in\{1, \dots, N\}$. Let $\alpha' := \min_k \nu^{*n_k}(S'_k)$. 
	\end{proof}
	
	\begin{Lem}\label{lem:schottky}
		Let $E$ be a Euclidean space, let $\nu$ be a strongly irreducible and proximal probability measure on $\mathrm{End}(E)$, let $0 < \rho < 1$ and let $K \ge 8$. There exist an integer $N$, two constants $\alpha, \eps \in (0,1)$, an integer $m$ and a family $(S_{k})_{1\le k \le N}$ of measurable subsets of $\mathrm{End}(E)$ such that $\nu^{*m} (S_{k}) \ge \alpha$ for all $1\le k \le N$ and such that:
		\begin{gather}
			\forall h \in\mathrm{End}(E),\; \#\{k\,|\,S_k\AA^{\eps} h\} \ge (1-\rho) N
			\text{ and }\#\{k\,|\,h\AA^{\eps} S_k\} \ge (1-\rho) N,\label{item:Schottky}\\
			\forall h \in \bigcup_{k = 1}^N S_k, \; \sqz(h) \ge  K |\log(\eps)| + K \log(2).
		\end{gather}
	\end{Lem}
	
	\begin{proof}
		Without loss of generality, we may assume that $\rho < \frac{1}{3}$
		Let $N \in\NN$, let $\alpha', \eps \in (0,1)$, let $(n_k)$ and let $(S'_k)$ be as in Corollary \ref{cor:S'-schottky}. To all index $k\in\{1, \dots, N\}$, we associate two indices index $i_k, j_k \in \{1, \dots, N\}$ such that $S'_k \AA^\eps S'_{i_k}$ and $S'_{i_k} \AA^\eps S'_{j_k}$ and $S'_{j_k} \AA^\eps S'_{k}$. By \eqref{item:self-align}, such indices $i_k,j_k$ exist because:
		\begin{equation*}
			\#\{(i,j) \,| \,S'_{i}\AA^\eps S'_k \AA^\eps S'_{i}\AA^\eps S'_{j} \AA^\eps S'_{k}\AA^\eps S'_{j}\}\ge (1 - 2\rho) N (1-3\rho) N > 0.
		\end{equation*}
		Hence the set of all possible values for $i_k,j_k$ is non-empty. Now let $m'$ be the smallest common multiple of $\{n_k + n_{i_k}\,|\, 1 \le k \le N\}$, let $m''$ be the smallest common multiple of $\{n_k + n_{j_k}\,|\, 1 \le k \le N\}$ and let $m = m' + m''$. Let $k \in \{1, \dots, N\}$. We define $p_k := \frac{m'}{n_k + n_{i_k}}$ and $q_k := \frac{m''}{n_k + n_{j_k}}$ and:
		\begin{equation*}
			S_k := (S'_k \cdot S'_{i_k})^{\cdot p_k}\cdot(S'_{j_k} \cdot S'_k)^{\cdot q_k} = \Pi\left( (S'_k \times S'_{j_k})^{p_k}\times(S'_{j_k} \times S'_k)^{q_k}\right).
		\end{equation*}
		Then we have $\nu^{*m}(S_k) \ge \left(\nu^{n_k}(S'_k)\nu^{n_{i_k}}(S'_{i_k})\right)^{p_k}\left(\nu^{n_{i_k}}(S'_{i_k})\nu^{n_{j_k}}(S'_{k})\right)^{q_k} \ge \alpha'^{2 p_k + 2 q_k} \ge \alpha'^{2m} =: \alpha$.
		
		Now let $h \in \mathrm{End}(E) \setminus \{0\}$ and let $k\in\{1, \dots, N\}$. Assume that $h \AA^{2\eps} S'_k$, then by Lemma \ref{lem:zouli-alignemont}, we have $h \AA^{\eps} S_k$. If we instead assume that $S'_k \AA^{2\eps} h$, then by the same argument applied to the transpose, we have $S_k \AA^\eps h$. Lemma \ref{lem:zouli-alignemont} also implies that $\min\sqz(S_k) \ge \min\sqz(S'_k)$ for all $k$.
	\end{proof}
	
	\begin{Cor}\label{cor:schottky}
		Let $E$ be a Euclidean space, let $\Gamma \in  \{\mathrm{End}(E), \mathrm{GL}(E)\}$ and let $\nu$ be a strongly irreducible and proximal probability measure on $\Gamma$, let $0 <\rho < 1$ and let $K \ge 8$. There exist an integer $m$, two constants $\alpha, \eps \in (0,1)$ and a probability measure $\tilde\nu_s$ on $\Gamma^m$ such that:
		\begin{enumerate}
			\item The measure $\Pi_*\tilde\nu_s$ is $\rho$-Schottky for $\AA^\eps$.
			\item The measure $\tilde\nu_s$ is absolutely continuous with respect to $\nu^{\otimes m}$ in the sense that $\alpha \tilde\nu_s \le \nu^{\otimes m}$.
			\item We have $\sqz_*\Pi_*\tilde\nu_s\left[K |\log(\eps)| + K \log(2),+\infty\right] = 1$ \ie for all $(h_1, \dots, h_m)$ in the support of $\tilde\nu_s$, we have $\sqz(h_1\cdots h_m) \ge K |\log(\eps)| + K \log(2)$.
			\item The support of $\tilde\nu_s$ is compact in $\Gamma^m$.
		\end{enumerate}
	\end{Cor}
	
	\begin{proof}
		Let $m,N\in\NN$, $\alpha,\eps\in(0,1)$ and $(S_k)_{1 \le k\le N}$ be as in Lemma \ref{lem:schottky}. 
		Define $f : \Gamma^m \to \RR_{\ge 0}$ as:
		\begin{equation*}
			f := \sum_{k = 1}^{N}\frac{\mathds{1}_{S_k} \circ \Pi}{\nu^{*m}(S_k)}.
		\end{equation*}
		Then $f \le \frac{N}{\alpha}$ because we assumed that $\nu^{*m}(S_k) \ge \alpha$ for all $k$. 
		Moreover, we can check that $\int_{\Gamma^m} f d\nu^{\otimes m} = N$ and for all $k \in\{1, \dots, N\}$ we have $\int_{\Pi^{-1}(S_k)} f d\nu^{\otimes m} \ge 1$. Let:
		\begin{equation*}
			\tilde\nu_s := \frac{f \nu^{\otimes m}}{\int_{\Gamma^m} f d\nu^{\otimes m}}\quad\text{and}\quad \nu_s := \Pi_*\tilde{\nu}_s.
		\end{equation*} 
		Then $\frac{\alpha}{N} \nu_s \le \nu^{*m}$ by definition. 
		Moreover, for all $I \subset \{1, \dots, N\}$, we have $\nu_s\left(\bigcup_{i\in I}S_i\right)\ge \frac{\# I}{N}$.
		Hence $\nu_s$ is $\rho$-Schottky by \eqref{item:Schottky}. 
		Moreover $\nu_s$ is supported on $\bigcup S_k$ and $\sqz(S_k)\subset \left[K |\log(\eps)| + K \log(2),+\infty\right]$ for all $k$ so $\sqz_*\nu_s\left[K |\log(\eps)| + K \log(2),+\infty\right] = 1$. 
		
		Now we only need to show that we can moreover assume that $\tilde\nu_s$ has compact support.
		Let $\beta \in (0,1)$. 
		There is a compact $\mathbf{C} \subset \Gamma^m$ such that $\tilde\nu_s({\mathbf{C}}) > 1-\beta$. 
		Let $\tilde{\nu}^{\mathbf{C}}_s := \frac{\mathds{1}_{\mathbf{C}}\tilde\nu_s}{\tilde\nu_s(K)}$. 
		Let $\nu_s^{\mathbf{C}} := \Pi_*\tilde{\nu}_s^{\mathbf{C}}$. 
		Then $\nu_s^{\mathbf{C}}\left(\bigcup_{i\in I}S_i\right)\ge \frac{1}{N} - \beta$ for all $I \subset \{1,\dots, N\}$. 
		Hence $\nu_s^{\mathbf{C}}$ is $(\rho + \beta)$-Schottky. 
		Moreover ${\alpha}{(1-\beta)} \tilde\nu_s \le \nu^{\otimes m}$. This is true for all $0 < \beta < \frac{1}{3}  - \rho$ so $\rho + \beta$ can take any value in $(0,1)$ and we always have ${\alpha}{(1-\beta)} > 0$.
	\end{proof}

	\section{Pivoting technique}\label{sec:pivot}

	\subsection{Statement of the result and motivation}

	In this section, we denote by $\Gamma$ a measurable semi-group that we endow with a binary relation $\AA$ and a subset $S \subset \Gamma$. The idea is to think of $\Gamma$ as $\mathrm{End}(E)$ or $\mathrm{GL}(E)$, think of $\AA$ as $\AA^\eps$ and think of $S$ as a compact subset of $\Gamma$ such that $\min\sqz(S) \ge K |\log(\eps)| + K \log(2)$.
	 
	We denote by $\widetilde\Gamma$ the associated semi-group of words. We define the semi-binary relation $\widetilde{\AA}^S \subset \Gamma \times \widetilde\Gamma$ recursively. Let $g \in\Gamma$ and let $\tilde\gamma \in \widetilde{\Gamma}$. We write $g \widetilde{\AA}^S \tilde\gamma$ if $g \in S$ and one of the following conditions holds:
	\begin{itemize}
		\item There exist words $\tilde{g}_0, \tilde{g}_1, \tilde{g}_2 \in\widetilde \Gamma$ such that $\tilde\gamma = \tilde{g}_0 \odot \tilde{g}_1 \odot \tilde{g}_2$, and $\Pi(\tilde{g}_1) \in S$ and $\Pi(\tilde{g}_0) \AA \Pi(\tilde{g}_1) \AA \Pi(\tilde{g}_2)$ and $g \AA \Pi(\tilde\gamma)$.
		\item There exist words $\tilde{g}_0, \tilde{g}_1, \tilde{g}_2 \in \widetilde\Gamma$ such that $\tilde\gamma = \tilde{g}_0 \odot \tilde{g}_1 \odot \tilde{g}_2$, and $\Pi(\tilde{g}_1) \in S$ and $g \widetilde{\AA}^S \tilde{g}_0$ and $\Pi(\tilde{g}_0)\AA \Pi(\tilde{g}_1) \AA \Pi(\tilde{g}_2)$.
	\end{itemize}
	
	Given $0 \le k < n \in\NN$, we write $\chi_k^n : \Gamma^{\{0, \dots,n-1\}} \to \Gamma\,;\, (\gamma_0, \dots, \gamma_{n-1})\mapsto \gamma_k$ for the $k$-th coordinate projection, in the same way, we define $\chi_k^\infty: \Gamma^{\NN} \to \gamma$. 
	We sometimes omit the total length $l \in\NN\cup \{\infty\}$ and write $\chi_k$ instead of $\chi_k^l$.
	Given $\gamma\in\tilde\Gamma$, and $k < L(\tilde\gamma)$, we call $\chi_k(\tilde\gamma)$ the $k$-th character of $\tilde{\gamma}$.
	
	\begin{Th}[Pivotal extraction]\label{th:pivot-extract}
		Let $\Gamma$ be a measurable semi-group, let $\AA$ be a binary relation on $\Gamma$ and let $S \subset\Gamma$ be measurable. Let $\nu$ be a probability measure on $\Gamma$, let $0< \alpha <1$, let $0< \rho <\frac{1}{5}$ and let $m \in\NN$. Let $\tilde\nu_s$ be a probability measure on $\Gamma^m$ such that $\alpha\tilde\nu_s\le \nu^{\otimes m}$ and let $\nu_s := \Pi_*\tilde\nu_s$. Assume that $\nu_s$ is supported on $S$ and $\rho$-Schottky for $\AA$. Then there exist constants $C, \beta > 0$ that only depend on $(\alpha, \rho, m)$ and an extraction $\tilde\mu$ of $\nu^{\otimes\NN}$ such that for $(\tilde{g}_k)_{k\in\NN}\sim \tilde\mu$, for $(\gamma_n) = \bigodot_{m =0}^\infty \tilde{g}_k \sim \nu^{\otimes\NN}$, all the following conditions hold:
		\begin{enumerate}
			\item For all $k \in\NN$, we have $\Pi\left(\tilde{g}_{2k}\right) \AA \Pi(\tilde{g}_{2k +1}) \widetilde{\AA}^S \tilde{g}_{2k+2}$ almost surely.\label{item:atilde}
			\item For all $k \in\NN$, we have $L (\tilde{g}_{2k+1}) = m$ almost surely and the conditional distribution of $\tilde{g}_{2k +1}$ relatively to $(\tilde{g}_j)_{j \neq 2k+1}$ is almost surely bounded above by $\frac{\tilde\nu_s}{1-2\rho}$ \ie for all measurable $A \subset \Gamma^m$, we have almost surely:
			\begin{equation}\label{eq:tjr-scho}
				\PP\left(\tilde{g}_{2k +1} \in A \,\middle|\,(\tilde{g}_j)_{j \neq 2k+1}\right) \le \frac{\tilde\nu_s(A)}{1-2\rho}
			\end{equation}
			\item For all $k \in\NN$, we have almost surely:
			\begin{equation}\label{eq:mmt-exp}
				\forall l \in\NN,\;\PP(L (\tilde{g}_{k}) > l\,|\,(\tilde{g}_j)_{j \neq k}) \le C \exp(-\beta l).
			\end{equation} 
			\item For all $n \in\NN$, and for all measurable $A \subset \Gamma \setminus \bigcup_{k = 0}^{m-1}\chi_k^m(\mathbf{supp}(\tilde\nu_s))$, we have:
			\begin{equation}\label{eq:densite}
				\PP\left(\gamma_n\in A \,\middle|\,(L(\tilde{g}_{2k}))_{k\in\NN}\right) \le \frac{\nu(A)}{1-\alpha}.
			\end{equation}
		\end{enumerate}
	\end{Th}

 	Sections \ref{schottky} and \ref{sec:algo-pivot} are devoted to the proof of Theorem \ref{th:pivot-extract}.
 	In section \ref{schottky}, we construct a preliminary extraction that gives us a sequence of independent random matrices alternating between an unknown distribution and the Schottky distribution $\nu_s$ of Theorem \ref{th:pivot-extract}.
 	Then from this ping-pong extraction, we construct an extraction for which the unknown words now have a large squeeze coefficient.
 	We do not claim that the words in this preliminary extraction are aligned.
 	
 	In section \ref{sec:algo-pivot}, we implement the pivoting technique to construct an extraction which is aligned.
 	This means that we look at the sequence of words that we have constructed in section \ref{schottky} from past to future.
 	All oddly indexed words are candidate pivotal times. 
 	The pivoting technique is an inductive way to eliminate pivotal times so that the word indexed by each selected pivotal time is aligned with its neighbours \ie aligned on the left with the product of everything until the previous selected pivotal time and aligned on the right with the product of everything until the next candidate or selected pivotal time.
 	
 	We move forward and select the current candidate pivotal time when the $\tilde\nu_s$-distributed word is aligned with both its neighbours.
 	Otherwise we eliminate the candidate pivotal time and move backwards \ie concatenate everything together and look at the last candidate pivotal time.
 	We then use a version of \eqref{eq:tjr-scho} which holds all over the inductive construction to show that the probability of backtracking each step is at most $\frac{\rho}{1-2\rho}$ and we get an exponential control over the size of the backtrack.
 	The issue is that the algorithm does not guarantee proper alignment but alignment in the sense of $\widetilde{\AA}^S$. 
 	Indeed the selected pivotal time that guarantees us the alignment satisfies three alignment conditions, hence \eqref{eq:tjr-scho} does not hold any more for this time, therefore we have to discard it. 
 	Then by induction we show that the previous candidate pivotal time is $\widetilde{\AA}^S$-aligned with the concatenation of all the words we have just concatenated together.
 	Hence the inductive and non-symmetric definition of $\widetilde{\AA}^S$.
 	
 	Note that in concrete applications the alignment $\widetilde{\AA}^S$ implies genuine alignment as testified by the next two results.
 	
 	\begin{Rem}[Rigidity of the alignment in the toy model]
 		Let $\Gamma = \langle a,b,c\,|\,a^2=b^2=c^2=\mathbf{1}\rangle$. Let $\AA := \{(g,h);\,|g \cdot h|= |g|+|h|\}$ and let $S := \Gamma\setminus\{\mathbf{1}\}$. Let $\gamma \widetilde{\AA}^S \tilde{g}$ and let $g = \Pi(\tilde{g})$. Then we have $\gamma \AA g$ and $g \in S$.
 	\end{Rem}

	\begin{Prop}[Rigidity of the alignment of matrices]\label{prop:pivot-is-aligned}
		Let $ E$ be a Euclidean vector space. Let $\Gamma = \mathrm{End}(E)$, let $\eps\in(0,1)$. Let $\AA \subset \AA^{\eps}$ and let $S \subset \{\gamma \in\Gamma\,|\,\sqz(\gamma)\ge 8|\log(\eps)|+ 10\log(2)\}$ be measurable. Let $(\tilde{\gamma}_n)_{n \in\NN} \in \widetilde\Gamma^{\NN}$ and let $({\gamma}_n)_{n \in\NN} := (\Pi(\tilde{\gamma}_n))_{n \in\NN}$. Assume that for all $n \in\NN$, we have:
		\begin{equation*}
			\gamma_{2n+1}\in S \quad\text{and}\quad
			\gamma_{2n} \AA \gamma_{2n+1} \widetilde{\AA}^S \tilde\gamma_{2n+2}.
		\end{equation*}
		Then we have $\gamma_{2n+1}\AA^\frac{\eps}{2} \gamma_{2n+2}$ and $\sqz(\gamma_{2n+2}) \ge 4|\log(\eps)| + 7 \log(2)$ for all $n \in\NN$ and $\gamma_i\cdots\gamma_{j-1}\AA^{\frac{\eps}{4}}\gamma_j\cdots\gamma_{k-1}$ for all $0\le i \le j \le k \in\NN$.
	\end{Prop}
	
	\begin{proof}
		Let $n \in\NN$. We want to show that $\gamma_{2n+1}\AA^\frac{\eps}{2} \gamma_{2n+2}$ and $\sqz(\gamma_{2n+2}) \ge 4|\log(\eps)| + 7 \log(2)$. 
		Let $\tilde{g}_0, \tilde{g}_1, \tilde{g}_2 \in\widetilde\Gamma$ be such that $\tilde\gamma_{2} = \tilde{g}_0 \odot \tilde{g}_1 \odot \tilde{g}_2$ and for all $i\in\{0,1,2\}$, let $g_i = \Pi(\tilde{g}_i)$. 
		Assume that $g_0 \AA g_1 \AA g_2$ and $g_1\in S$.
		Then by Lemma \ref{lem:triple-ali}, we have $\sqz(\gamma_{2k+2}) \ge 4|\log(\eps)| + 7\log(2)$.
		If we assume that $\gamma_{2n+1}\AA (g_0 g_1 g_2)$, then trivially $\gamma_{2n+1} \AA^\frac{\eps}{2} \gamma_{2n+2}$. 
		Otherwise, we assume that $\gamma_{2n+1} \widetilde{\AA}^S  \tilde{g}_0$. 
		
		We claim that for all $\gamma \widetilde{\AA}^S  \tilde{g}$, there is an integer $l \le L(\tilde{g})$ and a family $h_0,\dots, h_{2l} \in \mathrm{End}(E)$ such that $h_0\cdots h_{2l} = g$, and $\gamma \AA h_0$ and $\sqz(h_0) \ge 4|\log(\eps)| + 7\log(2)$ and for all $0\le i < l$, we have $h_0\cdots h_{2i}\AA h_{2i+1}\AA h_{2i+2}$ and $h_{2i+1}\in S$. 
		We prove the claim by induction on the length of $\tilde{g}$. Consider a decomposition $\tilde{g} = \tilde{g}_0 \odot \tilde{g}_1 \odot \tilde{g}_2$ with $g_1 \in S$. Since $g_1 \in S$, the word $\tilde{g}_1$ has positive length, therefore $L(\tilde{g}_0) < L(\tilde{g})$. 
		If $\gamma \AA g$, then we simply set $l := 0$ and $h:= g$. Note that if $L(\tilde{g})=1$, then $\gamma \AA g$. 
		If we do not have $\gamma \AA g$, then we are in the second case of the  definition of $\widetilde{\AA}^S$ and therefore we assume that $\gamma \widetilde{\AA}^S  \tilde{g}_0$. 
		By induction hypothesis there is an integer $l' \le L(\tilde{g}_0)$ and a family $h_0,\dots, h_{2l'}$ such that $h_0\cdots h_{2l'} = g_0$, and $\gamma \AA h_0$ and $\sqz(h_0) \ge 4|\log(\eps)| + 7\log(2)$ and for all $0\le i < l'$, we have $h_0\cdots h_{2i}\AA h_{2i+1}\AA h_{2i+2}$ and $h_{2i+1}\in S$. Let $l := l' + 1 \le L(\tilde{g})$, let $h_{2l-1} := g_1$ and let $h_{2l} := g_2$. Then the family $(h_0,\dots,h_{2l})$ satisfies the claim.
		
		We have constructed a family $h_0,\dots, h_{2l}$ such that $h_0\cdots h_{2l} = \gamma_{2n+2}$, and $\gamma_{2n+1}\AA^\eps h_0$ and $\sqz(h_0) \ge 4|\log(\eps)| + 5|\log(2)|$ and for all $0\le i <l$, we have $h_0\cdots h_{2i}\AA^\eps h_{2i+1}\AA^\eps h_{2i+2}$ and $\sqz(h_{2i+1})\ge 4|\log(\eps)| + 7\log(2)$. 
		Then by lemma \ref{lem:Atilde}, we have $\gamma_{2n+1}\AA^\frac{\eps}{2} \gamma_{2n+2}$.
		
		Let $0\le i \le j \le k \in\NN$, we have $\gamma_i\AA^\frac{\eps}{2} \dots \AA^\frac{\eps}{2} \gamma_{k-1}$ and $\sqz(\gamma_n)\ge 4|\log(\eps)| + 7\log(2) \ge 2|\log(\eps/2)| + 3\log(2)$ so by Lemma \ref{lem:alipart}, we have $\gamma_i\cdots\gamma_{j-1}\AA^{\frac{\eps}{4}}\gamma_j\cdots\gamma_{k-1}$.
	\end{proof}

	\subsection{Construction of the ping-pong extraction}\label{schottky}

	Let $0 < \alpha < 1$. 
	We write $\mathcal{G}_\alpha$ for the geometric probability measure of scale factor $\alpha$ defined by $\mathcal{G}_\alpha\{k\} := {\alpha^k}{(1-\alpha)}$ for all $k \in\NN$. 
	Note that $\mathcal{G}_\alpha$ has a finite exponential moment.
	We write $\mathcal{B}_\alpha$ for the Bernoulli measure of parameter $\alpha$, \ie the probability measure which gives mass $\alpha$ to $1$ and mass $1-\alpha$ to $0$.
	
	Given $\nu$ a probability distribution over a measurable semi-group, given $\eta$ a probability distribution on $\NN$ and given $(w,(\gamma_k))\sim \eta \otimes\nu^{\otimes\NN}$, we write $\nu^{*\eta}$ for the distribution of $\gamma_0\cdots\gamma_{w-1}$ and $\nu^{\otimes\eta}$ for the distribution of $(\gamma_0,\dots,\gamma_{w-1})$. 
	When $\tilde\nu$ is defined on a semi-group of words, we write $\tilde\nu^{\odot\eta}$ instead of $\tilde\nu^{*\eta}$.
	
	Given $\eta$ be a probability distribution on $\ZZ$, given $m \in\ZZ$ and given $w \sim \eta$, we write $m +_* \eta $ for the distribution of $m+w$, and we write $m \times_*\eta$ for the distribution of $m\times w$.
	Given $\tilde\kappa$ a probability distribution on $\widetilde\Gamma$, we write $\tilde\kappa^{\odot\NN}$ for the distribution $\bigodot^\infty_*\tilde\kappa^{\otimes\NN}$, which is defined on $\Gamma^\NN$.
	The distribution $\tilde\kappa^{\otimes\NN}$ is defined on $\tilde\Gamma^\NN$ and $\tilde\kappa^{\odot\NN}$ is defined on $\Gamma^\NN$.
	
	The following lemma comes from basic probability theory, we give a complete proof as a warm up.
	
	\begin{Lem}\label{lem:ping-pong-pas-sqz}
		Let $\Gamma$ be a measurable semi-group. 
		Let $\nu$ be a probability measure on $\Gamma$, let $0 < \alpha < 1$, and let $m \in \NN$. 
		Let $\tilde\nu_s$ be a probability measure on $\Gamma^m$ such that $\alpha \nu_s \le \nu^{\otimes m}$. 
		Let $\tilde\kappa : = \left(\frac{\nu^{\otimes m} - \alpha \tilde{\nu}_s}{1 - \alpha}\right)^{\odot \mathcal{G}_{1-\alpha}}$. 
		Then $(\tilde\kappa \otimes \tilde\nu_s)^{\otimes\NN}$ is an extraction of $\nu^{\otimes \NN}$. 
		Moreover, for all $A \subset \Gamma$, and for $\tilde g \sim \tilde\kappa$, we have:
		\begin{equation}\label{eq:densite-ping-pong}
			\forall l \in m\NN,\forall  0\le k < l,\;\PP\left(\chi_k^{l}(\tilde{g})\in A\,\middle|\,L(\tilde{g}) = l\right) \le \frac{\nu(A)}{1-\alpha}.
		\end{equation} 
	\end{Lem}
	
	\begin{proof}
		Let $(x_n)\sim \left(\alpha\delta_1^{\otimes m} + (1-\alpha)\delta_0^{\otimes m}\right)^{\odot\NN}$ \ie $x_{km +r} = x_{km}$ for all $k\in\NN$ and all $0\le r < m$ and $(x_{km})_{k \in\NN}\sim\mathcal{B}_\alpha^{\otimes\NN}$. 
		Let $(g_n) \sim \tilde\nu_s^{\odot\NN}$ and let $(h_n) \sim \left(\frac{\nu^{\otimes m} -\alpha \tilde\nu_s}{1-\alpha}\right)^{\odot\NN}$. 
		Assume that $x, g, h$ are independent. 
		Let $(\gamma_n)_{n\in\NN} := \left( g_n^{x_n} h_n^{1-x_n} \right)_{n\in\NN}$ \ie $\gamma_n = g_n$ for all $n\in\NN$ such that $x_n = 1$ and $\gamma_n = h_n$ for all $n\in\NN$ such that $x_n = 0$.
		Then the sequence $((\gamma_{km},\dots, \gamma_{(k+1)m-1}))_{k\in\NN}$ is i.i.d. because the random sequence $\left((x_{km}, \dots, x_{(k+1)m-1}, g_{km}, \dots, g_{(k+1)m-1}, h_{km}, \dots, h_{(k+1)m-1})\right)_{k\in\NN}$ is. 
		Moreover, for all $k \in\NN$, the distribution of $\left(\gamma_{km}, \dots, \gamma_{(k+1)m-1}\right)$ is $\alpha \tilde\nu_s + (1-\alpha) \frac{\nu^{\otimes m} -\alpha \tilde\nu_s}{1-\alpha} = \nu^{\otimes m}$. 
		Hence $(\gamma_n) \sim \nu^{\otimes\NN}$. 
		
		Now let $(w'_j)_{j\in\NN}$ be a random sequence of integers such that for all $j \in \NN_{\ge 1}$, the integer $\overline{w}'_j$ is almost surely the $j$-th smallest element of $\{k \ge 1\,|\,x_{km - 1}= 1\}$. 
		With that definition, $(w'_k-1)_{k\in\NN}$ is the sequence of number of failures between consecutive successes of a Bernoulli process of parameter $\alpha$ so $w'_j \sim (1+_*\mathcal{G}_{1-\alpha})^{\otimes\NN}$. 
		Let $(w_k)$ be the random sequence of integers such that $w_{2k+1} = m$ and $w_{2k} = m '(w'_k-1)$ for all $k\in\NN$. 
		Then for all $k\in\NN$, we have $\widetilde\gamma^w_{2k+1} = \widetilde{g}^w_{2k+1}$ and $\widetilde\gamma^w_{2k} = {\widetilde{h}}^w_{2k}$.
		Moreover, the sequences $w$, $g$ and $h$ are independent so $\widetilde\gamma^w_{2k}$ is independent of $\widetilde\gamma^w_{2k+1}$ for all $k$ and $(\widetilde\gamma^w_{2k},\widetilde\gamma^w_{2k+1})$ is i.i.d. 
		
		Given $j,k\in\NN$ the event $(j = \overline{w}'_{k+1})$, implies that $\widetilde{\gamma}^w_{2k+1} = (g_{mj-m}, \dots, g_{mj-1})$.
		Moreover, the value of $(g_{mj-m}, \dots, g_{mj-1})$ is independent of $(x_n)$ so it is independent of $\overline{w}'_{k+1}$.
		Therefore, we have $\widetilde{\gamma}^w_{2k+1} \sim \tilde\nu_s$ conditionally to the event $(j = \overline{w}'_{k+1})$ so $\widetilde{\gamma}^w_{2k+1} \sim \tilde\nu_s$. 
		Given $k\in\NN$ and given $j := \overline{w}'_{k}$, we have $\widetilde{\gamma}^w_{2k} = (h_{jm}, \dots, h_{(j+{w}'_k-1)m-1})$ and ${w}'_k - 1 \sim \mathcal{G}_{1 - \alpha}$ and it is independent of the random sequence $h$ so $\widetilde{\gamma}^w_{2k} \sim \left(\frac{\nu^{\otimes m}-\alpha \tilde{\nu}_s}{1-\alpha}\right)^{\odot \mathcal{G}_{1-\alpha}}$.
		
		We have shown that the distribution of $(\widetilde\gamma^w_k)$ is $(\tilde\kappa \otimes \tilde\nu_s)^{\otimes\NN}$. 
		We also proved in the previous paragraph that $(\gamma_n) \sim \nu^{\otimes\NN}$. 
		We also note that  the $w_k$'s all have a bounded exponential moment.
		Therefore $(\tilde\kappa \otimes \tilde\nu_s)^{\otimes\NN}$ is an extraction of $\nu^{\otimes \NN}$ in the sense of Definition \ref{def:extract}. 
		
		Now we show \eqref{eq:densite-ping-pong}.
		Note that $\frac{\nu^{\otimes m}-\alpha \tilde{\nu}_s}{1-\alpha}\le \frac{\nu^{\otimes m}}{1-\alpha}$ so for all $j <m$, we have $(\chi_{j}^m)_*\left( \frac{\nu^{\otimes m} - \alpha \tilde{\nu}_s}{1 - \alpha} \right) \le \frac{\nu}{1-\alpha}$. 
		Moreover, given $\tilde{g}\sim \tilde\kappa$, given $l \in \NN$ and given $k < lm$, the distribution of $\tilde{g}$ conditioned to $\left( L(\tilde{g}) = lm \right)$ is $\left(\frac{\nu^{\otimes m} - \alpha \tilde{\nu}_s}{1-\alpha}\right)^{\otimes l}$.
		Therefore, the conditional distribution of $\chi^l_{k}(\tilde{g})$ is $\left(\chi_{k- m\lfloor\frac{k}{m}\rfloor}^m\right)_*\left(\frac{\nu^{\otimes m} - \alpha \tilde{\nu}_s}{1 - \alpha}\right) \le \frac{\nu}{1 - \alpha}$.
	\end{proof}
	
	\begin{Lem}\label{lem:ping-pong}
		Let $\Gamma$ be a measurable semi-group, let $\AA$ be a binary relation on $\Gamma$ and let $S \subset\Gamma$ be measurable. Let $\nu$ be a probability measure on $\Gamma$, let $0< \alpha <1$, let $0< \rho <\frac{1}{5}$ and let $m \in\NN$. Let $\tilde\nu_s$ be a probability measure on $\Gamma^m$ such that $\alpha\nu_s\le \nu^{\otimes m}$ and let $\nu_s := \Pi_*\tilde\nu_s$. Assume that $\nu_s$ is supported on $S$ and $\rho$-Schottky for $\AA$. Then there exists a distribution $\tilde\kappa$ on $\widetilde{\Gamma}$ such that $(\tilde\kappa\otimes\tilde\nu_s)^{\otimes\NN}$ is an extraction of $\nu$ and $L_*\tilde\kappa\{k\} = m+_*m\times_*(1+_*\mathcal{G}(1-\alpha))^{*(1+_*\mathcal{G}(2\rho))}$. Moreover, for $\tilde\kappa$ almost all $\tilde{g}\in\widetilde\Gamma$, there exist $\tilde{g}_1,\tilde{g}_2,\tilde{g}_3 \in\widetilde\Gamma$ such that $\tilde{g} = \tilde{g}_1 \odot \tilde{g}_2 \odot \tilde{g}_3$ and $\Pi(\tilde{g}_2) \in S$ and $\Pi(\tilde{g}_1) \AA \Pi(\tilde{g}_2) \AA \Pi(\tilde{g}_3)$ and for all measurable $A \subset \Gamma \setminus \bigcup_{k = 0}^{m-1}\chi_k^m(\mathbf{supp}(\tilde\nu_s))$, we have almost surely:
		\begin{equation}\label{eq:densite-sqz}
			\forall l \in m\NN_{\ge 1},\forall  0\le n < l,\;\PP\left(\chi_n^{l}(\tilde{g})\in A\,\middle|\,L(\tilde{g}) = l\right) \le \frac{\nu(A)}{1-\alpha}.
		\end{equation} 
	\end{Lem}
	
	\begin{proof}
		Let $\tilde\kappa_0 : = \left(\frac{\nu^{\otimes m}-\alpha \tilde{\nu}_s}{1-\alpha}\right)^{\odot \mathcal{G}_{1-\alpha}}$ be as in Lemma \ref{lem:ping-pong-pas-sqz}. Let $(\gamma_n) \sim\nu^{\otimes\NN}$ and let $(w_n)_{n\in\NN}$ be a random sequence of integers on the same probability space such that $(\widetilde\gamma^w_k) \sim \left(\tilde\kappa_0 \otimes \tilde\nu_s\right)^{\otimes\NN}$. Let $\mathcal{U}_{[0,1]}$ be the uniform probability measure on the interval $[0,1]$ and let $(\tau_j)_{j\in\NN}\sim\mathcal{U}_{[0,1]}^{\otimes \NN}$. Assume that the random sequences $\widetilde\gamma^w$ and $\tau$ are independent. We define the penalty function:
		\begin{equation*}
			P_{\nu_s} :\, \Gamma^3 \longrightarrow [0,1];\, (f,g,h) \longmapsto \frac{\mathds{1}_{f\AA g\AA h}}{\nu_s\{\gamma\in\Gamma\,|\,f\AA\gamma\AA h\}}(1-2\rho).
		\end{equation*}
		We check that $P_{\nu_s} \le 1$ because $\nu_s$ is $\rho$-Schottky for $\AA$.
		Note also that for $\gamma\sim \nu_s$ and for all non random $g,h \in\Gamma$, one has $\EE(P_{\nu_s} (f, \gamma, h)) = 1-2\rho$. 
		Hence, for all $k\in\NN$, for all random $f,h\in\Gamma$ that are independent of $(\tau_k, \gamma^w_{2k+1})$, we have $\PP\left(\tau_k < P_{\nu_s}\left(f,\gamma^w_{2k+1},h\right)\right)= 1- 2\rho$ by definition of $P_{\nu_s}$. 
		Moreover $\tau_k < P_{\nu_s}\left(f,\gamma^w_{2k+1},h\right) \Rightarrow f \AA \gamma^w_{2k+1} \AA h$. Now we define:
		\begin{equation*}
			k_{\gamma, w}^\tau := \min \left\{k \in \NN\,|\,\tau_k < P_{\nu_s}\left(\overline{\gamma}^w_{2k+1},\gamma^w_{2k+1},\gamma^w_{2k+2}\right) \right\}.
		\end{equation*}
		Then for all $k \in\NN$, we have:
		\begin{equation*}
			\PP\left(k_{\gamma, w}^\tau = k\,\middle|\,(\widetilde\gamma^w_{2k'})_{k'\in\NN}, k_{\gamma, w}^\tau \ge k\right) = \EE\left(P_{\nu_s}\left(\overline{\gamma}^w_{2k+1},\gamma^w_{2k+1},\gamma^w_{2k+2}\right) \,\middle|\,(\widetilde\gamma^w_{2k'})_{k'\in\NN}\right) = 1 - 2\rho.
		\end{equation*}
		Therefore $k_{\gamma, w}^\tau \sim \mathcal{G}_{(2\rho)}$ and $k_{\gamma, w}^\tau$ is independent of $(\widetilde\gamma^w_{2k})_{k\in\NN}$ .
		Let $j_{\gamma, w}^\tau := \overline{w}_{2k_{\gamma, w}+3}$. 
		Then $j_{\gamma, w}^\tau \sim m +_* m \times_* (1 +_* \mathcal{G}(1 - \alpha))^{* \left(1 +_* \mathcal{G}(2\rho) \right)}$ so $j_{\gamma, w}^\tau$ has finite exponential moment by Lemma \ref{lem:sumexp}.
		Let $\tilde\kappa$ be the distribution law of $\tilde{g}^\tau_{\gamma,w} : = (\gamma_0,\cdots,\gamma_{j_{\gamma, w}^\tau - 1})$. 
		It follows from the definition that $\tilde{g}_{\gamma,w}^\tau = (\bigodot_{i =0}^{2k}\widetilde\gamma^w_{i})\odot\widetilde\gamma^w_{2k+1}\odot \widetilde\gamma^w_{2k+2}$ and $\Pi(\bigodot_{i =0}^{2k}\widetilde\gamma^w_{i})\AA \Pi (\widetilde\gamma^w_{2k+1})\AA \Pi(\widetilde\gamma^w_{2k+2})$ for $k = k_{\gamma,w}^\tau$. 
		Moreover $\Pi (\widetilde\gamma^w_{2k+1}) \sim \nu_s$ so $\Pi (\widetilde\gamma^w_{2k+1})\in S$ almost surely.
		
		Note also that $k_{\gamma,w}^\tau$ is constructed as a stopping time for the sequence $(\tau_k, \gamma^w_{2k+1})_{k\in\NN}$ and it is independent of $(\gamma^w_{2k})_{k\in\NN}$, so the conditional distribution of the random sequence $\left(\widetilde\gamma^w_{k+2k_{\gamma, w} +3}\right)_{k\in\NN}$ knowing $\tilde{g}_{\gamma,w}$ is $\left(\tilde\nu_s \otimes \tilde\kappa_0\right)^{\otimes\NN}$. Hence, we have $\tilde\kappa \odot \tilde\nu_s \odot \left(\tilde\kappa_0 \odot \tilde\nu_s\right)^{\odot\NN} = \left(\tilde\kappa_0 \odot \tilde\nu_s\right)^{\odot\NN}$ so $\left(\tilde\kappa \odot \tilde\nu_s\right)^{\odot\NN} = \left(\tilde\kappa_0 \odot \tilde\nu_s\right)^{\odot\NN} = \nu^{\otimes\NN}$.
		
		Now let $A \subset \Gamma \setminus \bigcup_{k = 0}^{m-1} \chi_k^m(\mathbf{supp}(\tilde\nu_s))$ be measurable and let $l \in m \NN_{\ge 1}$. Let $q \in\NN$ and let $x_0, \dots, x_{q+1} \in m\NN$ be such that $q m + \sum_{i = 0}^{q+1} x_i = l$. 
		Now we work on the sub-probability space $(\Omega', \PP')$, (where $\PP'$ is short for $\PP^{(l, q, (x_i))}$), defined as $\Omega' := (j_{\gamma,w}^\tau = l)\cap (k_{\gamma,w}^\tau = q) \cap \bigcap_{i = 0}^{q+1} (w_{2i} = x_i)$ and $\PP' := \frac{\PP}{\PP(\Omega')}$. 
		Let $n < l$. We claim that $\PP'(\gamma_n \in A) \le \frac{1}{1-\alpha} \nu(A)$.
		If there exists an integer $k \le q$ such that $\overline{w}_{2k+1} \le n <  \overline{w}_{2k+2}$, then $\PP'(\gamma_n \in A) = 0$.
		Otherwise, let $k \le q + 1$ be such that $\overline{w}_{2k} \le n <  \overline{w}_{2k+1}$, then we have $\PP'$ almost surely $\gamma_n = \chi^{x_k}_{n-\overline{w}_{2k}} (\widetilde\gamma^w_{2k})$ and $k_{\gamma,w}$ is independent of $(\widetilde\gamma^w_{2j})_{j\in\NN}$, so the distribution of $\gamma_n$ for $\PP'$ is bounded by $\frac{\nu}{1-\alpha}$ by \eqref{eq:densite-ping-pong} in Lemma \ref{lem:ping-pong-pas-sqz}. 
		Then we have $\PP\left(\gamma_n \in A\,\middle|\, j_{\gamma,w}^\tau = l\right) \le \max_{q,(x_i)}\PP'(A) \le \frac{1}{1-\alpha} \nu(A)$ and this is true for all $l \in m\NN_{\ge 1}$.
		Therefore, we have \eqref{eq:densite-sqz}.
	\end{proof}

\subsection{Construction of the aligned extraction}\label{sec:algo-pivot}

The following definition describes the pivot algorithm.
Starting from an already merged sequence $(\widetilde\gamma^w_n)$, we will merge some words recursively.
At each step $j$, we only look at $(\gamma^w_0, \gamma^w_1, \dots, \gamma^w_{2j})$ and merge some of them together.
We will denote by $(p^k_j)_{k\in\NN}$ the sequence of waiting times (or lengths) at step $j$, starting with $p^0 =w$.
We will denote by $2m_j + 1$ the number of words left after merging $(\gamma^w_0, \gamma^w_1, \dots, \gamma^w_{2j})$.
At each step, we make sure that $\left(\widetilde\gamma^{p^{j}}_k\right)_{k \le m_j}$ satisfies Theorem \ref{th:pivot-extract}.
It means that every oddly indexed block is a single oddly indexed word and that its distribution relatively to the merging process is still a Schottky distribution.
The merging process consists in backtracking when the right alignment conditions are satisfied.

\begin{Def}[Weighted Pivot algorithm]\label{def:pivot}
	Let $\Gamma$ be a measurable semi-group endowed with a measurable relation $\AA$. 
	Let $\nu_s$ be a probability distribution on $\Gamma$ that is $\rho$-Schottky for $\AA$.
	We define the $\nu_s$ penalty functions.
	\begin{gather*}
		P_{\nu_s}: \Gamma^3 \longrightarrow [0,1];\,  (f,g,h) \longmapsto \frac{\mathds{1}_{f\AA g\AA h}}{\nu_s\{\gamma\in\Gamma\,|\,f\AA\gamma\AA h\}}(1-2\rho), \\
		P'_{\nu_s}:  \Gamma^4 \longrightarrow [0,1];\,(f,g,h, h')\longmapsto \frac{\mathds{1}_{f\AA g\AA h}\mathds{1}_{g\AA h'}}{\nu_s\{\gamma\in\Gamma\,|\,f\AA\gamma\AA h\text{ and }\gamma \AA h'\}}(1-3\rho).
	\end{gather*}
	Let $(\gamma_n)\in\Gamma^{\NN}$, let $(w_k)\in\NN_{\ge 1}^\NN$ and let $(\tau_k)\in[0,1]^\NN$ be non-random sequences.
	Let $\left(p_k^j\right)_{j\in\NN, k\in\NN}\in\NN_{\ge 1}^{\NN^2}$.
	Assume that for all $j \in\NN$, there is an even integer $m_j$ such that $\overline{p}_{2m_j+1}^k := \sum_{k = 0}^{2m_j} p^j_k = \overline{w}_{2j +1}$ and let $(m_j)_{j\in\NN}$ be such a sequence.
	Given $j,k \in\NN$, we write $l_k^j := \max\left\{l \le j\,\middle|\,m_l \le k\right\}$. For all $k \in\NN$, we write $l_k := \sup\left\{l \in\NN \,\middle|\,m_l \le k\right\}$ for the time of the last visit in $k$.
	We say that $(p_k^j)$ is the family of length of the pivotal blocks associated to the sequence $(\widetilde\gamma^w_k)$ with weights $(\tau_k)$ if:
	\begin{enumerate}
		\item For all $j \in\NN$, we have $\left(p^j_{k + 2m_j +1}\right)_{k\in\NN} = (w_{k + 2j +1})_{k\in\NN}$ and $\left\{\overline{p}^{j+1}_k\,\middle|\, k \in \NN\right\} \subset \left\{\overline{p}^j_k\,\middle|\, k \in\NN \right\}$. Note that it implies that $m_0 =0$ and that $p^0_k = w_k$ for all $k \in\NN$.
		\item For all $j \in\NN$, we have $\left(p^{j+1}_k\right)_{k\in\NN} = \left(p^j_k\right)_{k\in\NN}$ and $m_{j+1} = m_j + 1$ if and only if:
		\begin{equation*}
			\tau_j < P_{\nu_s}\left(\gamma^{p^j}_{2m_j},\gamma^{p^j}_{2m_j + 1},\gamma^{p^j}_{2m_j + 2}\right).
		\end{equation*}
		\item \label{item:backtrack} For all $j \in\NN$ such that $\left(p^{j+1}_k\right)_{k\in\NN} \neq \left(p^j_k\right)_{k\in\NN}$, we have $\left(p^{j+1}_k\right)_{0\le k < 2m_{j+1}} = \left(p^j_k\right)_{0 \le k < 2m_{j+1}}$ and
		\begin{equation}\label{eq:def-m+1}
			m_{j+1} = \max \left(\left\{k < m_j \,\middle|\,\tau_{l^j_k}< P'_{\nu_s}\left(\gamma^{p^j}_{2k}, \gamma^{p^j}_{2k+1}, \gamma^w_{2l_{k}^j+2},\gamma^{p^j}_{2k+2}\cdots\gamma^{p^j}_{2m_j+2}\right)\right\}\cup\{0\}\right).
		\end{equation}
	\end{enumerate}
	If $p^j_k$ converges to a limit $p_k$ for all $k$, as $j \to +\infty$, then we say that $(p_k)_{k\in\NN}$ is the sequence of pivotal times associated to $\widetilde\gamma^w$ with weights $\tau$.
\end{Def}

Let us illustrate the first steps of the algorithm on an example.
Initially, the letters are grouped into blocks of length $p^0_0 = w_0, p^0_1 =w_1, p^0_2 = w_2, \dots$.
For simplicity, we will take all $w_k$ equal to $1$ in our example.
With our previous construction, this happens when $\nu_s = \nu$, note also that the identity must be aligned with everyone.
The important thing to note is that in that case, all words are in $S$.
That way, for all $0 \le k \le j$, we have $\overline{p}^j_{2k+1} = 2l^j_k+1$.
We will describe the construction for $j \in \{0,1,2,3,4\}$.
For that construction, we only look at the first $11$ words that all have a single letter, which is an element of a semi-group (a semi- group of matrices in our case):
\begin{equation*}
	(\gamma_{0});
	[\gamma_{1}],
	(\gamma_{2}),
	[\gamma_{3}],
	(\gamma_{4}),
	[\gamma_{5}],
	(\gamma_{6}),
	[\gamma_{7}],
	(\gamma_{8}),
	[\gamma_{9}],
	(\gamma_{10}).
\end{equation*}
We mark with brackets (instead of the usual parenthesis for words) the candidate pivotal times, at step $j= 0$ they are all the oddly indexed times.
At all times, the word within brackets will have a single letter.
We mark with a semi column, our position. 
At all time, all the words that are on the left of this semi-column are aligned and the oddly indexed ones are candidate pivotal times, there are $m_j$ of them.
At step $j = 0$, we will add $\gamma_1$ and $\gamma_2$. 
We check whether $\tau_0 < P_{\nu_s} (\gamma_0, \gamma_1, \gamma_2)$, which is a proxy for $\gamma_0 \AA \gamma_1 \AA \gamma_2$ but with a controlled conditional probability, constant and equal to $1 - 2\rho$.

If this condition fails (which is always the case when the above alignment condition si not satisfied), then we merge $(\gamma_0, \gamma_1, \gamma_2)$ into a single word.
Then, there is nothing more to check because there is no candidate $k < m_{1}$ for $m_{1}$ to satisfy \eqref{item:backtrack}.
In this case, $m_1 = 0$, so there are no candidate pivotal times left of the semi-column and the newly merged sequence is the following:
\begin{equation*}
	(\gamma_{0},\gamma_{1},\gamma_{2});
	[\gamma_{3}],
	(\gamma_{4}),
	[\gamma_{5}],
	(\gamma_{6}),
	[\gamma_{7}],
	(\gamma_{8}),
	[\gamma_{9}],
	(\gamma_{10}),\dots
\end{equation*}
Then at step $j = 1$, we check whether $\tau_1 < P_{\nu_s} (\gamma_0\gamma_1\gamma_2, \gamma_3, \gamma_4)$. 
Assume that this condition holds. 
This implies that we have the alignment $\gamma_0\gamma_1\gamma_2 \AA \gamma_3 \AA \gamma_4$.
Then we do not merge any block and move to $m_2 = 1$:
In this case, we have $p^2_0 = 3, p^2_1 = 1, p^2_2 =1, p^2_3 =1, p^2_4= 1, \dots$ and the new sequence is the following:
\begin{equation*}
	(\gamma_{0},\gamma_{1},\gamma_{2}),
	[\gamma_{3}],
	(\gamma_{4});
	[\gamma_{5}],
	(\gamma_{6}),
	[\gamma_{7}],
	(\gamma_{8}),
	[\gamma_{9}],
	(\gamma_{10}),\dots
\end{equation*}
Note that it is useless to specify that $p^2_1 =1$ and $p^2_3 =1$ or that $p^2_{2k +1} = 1$ for all $k \in \NN$, because the oddly indexed blocks are the ones in brackets and they always have length $1$. 
It is also useless to mention that $p^2_4 = 1$, because it is the length of the block $(\gamma_6)$, which we have not yet considered.
At step $j = 2$, we check whether $\tau_2 < P_{\nu_s} (\gamma_4, \gamma_5, \gamma_6)$. Assume that this holds. 
Then $m_3 = 2$ and the new sequence is:
\begin{equation*}
	(\gamma_{0},\gamma_{1},\gamma_{2}),
	[\gamma_{3}],
	(\gamma_{4}),
	[\gamma_{5}],
	(\gamma_{6});
	[\gamma_{7}],
	(\gamma_{8}),
	[\gamma_{9}],
	(\gamma_{10}),\dots
\end{equation*}
By construction, we have $\gamma_{0}\gamma_{1}\gamma_{2} \AA \gamma_3 \AA \gamma_4 \AA \gamma_5 \AA \gamma_6$.
At step $j = 3$, we check whether $\tau_3 < P_{\nu_s} (\gamma_6, \gamma_7, \gamma_8)$.
Assume that this time, that condition fails. 
Then we have to backtrack to the previous candidate pivotal time: $\gamma^{p^3}_{2 m_3 - 1}$, which is simply $\gamma_5$.
In other words, we look at the definition of $m_{j+1}$~\eqref{eq:def-m+1} with $k = 3$ and $l^j_k = 5$. 
We check whether $\tau_2 < P'_{\nu_s}(\gamma_4, \gamma_5, \gamma_6,\gamma_6\gamma_7\gamma_8)$, which is a proxy for $\gamma_5 \AA \gamma_6\gamma_7\gamma_8$ but with a controlled conditional probability constant and equal to $\frac{1-3\rho}{1-2\rho}$.
Indeed, we already know that $\tau_2 < P_{\nu_s}(\gamma_4, \gamma_5, \gamma_6)$ from the previous step of the construction.
Assume that this holds.
By \eqref{item:backtrack}, this means that $m_4 = 1$ and the only candidate pivotal time left is $[\gamma_3]$.
The newly merged sequence becomes:
\begin{equation*}
	(\gamma_{0},\gamma_{1},\gamma_{2}),
	[\gamma_{3}],
	(\gamma_{4},\gamma_{5},\gamma_{6},\gamma_{7},\gamma_{8});
	[\gamma_{9}],
	(\gamma_{10}),\dots
\end{equation*}
Note that $\gamma_3 \AA \gamma_4 \AA \gamma_5 \AA \gamma_6\gamma_7\gamma_8$ and $\gamma_4\in S$, so $\gamma_3 \widetilde\AA^S(\gamma_4)$ and $\gamma_5 \in S$ so $\gamma_3 \widetilde\AA^S(\gamma_4,\gamma_5,\gamma_6,\gamma_7,\gamma_8)$. 
Moreover, we still have $\gamma_0\gamma_1\gamma_3 \AA \gamma_3$.

At the next step ($j=4$), we check whether $\tau_4 < P_{\nu_s} (\gamma_{4}\gamma_{5}\gamma_{6}\gamma_{7}\gamma_{8},\gamma_9,\gamma_{10})$, which is a proxy for $\gamma_{4}\gamma_{5}\gamma_{6}\gamma_{7}\gamma_{8}\AA \gamma_9 \AA \gamma_{10}$.
Assume that this holds.
Then $m_5 = 2$ the newly merged sequence is
\begin{equation*}
	(\gamma_{0},\gamma_{1},\gamma_{2}),
	[\gamma_{3}],
	(\gamma_{4},\gamma_{5},\gamma_{6},\gamma_{7},\gamma_{8}),
	[\gamma_{9}],
	(\gamma_{10});\dots
\end{equation*}

It feels a bit frustrating to lose two pivotal times instead of only one, because we had the alignment $\gamma_4 \AA \gamma_5 \AA \gamma_6\gamma_7\gamma_8$ so why not keep $[\gamma_5]$ as a pivotal time. 
The issue with that is that we do not have any control over $\gamma_6\gamma_7\gamma_8$.
For example, it may be the identity, which is aligned with everyone. 
In that case the alignment condition $\gamma_5 \AA \gamma_6\gamma_7\gamma_8 \AA \gamma_9$ is trivial and does not tell us anything about the product $\gamma_5\gamma_6\gamma_7\gamma_8\gamma_9$, which may again be the identity.
Therefore, we really need to discard this pivotal time in order to be able to use Proposition \ref{prop:pivot-is-aligned}.

The other issue is that even if we somehow get rid of that problem, then $[\gamma_5]$ would not have the same probabilistic behaviour as the other pivotal times.
Indeed, knowing the construction up to step $4$, it satisfies $3$ alignment conditions.
Note also that the first block $(\gamma_0,\gamma_1,\gamma_2)$ has a particular status because we do not have any informations on its structure and only know that its product is aligned with $\gamma_3$.

We need the index $l^j_k$ in \eqref{eq:def-m+1} because if we ever backtrack to $[\gamma_3] = \widetilde\gamma^{p_5}_1$ for example, then the alignment condition on $\gamma_3$ is not with the merged word $(\gamma_{4},\gamma_{5},\gamma_{6},\gamma_{7},\gamma_{8})$ but only with the first sub-word, namely $(\gamma_4)$, so we need to keep track of its index, this is the role of $2l^5_0 + 2 = 4$ and $l^5_0 = 1$ is indeed the last step at which we had $m_j = m_1 =1$.

Let us recap in the next remark basic properties of the algorithm that follow
readily from its definition.

\begin{Rem}\label{rem:descri-piv}
	Let $(\gamma_n)\in\Gamma^{\NN}$, let $(w_k)\in\NN_{\ge 1}^\NN$ and let $(\tau_k)\in[0,1]^\NN$ be non-random sequences.
	Note that, the family of lengths of the pivotal blocks $\left(p^j_k\right)$ in the sense of Definition \ref{def:pivot} is unique and we can construct it by induction.
	Moreover that the map $((\gamma_n),(w_k),(\tau_k))\mapsto\left(p^j_k\right)$ is measurable. 
	Let $\left(p^j_k\right)$ be the family of lengths of the pivotal blocks, and let $(m_j)_{j\in\NN}$ and $(l^j_k)_{j\in\NN, 0\le k \le m_j}$ be as in Definition \ref{def:pivot}.
	By induction, we can easily check that the following facts hold:
	\begin{enumerate}
		\item For all $j\le j'\in\NN$, we have $l^j_{j'} = j$.
		\item For all $j \in\NN$, and for all $0\le k < m_j$, we have $\gamma^{p^{l^j_k}}_{2k} \AA \gamma^{p^{l^j_k}}_{2k +1}\AA \gamma^{p^{l^j_k}}_{2k + 2}$ (because $m_{l^j_k+1} > m_{l^j_k}$ by definition of $l^j_k$).
		\item For all $j \in\NN$, and for all $0\le k < m_j$, we have $\widetilde\gamma^{p^j}_{k'}= \widetilde\gamma^{p^{j'}}_{k'}$ for all $l^j_k \le j' \le j$ and all $0 \le k' \le 2k +1$.
		\item For all $j \in\NN$, we have $\gamma^{p_j}_{2m_j+2} = \gamma^w_{2j+2}$. Hence, for all $k < m_j$, we have $\gamma^{p^{l^j_k}}_{2k + 2} = \gamma^w_{2l_{k}^j+2}$.
		\item For all $j \in\NN$, and for all $0\le k < m_j$, we have $\gamma^{p^j}_{2k}\AA \gamma^{p^j}_{2k+1}\AA \gamma^w_{2l_{k}^j+2}$.
		\item The family $(l^j_k)_{j,k}$ is determined by the data of the sequence $(m_j)_j$.
		\item The family $(p^j_k)_{j,k}$ is determined by the data of the sequences $(m_j)_j$ and $(w_k)_{k}$.
	\end{enumerate}
\end{Rem}

\begin{Lem}
	Let $(\gamma_n)\in\Gamma^{\NN}$, let $(w_k)\in\NN_{\ge 1}^\NN$ and let $(\tau_k)\in[0,1]^\NN$ be non-random sequences. 
	Let $\AA$ be a binary relation on $\Gamma$ and let $S \subset \Gamma$.
	Assume that for all $k \in\NN$, there exist three sub-words $(\tilde{g}_0,\tilde{g}_1,\tilde{g}_2)$ such that $\widetilde\gamma_{2k+2}^w = \tilde{g}_0 \odot \tilde{g}_1 \odot \tilde{g}_2$ and $\Pi(\tilde{g}_0)\AA \Pi(\tilde{g}_1) \AA \Pi(\tilde{g}_2)$ and $\Pi(\tilde{g}_1) \in S$ and $\gamma^w_{2k+1} \in S$.
	Let $\left(p^j_k\right)_{j,k}$ be the family of lengths of the pivotal blocks in the sense of Definition \ref{def:pivot} and let $(m_j)_j$ be as in Definition \ref{def:pivot}.
	For all $j \in\NN$, and for all $0 \le k < m_j$, we have $\gamma^{p^j}_{2k+1}\widetilde{\AA}^S\widetilde\gamma^{p^j}_{2k+2}$.
\end{Lem}

\begin{proof}
	We prove the claim by induction. Assume that for all $j' \le j$, and for all $0 \le k < m_{j'}$, we have $\gamma^{p^{j'}}_{2k+1}\widetilde{\AA}^S\widetilde\gamma^{p^{j'}}_{2k+2}$.
	If $m_{j+1} = m_j +1$, then $\tau_j < P_{\nu_s}\left(\gamma^{p^j}_{2m_j},\gamma^{p^j}_{2m_j + 1},\gamma^{p^j}_{2m_j + 2}\right)$. 
	Therefore $\gamma^{p^j}_{2m_j} \AA \gamma^{p^j}_{2m_j + 1} \AA \gamma^{p^j}_{2m_j + 2}$, so we have $\gamma^{p^j}_{2m_j + 1} \widetilde{\AA} \widetilde\gamma^{p^j}_{2m_j + 2}$. 
	For smaller values of $k$, we use the induction hypothesis.
	If $0 < m_{j+1} < m_j$ we have $\gamma^{p^j}_{2m_{j+1}} \AA \gamma^{p^j}_{2m_{j+1}+1} \AA \gamma^{p^j}_{2k+2}\cdots\gamma^{p^j}_{2m_j+2}$ and by induction hypothesis, we have $\gamma^{p^j}_{2m_{j+1} - 1} \widetilde{\AA}^S \gamma^{p^j}_{2m_{j+1}+1}$ and $\gamma^{p^{j+1}}_{2m_{j+1} - 1} = \gamma^{p^j}_{2m_{j+1} - 1}$ and $\widetilde\gamma^{p^{j+1}}_{2m_{j+1}} = \widetilde\gamma^{p^j}_{2m_{j+1}} \odot \widetilde\gamma^{p^j}_{2m_{j+1}+1} \odot \widetilde\gamma^{p^j}_{2k+2}\cdots\gamma^{p^j}_{2m_j+2}$. 
	Moreover $\gamma^{p^j}_{2m_{j+1}+1}$ is equal to one of the $\gamma^w_{2k+1}$ for some $k \in\NN$, therefore $\gamma^{p^j}_{2m_{j+1}+1} \in S$ by assumption so $\gamma^{p^j}_{2m_{j+1} - 1} \widetilde{\AA}^S \widetilde\gamma^{p^{j+1}}_{2m_{j+1}}$ by definition of $\widetilde{\AA}^S$.	
\end{proof}

\begin{Lem}\label{lem:m-to-infty}
	Let $\Gamma$ be a measurable semi-group endowed with a measurable relation $\AA$. 
	Let $\nu_s$ be a probability distribution on $\Gamma$ that is $\rho$-Schottky for $\AA$ and let $S := \mathbf{supp}(\nu_s)$.
	Let $\left(\widetilde\gamma^w_{n}\right)_{n\in\NN}\in\widetilde\Gamma^\NN$, and let $(\tau_j)_{j\in\NN}\sim\mathcal{U}_{[0,1]}^{\otimes \NN}$ be independent random sequences defined on the same probability space. 
	Assume that $\left(w_{2k+1}\right)_{k\in\NN}$ is almost surely equal to a non-random constant.
	Assume also that the sequences $\left(\widetilde\gamma^w_{2k}\right)_{k\in\NN} $ and $\left(\widetilde\gamma^w_{2k+1}\right)_{k\in\NN}$ are independent.
	Assume that $(\gamma^w_{2k+1})_{k\in\NN} \sim \nu_s^{\otimes\NN}$.
	Let $(p^j_k)$ be the random family of lengths of the pivotal blocks associated to $\widetilde\gamma^w$ with weights $\tau$ and let $(l^j_k)$ and $(m_j)$ be as in Definition \ref{def:pivot}. Then for all $j \in\NN$, and for all $0 \le k \le m_j-1$, we have:
	\begin{align}
		\PP\left(m_{j+1}=m_j +1\,\middle|\,\left(\widetilde\gamma^w_{2k}\right)_{k\in\NN}, (m_{j'})_{j'\le j}\right) & = 1-2\rho, \label{eq:m-plus-un} \\
		\PP\left(m_{j+1} < m_j-k\,\middle|\,\left(\widetilde\gamma^w_{2k}\right)_{k\in\NN}, (m_{j'})_{j'\le j}\right) & = 2\rho\left(\frac{\rho}{1-2\rho}\right)^{k}. \label{eq:m-moins-k}
	\end{align}
\end{Lem}

\begin{proof}
	First note that for all $f,g,h,h'$, we have $0 \le P'_{\nu_s}(f,g,h, h') \le P_{\nu_s}(f,g,h) \le 1$. Moreover, given a random $\gamma \sim \nu_s$ and given $g, h, h' \in\Gamma$ non-random or independent of $\gamma$, we have $\EE(P_{\nu_s}(f,\gamma,h)) =1-2\rho$ and $\EE(P'_{\nu_s}(f,\gamma,h,h')) =1-3\rho$. 

	Let $j \in \NN$ and let $k < m_j$. We have:
	\begin{equation}
		\tau_{l^j_k}< P_{\nu_s}\left(\gamma^{p^j}_{2k}, \gamma^{p^j}_{2k+1}, \gamma^w_{2l_{k}^j+2}\right).
	\end{equation}
	Indeed, by definition of $l^j_k$, we have $m_{l^j_k+1} > m_{l^j_k}$. Therefore:
	\begin{equation*}
		\tau_{l^j_k} < P_{\nu_s}\left(\gamma^{p^{l^j_k}}_{2k}, \gamma^{p^{l^j_k}}_{2k +1}, \gamma^{p^{l^j_k}}_{2k + 2}\right).
	\end{equation*}
	Moreover, for all $l^j_k < j' \le j$, we have $m_{j'} > k$ so $\left(\widetilde\gamma^{p^j}_{k'}\right)_{k' \le 2k+1}=\left(\widetilde\gamma^{p^{l_k^j}}_{k'}\right)_{k' \le 2k+1}$. Hence, we have $\gamma^{p^{l^j_k}}_{2k} = \gamma^{p^{j}}_{2k}$ and $\gamma^{p^{l^j_k}}_{2k +1} = \gamma^{p^{j}}_{2k +1}$. Moreover, we have $\gamma^{p^{l^j_k}}_{2k + i} = \gamma^{w}_{2l^j_k + i}$ for all $ i \ge 1$. 
	
	Now Let $\gamma$, $w$, $\tau$ be random sequences as in Lemma \ref{lem:m-to-infty}. Given $j \in\NN$, let $\mathcal{P}_j$ be the $\sigma$-algebra generated by $\left(p^{j'}_k\right)_{k\in\NN, j' \le j}$ and $\left(\widetilde{\gamma}^{p^j}_{2k}\right)_{k\in\NN}$. Then $(\mathcal{P}_j)_{j\in\NN}$ is a filtration. 
	
	Let $i, j \in\NN$, we claim that the conditional distribution of $\gamma^{p^j}_{2m_j + 2 i +1}$ relatively to $\mathcal{P}_j$ is almost surely $\nu_s$. 
	We prove the claim by induction. 
	For $j = 0$, we assumed that $\left(\gamma^w_{2k+1}\right)_{k\in\NN} \sim \nu_s^{\otimes\NN}$, and that $\left(\gamma^w_{2k+1}\right)_{k\in\NN}$ is independent of $(\widetilde{\gamma}^w_{2k})_{k\in\NN}$, now since $(w_{2k+1})$ is non-random, the sequence $(\widetilde{\gamma}^w_{2k})_{k\in\NN}$ determines $(p^0_k) = (w_k)$, hence it generates $\mathcal{P}_0$ which proves the claim. 
	Given $j \in\NN$, we have $(\gamma^{p^{j+1}}_{2m_{j+1} + 2 i + 1})_i = (\gamma^{p^j}_{2m_j + 2 i + 3})_{i\in\NN}$ and the construction of $(p^{j+1}_k)_k$ from $(p^{j}_k)_k$ does not depend on $(\gamma^{p^j}_{2m_j + 2 i + 3})_{i\in\NN}$. Therefore, by induction on $j$, we have $(\gamma^{p^{j}}_{2m_{j} + 2 i + 1})_{i\in\NN} \sim \nu_s^{\otimes\NN}$. By the same argument, we show by induction on $j$ that conditionally to $\mathcal{P}_j$, we have:
	\begin{equation}
		\forall j\in\NN,\;\left((\gamma^{p^{j}}_{2m_{j} + 2 i + 1},\tau_{2j+2i+1})\right)_{i\in\NN} \sim \left(\nu_s \otimes \mathcal{U}_{[0,1]}\right)^{\otimes\NN}.
	\end{equation}
	Taking $i = 0$, we have $\PP\left(\tau_{2j+1}< P_{\nu_s}\left(\gamma^{p^j}_{2m_j},\gamma^{p^j}_{2m_j + 1},\gamma^{p^j}_{2m_j + 2}\right)\,\middle|\,\mathcal{P}_j\right) = 1 - 2\rho$. 
	Hence we have \eqref{eq:m-plus-un}. 
	
	Now let $j \in \NN$ be fixed, and let $q : \Omega \to \{0, \dots, m_{j} - 1\}$ be a $\mathcal{P}_j$-measurable random variable.
	Then, we have almost surely $\gamma^{p^j}_{2 q} \AA \gamma^{p^j}_{2 q + 1} \AA \gamma^{w}_{2 l^j_q + 2}$ the conditional distribution of $\gamma^{p^j}_{2 q + 1}$ relatively to $\mathcal{P}_j$ is exactly the rescaled restriction of $\nu_s$ to $\left\{\gamma\in\Gamma\,\middle|\,\gamma^{p^j}_{2 q}\AA\gamma\AA\gamma^{w}_{2 l^j_q + 2}\right\}$.
	It means that for all $A : \Omega \to \mathcal{A}_\Gamma$ which is $\mathcal{P}_j$-measurable, we have:
	\begin{equation}\label{eq:cond-dist}
		\PP\left(\gamma^{p^j}_{2 q + 1}\in A \,\middle|\, \mathcal{P}_j\right) = \frac{\nu_s\left(A \cap \left\{\gamma\in\Gamma\,\middle|\,\gamma^{p^j}_{2 q}\AA\gamma\AA\gamma^{w}_{2 l^j_q + 2}\right\}\right)}{\nu_s\left\{\gamma\in\Gamma\,\middle|\,\gamma^{p^j}_{2 q}\AA\gamma\AA\gamma^{w}_{2 l^j_q + 2}\right\}}.
	\end{equation}
	Let $h' : \Omega \to \Gamma$, be a random variable which is independent of  $\gamma^{p^j}_{2 q + 1}$ relatively to\footnote{We say that two events are independent relatively to a $\sigma$-algebra if the conditional probability of their intersection is almost surely equal to the product of their conditional probability. We say that two random variables are relatively independent if their level sets are. By Bayes formula, it implies that the conditional distribution of one with respect to the other and said $\sigma$-algebra is almost surely equal to its conditional distribution with respect to the $\sigma$-algebra alone.} $\mathcal{P}_j$.
	Then by \eqref{eq:cond-dist}, and because $\mathds{1}_{\gamma^{p^j}_{2 q}\AA \gamma^{p^j}_{2 q + 1} \AA\gamma^{w}_{2 l^j_q + 2}} = 1$ almost surely, we have:
	\begin{align*}
		\EE\left(P'_{\nu_s}\left(\gamma^{p^j}_{2 q},\gamma^{p^j}_{2 q + 1},\gamma^{w}_{2 l^j_q + 2},h'\right)\,\middle|\,\mathcal{P}_j, h'\right) 
		& = \frac{(1-3\rho) \mathds{1}_{\gamma^{p^j}_{2 q + 1}\AA h'}}{\nu_s\left\{\gamma\in\Gamma\,\middle|\,\gamma^{p^j}_{2 q}\AA\gamma\AA\gamma^{w}_{2 l^j_q + 2} \text{ and }\gamma \AA h'\right\}} \\
		& = \frac{(1-3\rho) \PP\left(\gamma^{p^j}_{2 q + 1} \AA h'\,\middle|\,\mathcal{P}_j, h'\right)}{\nu_s\left\{\gamma\in\Gamma\,\middle|\,\gamma^{p^j}_{2 q}\AA\gamma\AA\gamma^{w}_{2 l^j_q + 2} \text{ and }\gamma \AA h'\right\}}\\ 
		& = \frac{1-3\rho}{\nu_s\left\{\gamma\in\Gamma\,\middle|\,\gamma^{p^j}_{2 q}\AA\gamma\AA\gamma^{w}_{2 l^j_q + 2}\right\}} \\ 
		& =  \frac{1-3\rho}{1-2\rho}P_{\nu_s}\left(\gamma^{p^j}_{2 q},\gamma^{p^j}_{2 q + 1},\gamma^{w}_{2 l^j_q + 2}\right).
	\end{align*} 
	Note that even though $\gamma^{p^j}_{2 q + 1}$ is not $\mathcal{P}_j$ measurable, its only role in the computation of $P_{\nu_s}\left(\gamma^{p^j}_{2 q},\gamma^{p^j}_{2 q + 1},\gamma^{w}_{2 l^j_q + 2}\right)$ is trough a $0-1$ indicator function that we know to be equal to $1$.
	So $P_{\nu_s}\left(\gamma^{p^j}_{2 q},\gamma^{p^j}_{2 q + 1},\gamma^{w}_{2 l^j_q + 2}\right)$ is indeed a $\mathcal{P}_j$ measurable quantity.
	Moreover, the conditional distribution of $\tau_{l^j_q}$ with respect to $\mathcal{P}_j$ and $\gamma^{p^j}_{2 q + 1}$ is almost surely uniform in $\left[0 , P_{\nu_s}\left(\gamma^{p^j}_{2 q},\gamma^{p^j}_{2 q + 1},\gamma^{w}_{2 l^j_q + 2}\right)\right]$. Indeed, for all $l \le j$ and for all constant $k \le l$, the conditions to have $l_k^j  = l$ are:
	\begin{enumerate}
		\item We have $m_l = k$. Note that this event only depends on $(\gamma^w_{k'})_{0\le k' \le 2l}$ and $(\tau_{l'})_{0 \le l' < l}$, hence it is independent of $(\widetilde\gamma^w_{2l+1}, \tau_l)$.
		\item For all $l < j' \le j$, we have $m_{j'} > k$. Note that this event only depends on $(\gamma^w_{k'})_{2l+2 \le k' \le 2j+2}$ and $(\tau_{l'})_{l < l' \le j}$, hence it is independent of $(\widetilde\gamma^w_{2l+1}, \tau_l)$.
		\item We have $\tau_l < P_{\nu_s}\left(\gamma^{p^{l}}_{2 m_l},\gamma^{w}_{2 l + 1},\gamma^{w}_{2 l + 2}\right)$. 
		In particular $0 < P_{\nu_s}\left(\gamma^{p^{l}}_{2 m_l},\gamma^{w}_{2 l + 1},\gamma^{w}_{2 l + 2}\right)$, therefore, we have $\gamma^{p^j}_{2 k} \AA \gamma^{p^j}_{2 k + 1} \AA \gamma^{w}_{2 l^j_k + 2}$. 
		Moreover $\tau_l$ and $\widetilde\gamma^w_{2l+1}$ are independent so the distribution of $\gamma^{w}_{2 l + 1}$ knowing $\tau_l < P_{\nu_s}\left(\gamma^{p^{l}}_{2 m_l},\gamma^{w}_{2 l + 1},\gamma^{w}_{2 l + 2}\right)$ is the restriction of $\nu_s$ to $\left\{\gamma\in\Gamma\,\middle|\,\gamma^{p^j}_{2 k}\AA\gamma\AA\gamma^{w}_{2 l^j_k + 2}\right\}$ and $\tau_l$ and $\widetilde\gamma^w_{2l+1}$ are still independent.
	\end{enumerate}
	The above argument implies moreover that the family $\left((\tau_{l^j_k},\gamma^{p^j}_{2k+1})\right)_{0 \le k < m_j}$ is independent with respect to $\mathcal{P}_j$ \ie there is a $\mathcal{P}_j$ measurable family of distribution $(\eta_k, \kappa_k)_{0 \le k < m_j} \in \widetilde{\mathrm{Prob}(\RR) \times \mathrm{Prob}(\RR)}$ such that:
	\begin{equation*}
		\left((\tau_{l^j_k},\gamma^{p^j}_{2k+1})\right)_{0 \le k < m_j} \sim \int_{\Omega}\bigotimes_{k = 0}^{m_j-1} (\eta_k\otimes\kappa_k)d\PP.
	\end{equation*}
	Therefore, for all $j \in\NN$ and for all $k \le m_j$, we have:
	\begin{equation*}
		\PP\left(\tau_{l^j_k}< P'_{\nu_s}\left(\gamma^{p^j}_{2k}, \gamma^{p^j}_{2k+1}, \gamma^w_{2l_{k}^j+2},\gamma^{p_j}_{2k+2}\cdots\gamma^{p_j}_{2m_j+2}\right)\,\middle|\,\mathcal{P}_j,\left((\tau_{l^j_{l'}},\gamma^{p^j}_{2k'+1})\right)_{k < k' < m_j}\right) = \frac{1-3\rho}{1-2\rho}.
	\end{equation*}
	By induction on $k$, we have \eqref{eq:m-moins-k}.
\end{proof}

\begin{proof}[Proof of Theorem \ref{th:pivot-extract}]
	Let $\gamma$ be a measurable semi-group, let $\AA$ be a measurable binary relation $\Gamma$ and let $S \subset \Gamma$ be measurable. Let $0< \alpha < 1$, let $0 < \rho < \frac{1}{5}$ nd let $m \in\NN$. Let $\tilde\nu_s$ be a probability distribution on $\Gamma^m$ and let $\nu_s := \Pi_*\nu_s$. Assume that $\nu_s$ is $\rho$-Schottky for $\AA$ and supported on $S$. Let $\tilde\kappa$ be as in Lemma \ref{lem:ping-pong}, let $(\Omega,\PP)$ be a probability space and let $\left(\left(\widetilde\gamma^w_n\right)_n,(\tau_k)_k\right)\sim\left(\tilde\kappa\otimes\tilde\nu_s\right)^{\otimes\NN}\otimes\mathcal{U}_{[0,1]}^{\otimes\NN}$. Let $\left(p^j_k\right)_{j,k}$ be the random family of lengths of the pivotal blocks associated to $\widetilde\gamma^w$ with weights $\tau$ and let $(m_j)$ and $\left(l^k_j\right)$ be as in Definition \ref{def:pivot}
	
	Let $j \in\NN$. By Lemma \ref{lem:m-to-infty}, we have:
	\begin{align*}
		\EE\left(m_{j+1}\,\middle|\,(m_{j'})_{j'\le j}\right) & = m_j + (1-2\rho) - 2\rho\sum_{k=0}^{m_j-1}\left(\frac{\rho}{1-2\rho}\right) \\
		& = m_j + (1 - 2\rho) - 2 \rho \frac{1-2\rho}{1-3\rho} + 2 \rho \frac{1-2\rho}{1-3\rho}\left(\frac{\rho}{1-2\rho}\right)^{m_j}\\
		& = m_j + (1 - 2\rho) \frac{1-5\rho}{1-3\rho} + 2 \rho \frac{1-2\rho}{1-3\rho}\left(\frac{\rho}{1-2\rho}\right)^{m_j}
	\end{align*}
	Note that $(1 - 2\rho) \frac{1-5\rho}{1-3\rho} > 0$. By Lemma \ref{lem:proba:ldev} applied to $(m_j)_{j\in\NN}$, there are constants $C, \beta >0$ such that $\PP(m_j\le 0)\le C\exp(-\beta j)$ for all $j \in\NN$. Hence $l_0$ is almost surely finite and has finite exponential moment because $\PP(l_0 = j) \le {C}\exp(-\beta j)$ for all $j \in\NN$. Now let $0 \le q \le l \in\NN$ be fixed and let $j \ge l$. We claim that:
	\begin{equation}\label{eq:un}
		\EE\left(m_{j+1}\,\middle|\,(m_{j'})_{j'\le j}, l_q=l\right) \ge \EE\left(m_{j+1}\,\middle|\,(m_{j'})_{j'\le j}\right).
	\end{equation}
	Indeed, if we assume the values of $(m_{j'})_{j'\le j}$ to be fixed and that $l_q^j =l$. Then $l_q =l$ if and only if there is no $j' >j$ such that $m_{j'} = 0$. We claim that:
	\begin{equation}\label{eq:deux}
		\forall k \le k',\;\PP\left(l_q = l_q^j\,\middle|\,(m_{j'})_{j'\le j},m_{j+1} = k\right) \le \PP\left(l_q = l_q^j\,\middle|\,(m_{j'})_{j'\le j},m_{j+1} = k'\right).
	\end{equation}
	Note that \eqref{eq:deux} implies that $\PP\left(m_{j+1} \ge k\,\middle|\,(m_{j'})_{j'\le j}, l_q=l\right) \ge \PP\left(m_{j+1} \ge k\,\middle|\,(m_{j'})_{j'\le j}\right)$ almost surely and for all $k$, hence we have \eqref{eq:un}. 
	Now we prove \eqref{eq:deux}. Let $\eta$ be the probability measure on $\ZZ$ such that $\eta\{1\} = 1-2\rho$ and $\eta\{-k\} = 2\rho\frac{1-3\rho}{1-2\rho}\left(\frac{\rho}{1-2\rho}\right)^{k-1}$ for all $k \ge 1$. Let $(r_j)\sim \eta^{\otimes\NN}$. 
	Let $(r_j)\sim \eta^{\otimes\NN}$ be a random sequence defined on a probability space $(\Omega',\PP')$.
	Define $(m'_j)$ by induction taking $m_0 = 0$ and $m'_{j+1} := \max\{0,m'_j+r_j\}$ for all $j$. 
	With that construction, all the formerly defined random variables are defined on the coupling of $(\Omega',\PP')$ and $(\Omega, \PP)$ relatively to $m' = m$.
	From now on we work on that coupling.
	Then for all $q \le j \in\NN$, we have $l_q = l_q^j$ if and only if $\sum_{k = j}^{j'-1} r_{k} \ge 1+q-m_{j}$ for all $j' > j$. Hence, for all $0< k \le j \in\NN$, we have $l_q = l_q^j$ and $m_{j+1} = k$ if and only if $\sum_{k = j+1}^{j'-1} r_{k} \ge 1-k$ and $m_{j+1} = k$. Moreover, the events $\left( \forall j' > j+1,\sum_{k = j+1}^{j'-1} r_{k} \ge 1+q-k\right)$ and $m_{j+1} = k$ are independent so:
	\begin{equation*}
		\PP\left(l_q = l_q^j\,\middle|\,(m_{j'})_{j'\le j},m_{j+1} = k\right) = \PP\left( \forall j' > j+1,\sum_{k' = j+1}^{j'-1} r_{k'} \ge 1+q-k\right).
	\end{equation*}
	This makes \eqref{eq:deux} obvious. 
	Then by lemma \ref{lem:proba:ldev} applied to $(m_j)$, there exist constants $C, \beta > 0$ such that $\PP(l_{q+1}-l_{q} = j) \le C \exp(-\beta j)$ for all $j$ and for all $q$. Moreover, the distribution of $l_{q+1}-l_q$ does not depend on $q$ and the family $(l_{q+1}-l_q)_{q\in\NN}$ is i.i.d. and independent of $l_0$. Now let $v_ 0 := l_0$ and for all $q \in\NN$, let $v_{2q+2} := 2(l_{q+1}-l_q)-1$ and let $v_{2q+1} = 1$. Let $p = w^v$. Then note that $p^j \underset{j \to +\infty}{\longrightarrow}p$ almost surely for the simple convergence topology. By Lemma \ref{lem:m-to-infty}, the random sequence $(v_q)$ is independent of $(w_k)$ so the sequences $(p_{2k+1})$ and $(p_{2k+2})$ are i.i.d. and independent of each other and of $p_0$. Moreover, by Lemma \ref{lem:sumexp}, each $p_k$ has finite exponential moment. Let $\tilde\mu$ be the distribution of $\widetilde\gamma^p$. We have just proven \eqref{eq:mmt-exp}.
	
	Let $k \in\NN$, we want to show \eqref{eq:tjr-scho} in Theorem \ref{th:pivot}, which states that the conditional distribution of $\widetilde\gamma^p_{2k+1}$ relatively to $\left(\widetilde\gamma^p_{k'}\right)_{k'\neq 2k+1}$ is bounded above by $\frac{\tilde\nu_s}{1-\alpha}$. 
	Let $j \in\NN$. Saying that $l_k= j$ is equivalent to saying that $\tau_j< P_{\nu_s}(\gamma^p_{2k},\gamma^w_{2j+1},\gamma^w_{2j+2})$, that $m_j = k$ and that $m_{j'} > k$ for all $j' > j$. 
	Once we assume that $\tau_j < P_{\nu_s}(\gamma^p_{2k},\gamma^w_{2j+1},\gamma^w_{2j+2})$, the conditions $m_j = k$ and  $m_{j'} > k$ for all $j' > j$ can be expressed in terms of $\left(\widetilde\gamma^w_{k'}\right)_{k'\neq 2j+1}$. Moreover, once we assume that $l_k =j$, the random sequence $(\widetilde\gamma^p_{k'})_{k'\neq 2k+1}$ is the image of the random sequence $\left(\widetilde\gamma^w_{k'}\right)_{k'\neq 2j+1}$ by a measurable function (which is defined on the set $l_k = j$).
	Hence, the distribution of $\widetilde\gamma^p_{2k+1}$ knowing $l_k = j$ and $(w_k)_{k\in\NN}$ and $(\widetilde\gamma^p_{k'})_{k'\neq 2k+1}$ is $\frac{\mathds{1}_{\Pi^{-1}(A')}}{\nu_s(A')}\nu_s$ for $A' : = \{g\in\Gamma\,|\,\gamma^p_{2k}\AA g \AA \gamma^w_{2j+2}\}$. By the Schottky property, we have $\nu_s(A') \ge 1-2\rho$. This proves \eqref{eq:tjr-scho}.
	
	Let $n\in \NN$. Let $q := \max\{k\in\NN\,|\,\overline{w}_{k} \le n\}$ and let $r := n-\overline{w}_{q}$. We claim that the conditional distribution of $\gamma_n$ knowing $(p_{k}^j)_{k,j}$ and knowing that $q$ is even is $\left(\chi_{r}^{w_{q}}\right)_*\widetilde\kappa$. We prove by induction on $j'\in\NN$ that the conditional distribution of $\gamma_n$ knowing $(p_{k}^j)_{k,j \le j'}$ and knowing that $q$ is even is $\left(\chi_{r}^{w_{q}}\right)_*\widetilde\kappa$. For $j' = 0$, this comes from the definition of the random sequence $\gamma^w$. For larger $j' \in\NN$, note that the construction of $(p_{k}^j)_{k,j \le j'+1}$ from $(p_{k}^j)_{k,j \le j'}$ only depends on events that are independent of $(\widetilde{\gamma}^w_{2k})_{k\in\NN}$ and therefore independent of $\gamma_n = \chi_r^{w_q}\left(\widetilde{\gamma}^w_q\right)$. Hence the conditional distribution of $\gamma_n$ knowing $(p_{k}^j)_{k,j \le j'+1}$ and knowing that $q$ is even is the conditional distribution of $\gamma_n$ knowing $(p_{k}^j)_{k,j \le j'}$ and knowing that $q$ is even.
	
	Now let $A \subset \Gamma \setminus \bigcup_{k=0}^{m-1}\chi_k^m(\mathbf{supp}(\tilde\nu_s))$. We have $\PP(\gamma_n \in A\,|\,q\in 2\NN) =0$ and knowing that $q$ is odd, the conditional probability of $(\gamma_n \in A)$ is $\left(\chi_{r}^{w_{q}}\right)_*\widetilde\kappa(A)$, which is bounded above by $\frac{\nu(A)}{1-\alpha}$ by \eqref{eq:densite-sqz}. This proves \eqref{eq:densite} and concludes the proof of Theorem \ref{th:pivot-extract}.
\end{proof}

\subsection{Facts about ping-pong sequences}

Given $(\Omega,\mathcal{A}_\Omega)$, and $(\Gamma,\mathcal{A}_\Gamma)$ two measurable spaces, and $\gamma: \Omega \to \Gamma$ a measurable map, we write $\langle \gamma \rangle_\sigma := \gamma^*\mathcal{A}_\Gamma \subset \mathcal{A}_\Omega$ for the $\sigma$-algebra generated by $\gamma$.

\begin{Def}[Ping-pong sequence]
	Let $\Gamma$ be a semi-group, let $\AA$ be a measurable binary relation on $\Gamma$ and let $\rho\in(0,1)$. Let $N\in\NN\cup\{+\infty\}$ and let $(\gamma_k)_{0\le k < N}$ be a random sequence. We say that $(\gamma_k)$ is $\rho$-ping-pong for $\AA$ if for all $k\in\NN$ such that $0\le 2k+1 < N$, the conditional distribution of $\gamma_{2k+1}$ relatively to $(\gamma_{k'})_{k' \neq 2k+1}$ is almost surely $\rho$-Schottky for $\AA$. Then we say that the distribution of $(\gamma_k)_{0 \le k \le N}$ is $\rho$-ping-pong for $\AA$.
\end{Def}

\begin{Lem}[Pivoting technique]\label{lem:r-piv}
	Let $\Gamma$ be a semi-group, let $\AA$ be a measurable binary relation on $\Gamma$ and let $\rho\in(0,1)$. Let $n\in \NN$ and let $\mu$ be a probability distribution on $\Gamma^{\{0,\dots, 2n\}}$ that is $\rho$-ping-pong for $\AA$. There exists a probability space $(\Omega, \PP)$ a random sequence $(\gamma_k)_{0\le k \le 2n}\sim \mu$ and and a random integer $r\sim\mathcal{G}_{\rho}$ such that $\gamma_{2n-2r-1}\AA(\gamma_{2n-2r}\cdots\gamma_{2n})$ or $n\le r$ and $r$ and $(\gamma_{2k})_{0\le k\le n}$ are independent.
\end{Lem}

\begin{proof}
	Let $\left((\gamma_k)_{0\le k \le 2n},(\tau_j)_{0\le j}\right) \sim \mu\otimes\mathcal{U}_{[0,1]}^{\otimes \NN}$. 
	Given $j \ge n$, we define $P_j := 1-\rho$. Given $0\le j < n$, we define:
	\begin{equation*}
		P_j := \frac{(1-\rho)\mathds{1}_\AA(\gamma_{2n-2j-1}, \gamma_{2n-2j}\cdots\gamma_{2n})}{\PP\left( \gamma_{2n-2j-1} \AA (\gamma_{2n-2j}\cdots\gamma_{2n}) \,\middle|\, (\gamma_{k})_{k\neq 2n-2j-1}\right)}.
	\end{equation*}
	Note that $0\le P_j \le 1$ almost surely because $(\gamma_k)$ is $\rho$-ping pong. Moreover $\EE(P_j\,|\,(\gamma_{k})_{k\neq 2n-2j-1}) = 1-\rho$ almost surely and $P_j$ is independent of $(\tau_{j'})_{j'\in\NN}$. Therefore, we have:
	\begin{equation*}
		\forall j\in\NN,\;\PP\left(\tau_j < P_j\,\middle|\,(\gamma_{k})_{k\neq 2n-2j-1},(\tau_{j'})_{j'\neq j}\right) = 1-\rho.
	\end{equation*}
	Moreover, for all $j' < j\in\NN$, the random variable $P_{j'}$ is measurable for $\langle(\gamma_k)_{k \ge 2n-2j'-2}\rangle$, hence it is measurable for $\langle(\gamma_k)_{k \neq 2n-2j-1}\rangle$. Therefore:
	\begin{equation}\label{eq:estim-r}
		\forall j\in\NN,\;\PP\left(\tau_j < P_j\,\middle|\,(\gamma_{2k})_{0\le k \le n}, \forall j' < j,\tau_{j'} \ge P_{j'}\right) = 1-\rho.
	\end{equation}	
	Let $r := \min\{j\in\NN\,|\,\tau_j < P_j\}$. Assume that $r < n$, then $P_r >0$ so $\gamma_{2n-2r-1}\AA(\gamma_{2n-2r}\cdots\gamma_{2n})$. Then by \eqref{eq:estim-r}, we have $\PP\left(r \ge j\,\middle|\,(\gamma_{2k})_{0\le k \le n}\right) = \rho^j$, almost surely and for all $j$. Hence $r \sim\mathcal{G}_{\rho}$ and $r$ is independent of $(\gamma_{2k})_{0\le k\le n}$.
\end{proof}

Given $N \in\NN$, given $\Gamma$ a semi-group, given $\gamma_0,\dots, \gamma_N$ a finite sequence in $\Gamma$ and given $0\le j < i\le N$, we write $\gamma_{i}\cdots\gamma_j$ for $\gamma_i\cdots\gamma_{N}\gamma_0\cdots\gamma_{j}$.

\begin{Lem}[Cyclical pivoting technique]\label{lem:c-piv}
	Let $\Gamma$ be a semi-group, let $\AA$ be a measurable binary relation on $\Gamma$ and let $\rho\in(0,1)$. Let $n\in \NN$ and let $\mu$ be a probability distribution on $\Gamma^{\{0,\dots, 2n\}}$ that is $\rho$-ping-pong for $\AA$. There exist a probability space $(\Omega, \PP)$, a random sequence $(\gamma_k)_{0\le k \le 2n}\sim \mu$ and an integer $c\sim\mathcal{G}_{2\rho}$ such that $\gamma_{2n-2c-1}\AA(\gamma_{2n-2c}\cdots\gamma_{2c})$ and $(\gamma_{2n-2c-2}\cdots\gamma_{2c})\AA \gamma_{2c+1}$ or $n \le 2c - 1$ and $c$ and $(\gamma_{2k})_{0\le k\le n}$ are independent. 
\end{Lem}

\begin{proof}
	Let $(\Omega, \PP) := \left(\Gamma^{2n+1}\times[0,1]^{\NN},\mu\otimes\mathcal{U}_{[0,1]}^{\otimes \NN}\right)$. Let $\left((\gamma_k)_{0\le k \le 2n},(\tau_j)_{0\le j}\right) \sim \mu\otimes\mathcal{U}_{[0,1]}^{\otimes \NN}$. Given $j \ge n/2$, we define $P_j := 1-2\rho$. Given $0\le j < n/2$, we define:
	\begin{equation*}
		P_j := \frac{(1-2\rho)\mathds{1}_\AA(\gamma_{2n-2j-1}, \gamma_{2n-2j}\cdots\gamma_{2j})\mathds{1}_\AA(\gamma_{2n-2j-2}\cdots\gamma_{2j},\gamma_{2j+1})}{\PP\left( \gamma_{2n-2j-1} \AA (\gamma_{2n-2j}\cdots\gamma_{2j})\cap (\gamma_{2n-2j-2}\cdots\gamma_{2j})\AA\gamma_{2j+1} \,\middle|\, (\gamma_{k})_{k\notin \{2n-2j-1, 2j +1\}}\right)}.
	\end{equation*}
	Note that $0\le P_j \le 1$ and $\EE(P_j\,|\,(\gamma_{k})_{k\notin \{2n-2j-1, 2j +1\}}) = 1-2\rho$. Let $c := \min\{j\in\NN\,|\,\tau_j < P_j\}$. Then $r\sim\mathcal{G}_{2\rho}$ and $c$ is independent of $(\gamma_{2k})_{0\le k\le n}$.
\end{proof}

\begin{Rem}\label{rem:c-piv}
	Let $\Gamma$ be a semi-group, let $\AA$ be a binary relation on $\Gamma$ and let $S\subset \Gamma$. 
	Let $(\tilde\gamma_k)_{0\le k \le 2n}\in\widetilde{\Gamma}^{2n+1}$ and let $\gamma_k :=\Pi(\tilde\gamma_k)$ for all $0\le k \le 2n$.
	Assume for the sake of the argument that the identity of $\Gamma$ is aligned with everyone and not in $S$.
	Assume that for all $0\le k\le n$, we have: $\gamma_{2k} \AA \gamma_{2k+1} \widetilde{\AA}^S \tilde{\gamma}_{2k+2}$. 
	
	In Lemma \ref{lem:c-piv}, we want to have $\gamma_{2n-2c-2}\cdots\gamma_{2c}\AA \gamma_{2c+1}$ instead of just $\gamma_{2n-2c}\cdots\gamma_{2c}\AA \gamma_{2c+1}$ because we have no control over the product $\gamma_{2n-2c}\cdots\gamma_{2c}$ (it may be the identity for example and the alignment would be meaningless).
	We however have control over the product $\gamma_{2n-2c-2}\cdots\gamma_{2c}$.
	Indeed, if $c < n/2$ is such that $\gamma_{2n-2c-1}\AA(\gamma_{2n-2c}\cdots\gamma_{2c})$, 
	then $\gamma_{2n-2c-3} \widetilde{\AA}^S \tilde\gamma_{2n-2c-2}\odot\cdots\odot\tilde\gamma_{2c}$.
	In concrete cases, we have shown in Proposition \ref{prop:pivot-is-aligned} that this implies a genuine alignment.
\end{Rem}

\subsection{Factorization of the pivotal extraction}

In this section, we prove Theorem \ref{th:pivot}. 
Given $X$ and $Y$ two measurable sets, we write $\chi_{\widetilde\Gamma}: Y \times X \to Y$ and $\chi_X: Y \times X \to X$ for the first and second coordinate projections.

\begin{Lem}[Factorization of the pivotal extraction]\label{lem:pivot-factor}
	Let $\Gamma$ be a measurable semi-group and let $S \subset\Gamma$ be measurable. 
	Let $M\in\NN$ and let $(L_i)_{1\le i \le M}$ and $(R_j)_{1\le j \le M}$ be two non-random families of disjoint measurable subsets of $\Gamma$. 
	Let $A \subset \{1,\dots, M\}^2$ and let $\AA := \bigsqcup_{(i,j)\in A} L_i \times R_j$.
	Let $\nu$ be a probability measure on $\Gamma$, let $0< \alpha < 1$, let $0< \rho <\frac{1}{5}$ and let $m \in\NN$. Let $\tilde\nu_s$ be a probability measure on $\Gamma^m$ such that $\alpha\nu_s\le \nu^{\otimes m}$ and let $\nu_s := \Pi_*\tilde\nu_s$. Assume that $\nu_s$ is supported on $S$ and $\rho$-Schottky for $\AA$. 
	Let $\tilde\mu$ be as in Theorem \ref{th:pivot} and let $(\widetilde\gamma^w) \sim \tilde\mu$. 
	Then there exists a Markov chain $(x_n)$ on $X \subset \{0, \dots, 2M\}$, with $x_0 = 0$, and a family $\left(\tilde\nu'_x\right)_{x\in X}\in\mathrm{Prob}\left( \widetilde\Gamma \times X \right)^X$ such that:
	\begin{enumerate}
		\item For all $n \in\NN$, the pair $(\widetilde\gamma^w_n, x_{n+1})$ has distribution law $\tilde\nu'_{x_n}$ conditionally to $(\widetilde\gamma^w_k)_{k < n}$ and $(x_k)_{k \le n}$.
		\item For all $k \in\NN$, one has $x_{2k+1} \in \{1, \dots, M\}$ and $x_{2k+2} \in \{M+1,\dots, 2M\}$.\label{item:2}
		\item For all $i\in\{1,\dots, M\}\cap X$ and all $j\in\{M+1,\dots, 2M\}\cap X$, one has $\tilde\nu'_i\left(\widetilde\Gamma\times\{j\}\right) > 0$. \label{item:3}
		\item For all $i\in\{1,\dots, M\}\cap X$ and all $j\in\{M+1,\dots, 2M\}\cap X$, the distribution:
		\begin{equation*}
			\nu_{i,j} := \Pi_*(\chi_{\widetilde{\Gamma}})_{*}\left(\frac{\mathds{1}_{(\chi_X=j)}\tilde{\nu}'_i}{(\chi_{X})_* \tilde\nu'_i\{j\}}\right)
		\end{equation*}
		is $\frac{\rho}{1-2\rho}$-Schottky.
	\end{enumerate}
\end{Lem}

Let $(x_n), (\widetilde{\gamma}^w_n)$ be as in Lemma \ref{lem:pivot-factor}
Note that items \eqref{item:2} and \eqref{item:3} imply that the supports of $x_{2k+2}$ and ${x_{2k+3}}$ do not depend on $k$. 
However, the support of $x_1$ may differ from the support of $x_3$. 
With that in mind, for all  $i\in\{1,\dots, M\}\cap X$ and all $j\in\{M+1,\dots, 2M\}\cap X$, the distribution $\nu_{i,j}$ is the distribution of $\gamma^w_k$ knowing that $x_k = i$ and $x_{k+1} = j$ for any $k \in\{1,3,\dots\}$ such that $\PP(x_k = i) > 0$.

\begin{proof}
	Let $\tilde\kappa$ be as in Lemma \ref{lem:ping-pong}. Let $\left((\widetilde\gamma^w_n),(\tau_k)\right)\sim(\tilde\kappa\otimes\tilde\nu_s)^{\otimes\NN}\otimes\mathcal{U}_{[0,1]}^{\otimes\NN}$ and let $(p_k)_{k\in\NN}$ be the associated random sequence of pivotal times and let $(l_k)_{k\in\NN}$ be as in definition \ref{def:pivot}. Let $\phi_L, \phi_R :\Gamma \to \{1,\dots, M\}$ be such that $L_i = \phi_L^{-1}\{i\}$ and $R_i = \phi_R^{-1}\{i\}$ for all $i\in\{1, \dots, M\}$.
	
	We define $x_0 := 0$ and we write $\tilde\nu'_0$ for the distribution of $\left(\widetilde\gamma^p_0, \phi_L(\gamma^p_0)\right)$. Given $k \in\NN$, we define:
	\begin{equation*}
		x_{2k+1} := \phi_L\left(\gamma^p_{2k}\right)\quad\text{and}\quad
		x_{2k + 2} := M+\phi_R\left(\gamma^w_{2l_{k+1}}\right).
	\end{equation*}
	Note that for all $g\in\Gamma$, the set $\{h\in\Gamma\,|\,g\AA h\}$ is determined by $\phi_L(\gamma)$ and the set $\{h\in\Gamma\,|\,h\AA g\}$ is determined by $\phi_R(\gamma)$.
	
	Note also that by construction, for all integer $k$, the conditional distribution of $\left(\widetilde\gamma^p_{k'}\right)_{k' \ge 2k+1}$ relatively to $\left(\widetilde\gamma^p_{k'}\right)_{k'\le 2k}$ and $\left(\tau_{l'}\right)_{l' < l_k}$ only depends on $x_{2k+1}$ and not on $k$. However the distribution of $x_{2k+1}$ itself may depend on $k$. Given $x \in \{1,\dots, M\}$ a possible value for $x_k$, write $\tilde\nu'_{x}$ for the distribution of $\left(\widetilde{\gamma}^w_{2k+1},x_{2k+2}\right)$ knowing $x_{2k+1} = x$.
	Note that by construction ,this distribution does not depend on $k$.
	
	For all integer $k$, the distribution of $\left(\widetilde\gamma^p_{k'}\right)_{k'\ge 2k+2}$ relatively to $\left(\widetilde\gamma^p_{k'}\right)_{k'\le 2k+1}$ and $\left(\tau_{l'}\right)_{l' \le l_k}$ only depends on $x_{2k+2}$ and not on $k$. Given $x \in \{M+1,\dots, 2M\}$ a possible value for $x_k$, write $\tilde\nu'_{x}$ for the distribution of $\left(\widetilde{\gamma}^w_{2k+2},x_{2k+3}\right)$ knowing $x_{2k+2} = x$.
	Again this distribution does not depend on $k$.
	
	Then for all $i \in \{1,\dots, M\} \cap X$ and all $j \in \{M+1,\dots, 2M\} \cap X$, the distribution $\nu_{i,j}$ is the distribution of $\gamma^p_{2k+1}$ knowing that $\phi_L \left(\gamma^p_{2k}\right) = i$ and $M + \phi_R\left(\gamma^w_{2l_{k+1}}\right) = j$. This distribution is bounded above by $\frac{\nu}{1-\alpha}$ by \eqref{eq:tjr-scho} in Theorem \ref{th:pivot-extract}.
\end{proof}

Now we can prove \ref{th:pivot} by taking an extraction.

\begin{proof}[Proof of Theorem \ref{th:pivot}]
	Let $\rho\in(0,1/3)$ and let $\rho' := \frac{\rho}{1+2\rho}\in(0,1/5)$. 
	Let $K \in \NN$  and let $K' = 2K$. 
	Without loss of generality, we assume that $K \ge 8$. 
	Let $m\in\NN$, let $0 < \eps', \alpha< 1$ and let $\tilde\nu_s$ be as in Corollary \ref{cor:schottky} applied to $\nu, \rho', K'$ and let $\nu_s := \Pi\tilde{\nu}_s$. 
	Then $\tilde\nu_s$ is compactly supported, and bounded above by $\frac{\nu^{\otimes m}}{\alpha}$. 
	Moreover $\nu_s$ is $\rho'$-Schottky for $\AA^{\eps'}$ and $\sqz_*\Pi_*\tilde\nu_s\left[K'|\log(\eps')| + K'\log(2)\right]$. Let $\eps := \frac{\eps'}{2}$ and let $\AA^{\eps'}\subset\AA \subset \AA^{\eps}$ be a finitely described binary relation. 
	Then $K'|\log(\eps')| + K'\log(2) = 2K |\log(\eps)| + K \log(2) \ge K |\log(\eps)| + K \log(2)$ and $\nu_s$ is $\rho'$-Schottky for $\AA$.
	
	Let $M \in\NN$, $X\subset\{0,\dots, 2M+1\}$, $\tilde\mu$ and $\left(\tilde\nu'_x\right)_{x\in X}$ be as in Lemma \ref{lem:pivot-factor}. Let $(\widetilde\gamma^p_k) \sim \tilde\mu$ and let $(x_n)$ be the underlying Markov chain. Let $i,j \in X$ be such that $0 < i \le M < j \le 2M$. Let $q_0 := \min \{q\in\NN\,|\,(x_{q},x_{q+1}) = (i,j)\}$, and define by induction $q_{2k+1} = 1$ and:
	\begin{equation*}
		q_{2k+2} := \min \{q \ge \overline{q}_{2k+2}\,|\,(x_{q},x_{q+1}) = (i,j)\} - \overline{q}_{2k+2}
	\end{equation*}
	for all $k$. 
	Let $\tilde\kappa_0$ be the distribution of $\widetilde\gamma^{p^{q}}_0$, let $\tilde\kappa_1$ be the distribution of $\widetilde\gamma^{p^{q}}_1$ and let $\tilde\kappa_2$ be the distribution of $\widetilde\gamma^{p^{q}}_2$. 
	By the factorization property, we have $\left(\widetilde\gamma^{p^{q}}_k\right)\sim\tilde\kappa_0\otimes\left(\tilde\kappa_1\otimes\tilde\kappa_2\right)^{\otimes\NN}$, which proves point \eqref{1}. Then each $q_k$ has bounded exponential moment because it is the hitting time of a finite Markov chain. 
	
	Moreover each $p_k$ has finite exponential moment so $p_k^q$ has finite exponential moment for all $q$, therefore $L_*\tilde\kappa_i$ has finite exponential moment for all $i$, this proves point \eqref{2}.
	
	By Proposition \ref{prop:pivot-is-aligned}, we have $\gamma^{p^q}_{i}\cdots\gamma^{p^q}_{j-1}\AA^\frac{\eps}{4} \gamma^{p^q}_{j}\cdots\gamma^{p^q}_{k-1}$ for all $0\le i\le j\le k$, which proves \eqref{4}.
	
	Note also that $\tilde\kappa_1$ is the restriction of $\tilde\nu_s$ to $\{\gamma\in\Gamma\,|\,L_i \AA \gamma\AA R_j\}$, which has measure at least $1-2\rho$, hence $\kappa_1$ is $\frac{\rho'}{1-2\rho'}$-Schottky for $\AA$, hence it is $\rho$-Schottky for $\AA^\eps$.
	
	Let $i,j \in X$ that do not satisfy $0< i \le M < j \le 2M$ and such that $(\chi_X)_*\tilde\nu'_i \{j\} > 0$. 
	The distribution $\tilde\nu_{i,j} := (\chi_{\Gamma})_*\frac{\mathds{1}_{\tilde\Gamma\times\{j\}}}{\nu'_i(\tilde\Gamma\times\{j\})}\tilde\nu'_i $ is absolutely continuous with respect to the distribution of $(\chi_{2k}^{\infty})_* \tilde\mu$ for $k \in\NN$ such that $\PP(x_{2k} = i) > 0$. Therefore, by \eqref{eq:densite} there is a constant $C$ such that for all $A \subset \Gamma \setminus\bigcup_{k = 0}^{m-1}\chi_k^m\mathbf{supp}(\tilde\nu_s)$, we have:
	\begin{equation}
		\forall k \le l, \forall i\in\{0,2\}, \tilde{\kappa}_i\left(L^{-1}\{l\}\cap (\chi_k^l)^{-1}(A)\right) \le C \nu(A) L_*\tilde\kappa_i\{l\}.
	\end{equation}
	Now assume that $\Gamma =\mathrm{GL}(E)$. The set $\bigcup_{k = 0}^{m-1}\chi_k^m\mathbf{supp}(\tilde\nu_s)$ is compact and $N$ is a continuous function on $\Gamma$. Let $B = \max N(\bigcup_{k = 0}^{m-1}\chi_k^m\mathbf{supp}(\tilde\nu_s))$. Then with the notations of Definitions \ref{def:trunking} and \ref{def:pushup} the distributions $(\zeta_{i,k,l})$ defined in Theorem \ref{th:pivot} are uniformly bounded by $B \wedge \left\lceil C N_{*}\nu\right\rceil$. When $\nu$ is not supported on $\mathrm{GL}(E)$, we have $N_*\nu\{+\infty\} > 0$, therefore $N_*\nu$ dominates any probability distribution and point \eqref{6} is trivial. However $C$ can not be expressed in terms of $(\alpha, \rho, m)$ because we did not give an explicit formula for the distribution of the sequence $(q_k)$.
\end{proof}

\section{Proof of the results}\label{results}
	
	In this section, we use Theorem \ref{th:pivot-extract} and Theorem \ref{th:pivot} together with Lemmas \ref{lem:r-piv} and \ref{lem:c-piv} to prove the results stated in the introduction. 
	Most of the proofs are straightforward application of Theorem \ref{th:pivot-extract} and Lemma \ref{lem:r-piv}, for the probabilistic estimates on the coefficients, on the $\sqz$ coefficient and on the speed of convergence to the invariant measure; and Lemma \ref{lem:c-piv} for the estimates on the spectral radius, on the spectral gap and on the dominant eigenspace. 
	Unexpectedly\footnote{This result is well known in the $\mathrm{L}^1$ case without any algebraic assumption on the support of the measure.}, the trickiest part is to show the almost sure convergence result in Theorem \ref{th:escspeed}, namely that $\sqz(\overline{\gamma}_n)/n \to \sigma(\nu)$.
	This is also the only reason why we need Lemma \ref{lem:pivot-factor}.
	The question whether $\prox(\overline{\gamma}_n)/n$ converges almost surely (or even in probability) without moment conditions remains open.
	
\subsection{Law of large numbers and large deviations inequalities for the singular gap}
	
	In this section, we define the escape speed of a random product of matrices using Theorem \ref{th:pivot} but not the moment estimate \eqref{6}. We will use usual ergodic theory but only for the proof of the almost sure convergence.
	Given $(x_n)$ a random sequence of real numbers and $\sigma \in \RR$, we say that the sequence $(x_n)$ satisfies (exponential) large deviations inequalities below the speed $\sigma$ if for all $\alpha < \sigma$, we have $\limsup\frac{1}{n} \log(\PP(x_n \le \sigma n)) < 0$.
	Note that if the distribution of $(x_n)$ is a Dirac measure then it satisfies large deviations inequalities below the speed $\sigma$ if and only if $\liminf\frac{x_n}{n} \ge \sigma$.
	In Lemma \ref{lem:proba:ldevcompo} in appendix, we show that the notion of large deviations behaves well when taking the sum~\eqref{compo:shift}~\eqref{compo:sum}, maximum~\eqref{compo:max} or minimum~\eqref{compo:min} of finitely many random sequences and also when composing random sequences of integers~\eqref{compo:compo}.

	\begin{Lem}[Escape speed and large deviations inequalities for self-aligned measures]\label{lem:pre-speed}
		Let $E$ be a Euclidean vector space and let $0 < \eps <1$. Let $\kappa$ be a probability distribution on $\mathrm{End}(E)$ and let $(g_k) \sim\kappa^{\otimes \NN}$. Assume that almost surely and for all $0 \le i\le j\le k$, we have $g_i\cdots g_{k-1} \AA^\frac{\eps}{4} g_j\cdots g_{k-1}$. Assume also that almost surely and for all $n \in\NN$, we have $\overline{g}_n \neq 0$. Then there is a limit $\sigma(\kappa)\in[0, +\infty]$ such that $\frac{\sqz(\overline{g}_n)}{n} \to \sigma(\kappa)$ almost surely and:
		\begin{equation}
			\forall \alpha < \sigma(\kappa), \; \exists C, \beta >0,\; \forall n\in\NN,\; \PP\left(\sqz(\overline{g}_n) \le \alpha n\right) \le C \exp(-\beta n).
		\end{equation}
		If we moreover assume that $\EE\left(\sqz_*\kappa\right) > 2|\log(\eps)|+4\log(2)$, then $\sigma(\kappa) > 0$.
	\end{Lem}
	
	\begin{proof}
		Let $N\in\NN$. Let $\sigma_N := \frac{1}{N} \EE\left(\sqz_*\kappa^{*N}\right) = \frac{1}{N} \EE\left(\sqz(\overline{g}_N)\right)$. For all $k \in \NN$, let $x_k^N := \sqz\left(g_{kN}\cdots g_{(k+1) N - 1}\right)$. Then $(x_n^N)$ is i.i.d. and takes positive values and $\EE(x_n^N) = N\sigma_N$ for all $n$. Then by Corollary \ref{cor:ldev-classique}, the sequence $(\overline{x}^N_k)_{k\in\NN}$ satisfies large deviations inequalities under the speed $N \sigma_N$. Moreover, by \eqref{eqn:lenali} in Lemma \ref{lem:c-prod} applied to $\overline{g}_{kN} \AA^\frac{\eps}{4} g_{kN}\cdots g_{(k+1) N - 1}$ for all $0\le k < \frac{n}{N}$ and then to $\overline{g}_{N \lfloor\frac{n}{N}\rfloor}\AA^\frac{\eps}{4}{g}_{N \lfloor\frac{n}{N}\rfloor} \cdots g_{n-1}$, we have:
		\begin{equation*}
			\forall n \in\NN, \; \sqz(\overline{g}_n) \ge x^N_{\lfloor\frac{n}{N}\rfloor} - n \frac{2|\log(\eps)|+4\log(2)}{N}.
		\end{equation*}
		Hence, by \eqref{compo:sum} applied to $\left(x^N_{\lfloor\frac{n}{N}\rfloor}\right)_n$ summed with $\left(n \frac{2|\log(\eps)|+4\log(2)}{N}\right)_n$ and \eqref{compo:compo} applied to $(x_n)_n$ composed with $\left(\lfloor\frac{n}{N}\rfloor\right)_n$ in Lemma \ref{lem:proba:ldevcompo}, the sequence $\sqz(\overline{g}_n)$ satisfies large deviations inequalities below the speed $\sigma_N-\frac{2|\log(\eps)|+4\log(2)}{N}$. This is true for all $N \in\NN$ so $\sqz(\overline{g}_n)$ satisfies large deviations inequalities below the speed $\sigma(\kappa) := \limsup_{N\in\NN} \sigma_N$.
		
		Let $T : \Gamma^\NN \to \Gamma^\NN;\,(\gamma_k)_{k\in\NN}\mapsto (\gamma_{k+1})_{k\in\NN}$. The transformation $T$ is ergodic for the measure $\mu := \kappa^{\otimes\NN}$. For all $n \in\NN$, let $f_n : (\gamma_k)_{k\in\NN} \mapsto 2|\log(\eps)|+4\log(2)-\sqz(\overline{\gamma}_n)$. Then $f_n$ is bounded above so $\EE_\mu(f_n)\le 2|\log(\eps)|+4\log(2)$ for all $n \in\NN$. Let $m,n$ be integers, then by \eqref{eqn:lenali} in Lemma \ref{lem:c-prod}, we have almost surely for $g \sim \mu$:
		\begin{equation*}
			\sqz(g_0\cdots g_{n+m-1}) \ge \sqz(g_0\cdots g_{n-1}) + \sqz(g_n \cdots g_{n+m-1}) - 2|\log(\eps)| - 4\log(2)
		\end{equation*}
		Hence $f_{n+m}(g) \le f_n(g) + f_m\circ T^n(g)$. So by Kingman's sub-additive ergodic Theorem~\cite{K68}, the sequence $\frac{f_n}{n}$ converges $\mu$-almost everywhere to $\liminf \frac{\EE(f_N*\mu)}{N} = \liminf \frac{2|\log(\eps)|+4\log(2)}{N} - \sigma_N$ and this inferior limit is actually a limit by classical sub-additivity. 
		Therefore $\frac{\sqz(\overline{g}_n)}{n} \to \sigma(\kappa)$ and $\sigma(\kappa) \ge \EE\left(\sqz_*\kappa\right) - 2|\log(\eps)| - 4\log(2)$ almost surely by sub-additivity.
	\end{proof}

	\begin{Th}[Large deviations inequalities for the singular gap]\label{th:ldev-sqz}
		Let $E$ be a Euclidean vector space and let $\nu$ be a strongly irreducible and proximal probability distribution on $\mathrm{End}(E)$. 
		Let $\tilde\kappa_0,\tilde\kappa_1,\tilde\kappa_2$ be as in Theorem \ref{th:pivot} for $\rho = \frac{1}{4}$ and $K = 10$. 
		Let $\tilde\kappa := \tilde\kappa_1 \odot \tilde\kappa_2$ and let $\kappa := \Pi_*\tilde\kappa$.
		Let $\sigma := \sigma(\kappa)/\EE(L_*\tilde\kappa)$. Let $(\gamma_n)\sim\nu^{\otimes\NN}$. Then the random sequence $\left(\sqz(\overline\gamma_n)\right)_{n\in\NN}$ satisfies large deviations inequalities below the speed $\sigma$ in the sense of Definition \ref{def:proba:ldev} \ie
		\begin{equation}\label{eq:ldev-sqz}
			\forall \alpha < \sigma, \; \exists C, \beta >0,\; \forall n\in\NN,\; \PP\left(\sqz(\overline{\gamma}_n) \le \alpha n\right) \le C \exp(-\beta n).
		\end{equation}
	\end{Th}
	
	\begin{proof}
		Let $0 < \eps \le 1$ be as in Theorem \ref{th:pivot}.
		Let $(\widetilde{\gamma}^w_k) \sim \tilde\kappa_0 \otimes \left(\tilde\kappa_1 \otimes \tilde\kappa_2\right)$. 
		To all integer $n \in\NN$, we associate $q_n =: \max\{q\in\NN\,|\, \overline{w}_{2q} \le n\}$ and a random integer $r_n$ such that $\gamma^w_{2q_n-2r_n-1}\AA^\eps(\gamma_{\overline{w}_{2q_n-2r_n}}\cdots \gamma_{n-1})$ or $r_n \ge q_n$. 
		Note that the conditional distribution of the sequence $\gamma^w_0,\gamma^w_1, \dots, \gamma^w_{2q_n-1},\gamma_{\overline{w}_{2q_n}} \cdots \gamma_{n-1}$ relatively to $q_n$ is $\frac{1}{4}$-ping-pong for all values of $q_n$. 
		Hence, we may assume that $r_n \sim \mathcal{G}_{\frac{1}{4}}$ for all $n$, by Lemma \ref{lem:r-piv}. 
		
		The distribution $L_{*}\tilde\kappa_i$ has finite exponential moment and is supported on $\NN_{\ge 1}$ for all $i$. 
		By Corollary \ref{cor:ldev-classique}, the random sequence $(\overline{w}_{2m}-w_0)_m$ satisfies large deviations inequalities around the speed $\EE(L_*\tilde\kappa) \in (0, +\infty)$.
		Then by \eqref{compo:shift} in Lemma \ref{lem:proba:ldevcompo}, the random sequence $(\overline{w}_{2m})_m$ also does.
		By \eqref{compo:rec} in Lemma \ref{lem:proba:ldevcompo}, the random sequence $(q_n)_n$ satisfies large deviations inequalities around the speed $\EE(L_*\tilde\kappa)^{-1}$ and by \eqref{compo:shift} in Lemma \ref{lem:proba:ldevcompo}, $(q_n-r_n-1)$ also does.
		Then by Lemma \ref{lem:pre-speed}, and by composition~\eqref{compo:compo} in Lemma~\ref{lem:proba:ldevcompo}), the sequence $(\sqz(\gamma^w_1\cdots\gamma^w_{2q_n-2r_n-2}))_n$ satisfies large deviations inequalities below the speed $\sigma$.
		
		Now by \eqref{4} in Theorem \ref{th:pivot}, we have  $\gamma^w_0\AA^\frac{\eps}{4}\gamma^w_1\cdots\gamma^w_{2q_n-2r_n-2}$, so by \eqref{eqn:lenali} in Lemma \ref{lem:c-prod}, the sequence $(\sqz(\gamma^w_0\cdots\gamma^w_{2q_n-2r_n-2}))_n$ satisfies large deviations inequalities below the speed $\sigma$. Moreover, we have:
		\begin{equation*}
			\gamma^w_0\cdots\gamma^w_{2q_n-2r_n-2} \AA^\frac{\eps}{4} \gamma^w_{2q_n-2r_n-1}\AA^\eps(\gamma_{\overline{w}_{2q_n-2r_n}}\cdots \gamma_{n-1})
		\end{equation*}
		and $\sqz(\gamma^w_{2q_n-2r_n-1})\ge K|\log(\eps)| + K\log(2) \ge 2|\log(\eps/2)|+3\log(2)$ so by the transpose of Lemma \ref{lem:herali}, we have:
		\begin{equation*}
			\gamma^w_0 \cdots \gamma^w_{2q_n-2r_n-1} \AA^\frac{\eps}{4} (\gamma_{\overline{w}_{2q_n-2r_n}}\cdots \gamma_{n-1}).
		\end{equation*}
		Hence $(\sqz(\overline{\gamma}_n))_n$ satisfies large deviations inequalities below the speed $\sigma$
	\end{proof}
	
	Note that in Theorem \ref{th:ldev-sqz}, we do not claim that $\frac{\sqz(\gamma_n)}{n} \to \sigma$ almost surely.
	We do however claim that in Theorem \ref{th:escspeed}. 
	The remaining part of this paragraph is dedicated to the proof of that claim. 
	The usual proof using Lyapunov coefficients does not work in our case because $\sigma$ may be finite even when $\nu$ has infinite first moment. 
	Kingman's theorem can not be used either because $\sqz$ is not sub-additive nor super-additive.
	In fact, we will really use the strong irreducibility of $\nu$ to prove it with the following trick.
	
	\begin{Lem}\label{lem:left-ali-tjrs}
		Let $E$ be a Euclidean vector space and let $\nu$ be a strongly irreducible and proximal probability distribution on $\mathrm{End}(E)$. 
		Let $(\gamma_n) \sim \nu^{\otimes\NN}$.
		Let $\eps$ be as in Theorem \ref{th:pivot} for $\rho= \frac{1}{4}$ and $K = 10$. There exist $l_0\in\NN$ and $0 < \beta < 1$ such that:
		\begin{equation}\label{eq:ghi}
			\forall g \in\Gamma,\;\PP\left(\forall n\ge l_0,\; g\AA^{\frac{\eps}{4}}\overline\gamma_n\right) > \beta.
		\end{equation}
	\end{Lem}
	
	\begin{proof}
		We use the notations of the proof of Theorem \ref{th:ldev-sqz}.
		Let $0< \eps \le 1$ and $\tilde\kappa_0,\tilde\kappa_1,\tilde\kappa_2$ be as in Theorem \ref{th:pivot} for $\rho = \frac{1}{4}$ and $K = 10$.
		Let $(\widetilde{\gamma}^w_k) \sim \tilde\kappa_0 \otimes \left(\tilde\kappa_1 \otimes \tilde\kappa_2\right)$.
		For all $n \in\NN$, let $q_n =: \max\{q\in\NN\,|\, \overline{w}_{2q} \le n\}$ and let  $r_n$ be the smallest integer such that $\gamma^w_{2q_n-2r_n-1}\AA^\eps(\gamma_{\overline{w}_{2q_n-2r_n}}\cdots \gamma_{n-1})$ or $r_n \ge q_n$. 
		By Lemma \ref{lem:r-piv}, $r_n$ has a bounded exponential moment that does not depend on $n$. 
		Note that if $r_n < q_n - 1$, then by \eqref{4} in Theorem \ref{th:pivot}, we have:
		\begin{equation*}
			\gamma_0\gamma_1\AA^\frac{\eps}{4} \gamma^w_2\cdots \gamma^w_{2q_n - 2 r_n -2} \AA^\frac{\eps}{4} \gamma^w_{2q_n - 2 r_n - 1} \AA^\eps \gamma_{\overline{w}_{2 q_n - 2r_n}}\cdots \gamma_{n-1}
		\end{equation*}
		so we have $\gamma_0\gamma_1\AA^\frac{\eps}{8} \gamma_{\overline{w}_2}\cdots \gamma_{n-1}$ by Lemma \ref{lem:herali}.
		By Lemma \ref{lem:r-piv}, we may assume that $(q_n-r_n)$ satisfies large deviations inequalities below the speed $\EE(L_*\tilde\kappa)^{-1} > 0$.
		Let $l'_0$ be such that $\PP(\forall n\ge l'_0,\;r_n < q_n) \ge \frac{1}{2}$. 
		Now let $m\in\NN$ and $\alpha >0$ be as in Corollary \ref{cor:schottky} for $\rho = \frac{1}{4}$ and $K = 10$. 
		Let $\gamma_{-m}, \dots, \gamma_{-1}\sim\nu^{\otimes m}$ be independent of $(\widetilde{\gamma}^w_k)$. 
		Then we have:
		\begin{equation}\label{eq:63}
			\PP\left(g \AA^\eps \gamma_{-m}\cdots \gamma_{-1}\AA^\eps \left(\gamma^w_0\gamma^w_1\right) \cap \sqz\left(\gamma_{-m}\cdots \gamma_{-1}\right)\ge K|\log(\eps)| + K\log(2)\right) \ge \alpha (1-2\rho) \ge \frac{\alpha}{3}.
		\end{equation}
		Now if we assume that:
		\begin{equation*}
			g \AA^\eps \gamma_{-m}\cdots \gamma_{-1}\AA^\eps \gamma^w_0\gamma^w_1\AA^\frac{\eps}{8} \gamma_{\overline{w}_2}\cdots \gamma_{n-1}
		\end{equation*}
		Then by Lemma \ref{lem:herali}, we have $g\AA^{\frac{\eps}{4}} \gamma_{-m}\cdots\gamma_{n-1}$. Now note that \eqref{eq:63} holds for the conditional distribution relatively to $(\gamma_n)_{n\ge 0}$. Hence, we have:
		\begin{equation}
			\PP\left(\forall n \ge l'_0,\, g\AA^{\frac{\eps}{4}} \gamma_{-m}\cdots\gamma_{n-1}\right) \ge \frac{\alpha}{6}.
		\end{equation}
		Moreover $(\gamma_{k-m})_{k\ge 0}\sim\nu^{\otimes\NN}$, therefore, we have \eqref{eq:ghi} for $l_0 =l'_0 +m$.
	\end{proof}
	
	Now we will use Lemma \ref{lem:left-ali-tjrs} to deduce the almost sure convergence result in Theorem \ref{th:escspeed} from the almost sure convergence result in Lemma \ref{lem:pre-speed}.
	
	\begin{Lem}[Almost sure convergence]\label{lem:conv}
		Let $E$ be a Euclidean vector space, let $\Gamma= \mathrm{End}(E)$, let $\nu$ be a strongly irreducible and proximal probability distribution on $\Gamma$. 
		Let $0< \eps <1$ and let $\tilde\kappa_0,\tilde\kappa_1,\tilde\kappa_2$ be as in Theorem \ref{th:pivot} for $\rho = \frac{1}{4}$ and $K = 10$. 
		Let $\tilde\kappa := \tilde\kappa_1 \odot \tilde\kappa_2$ and let $\kappa := \Pi_*\tilde\kappa$.
		Let $\sigma := \sigma(\kappa)/\EE(L_*\tilde\kappa)$. Let $(\gamma_n)\sim\nu^{\otimes\NN}$. Then $\frac{\sqz(\gamma_n)}{n} \to \sigma$ almost surely and $\sigma > 0$.
	\end{Lem}

	\begin{proof}
		Let $(\widetilde{\gamma}^w_k) \sim \tilde\kappa_0 \otimes \left(\tilde\kappa_1 \otimes \tilde\kappa_2\right)^{\otimes \NN}$.
		Then $(\gamma_n)\sim \nu^{\otimes\NN}$ by \eqref{1} in Theorem \ref{th:pivot}.
		By Theorem \ref{th:ldev-sqz}, $\left(\sqz(\gamma_n)\right)_{n\in\NN}$ satisfies large deviations inequalities, below the speed $\sigma$.
		Hence $\liminf \frac{\sqz(\gamma_n)}{n} \ge \sigma$ almost surely.
		Therefore, we only need to show that $\limsup \frac{\sqz(\gamma_n)}{n} \le \sigma$ almost surely.
		Note that $\sigma > 0$ by Lemma \ref{lem:pre-speed}.
		
		Assume by contradiction that $\sigma < +\infty$ and $\PP\left(\limsup \frac{\sqz(\gamma_n)}{n} > \sigma \right) > 0$. 
		Let $\delta > 0$ be such that $\PP\left(\limsup \frac{\sqz(\gamma_n)}{n} \ge (1 + 2\delta) \sigma \right) \ge \delta$. 
		Then for all integer $n_0 \in \NN$, we have:
		\begin{equation}
			\PP\left(\exists n\ge n_0,\; {\sqz(\gamma_n)}\ge n(1 + \delta)\sigma\right) \ge \delta.
		\end{equation}
		Let $l_0\in\NN$ and $0 < \beta < 1$ be as in Lemma \ref{lem:left-ali-tjrs} Assume that they also satisfy \eqref{eq:ghi} for the transpose of $\nu$ (which is strongly irreducible and proximal). 
		Then, for all $g \in \mathrm{End}(E)$, and for all $n \in\NN$, we have:
		\begin{equation}\label{eq:lazob}
			\PP\left(\forall l_0 \le l \le n, \gamma_{n-l}\cdots\gamma_{n-1}\AA^\frac{\eps}{4} g\right) \ge \beta.
		\end{equation}
		Note that \eqref{eq:lazob} also works when $g$ is a random endomorphism which is independent of the word $(\gamma_0, \dots , \gamma_{n-1})$.
		
		Let $a_0 \in\NN$ be the smallest integer such that $\PP(w_0 > a_0) \le \frac{\beta^2\delta}{4}$ and let $a_1 := a_0 + l_0$.
		We use the convention $\min\emptyset = +\infty$.
		Let $m_0 \in \NN$. 
		Let $n_0$ be the smallest integer such that $\PP\left(\overline{w}_{2m_0+1} \ge n_0 + a_1 + l_0\right) \le \frac{\beta^2\delta}{3}$. 
		We define:
		\begin{equation*}
			n_1 := \min\left\{n\ge n_0 \,|\,\sqz(\gamma_{a_1} \cdots \gamma_{a_1 + n})\ge n (1+\delta)\sigma\right\}.
		\end{equation*} 
		Then $\PP(n_1 \neq +\infty) \ge \delta$. 
		Moreover, $n_1$ is a stopping time so the random word $(\gamma_{a_1}, \dots , \gamma_{a_1 + n_1})$ is independent of the random sequence $(\gamma_{a_1 +n_1 + k + 1})_{k \in\NN}$.
		Both are also independent of the word $(\gamma_0,\dots,\gamma_{a_1 - 1})$ by construction.
		Let $(\gamma_n)_{n < 0} \sim \nu^{\otimes \ZZ_-}$ be independent of $(\gamma_n)_{n \ge 0}$.
		Then by Lemma \ref{lem:left-ali-tjrs} applied to $g = \gamma_{a_1} \cdots \gamma_{a_1 + n_1}$ and to the random sequence $(\gamma_{a_1 +n_1 + k + 1})_{k \in\NN}$ on the right ($k \ge a_1 + n_1 + l_0$) and  by \eqref{eq:lazob} on the left ($j \le a_0$), we have the following:
		\begin{equation*}
			\PP\left(n_1 < +\infty \cap \forall j \le a_0,\forall k \ge a_1 + n_1 + l_0,\gamma_j\cdots \gamma_{a_1-1}\AA^{\frac{\eps}{4}}\gamma_{a_1}\cdots \gamma_{a_1+n_1-1} \AA^\frac{\eps}{4} \gamma_{a_1+n_1}\cdots\gamma_{k-1}\right) \ge \beta^2 \delta.
		\end{equation*}
		Let $D := 4|\log(\eps)|+10\log(2)$. Then by Lemma \ref{lem:triple-ali}, we have:
		\begin{equation}\label{eq:gdfgf}
			\PP\left(n_1 < +\infty \cap \forall j \le a_1,\forall k \ge a_1 + n_1 + l_0,\;\sqz(\gamma_j\cdots\gamma_{k-1})\ge n_1(1+\delta)\sigma - D\right) \ge \beta^2\delta.
		\end{equation}
		Now we define the random integer:
		\begin{equation*}
			m_1 := \min\left\{k\in\NN\,|\,\overline{w}_{2k+1} > a_1 + n_1 + l_0\right\}.
		\end{equation*}
		Then by \eqref{eq:gdfgf} with $k = \overline{w}_{2m_1+1}-1 \ge a_1 + n_1 + l_0$ we have:
		\begin{equation*}
			\PP\left(m_1 < +\infty \cap \forall j \le a_1,\;\sqz(\gamma_j\cdots\gamma_{\overline{w}_{2m_1+1}-1})\ge n_1(1+\delta)\sigma - D \right) \ge \beta^2\delta.
		\end{equation*}
		Note also that with probability at least $1-\frac{\beta^2\delta}{3}$, we have $w_0 \le a_1$, therefore:
		\begin{equation*}
			\PP\left(m_1 < +\infty \cap \sqz(\gamma_{{w}_0}\cdots\gamma_{\overline{w}_{2m_1+1}-1})\ge n_1(1+\delta)\sigma - D \right) \ge \beta^2\delta.
		\end{equation*}
		Note also that by minimality of $m_1$, we have $\overline{w}_{2m_1-1}-a_1-l_0 \le n_1$. 
		Moreover, with our notations, we have $\gamma_{{w}_0}\cdots\gamma_{\overline{w}_{2m_1+1}-1} = \gamma^w_1\cdots\gamma^w_{2m_1}$. 
		Hence, we have:
		\begin{equation*}
			\PP\left(m_1 < +\infty \cap \sqz(\gamma_{w_0}\cdots \gamma_{\overline{w}_{2m_1+1}-1})\ge (\overline{w}_{2m_1-1}-a_1-l_0) (1+\delta) \sigma - D\right) \ge  {\beta^2\delta} - \frac{\beta^2\delta}{3}.
		\end{equation*} 
		Moreover, $\PP(m_1 < m_0) \le \frac{\beta^2\delta}{3}$ by construction, so we have:
		\begin{equation*}
			\PP\left( m_0 \le m_1 < +\infty \cap \sqz(\gamma_{w_0}\cdots \gamma_{\overline{w}_{2m_1+1}-1})\ge (\overline{w}_{2m-1}-a_1-l_0) (1+\delta) \sigma - D \right) \ge {\beta^2\delta} - 2 \frac{\beta^2\delta}{3}.
		\end{equation*}
		Hence, by taking $k = m_1$, we have:
		\begin{equation*}
			\PP\left(\exists k \ge m_0,\,\sqz(\gamma_{w_0}\cdots \gamma_{\overline{w}_{2k+1}-1})\ge (\overline{w}_{2k-1}-a_1-l_0) (1+\delta) \sigma - D\right) \ge \frac{\beta^2\delta}{3}.
		\end{equation*}
		The above is true for all $m_0$, therefore:
		\begin{equation*}
			\PP\left( \limsup_{k \to + \infty}\frac{\sqz\left(\gamma^w_1\cdots \gamma^w_{2k}\right)+D}{\overline{w}_{2k-1}-a-1-l_0} \ge (1+\delta)\sigma\right) \ge \frac{\beta^2\delta}{3}.
		\end{equation*}
		Moreover $\overline{w}_{2k-1} \ge 2k-1$ for all $k$ so we can get rid of the constants in the $\limsup$ and we have:
		\begin{equation}\label{eq:limsup}
			\PP\left( \limsup_{k \to + \infty}\frac{\sqz\left(\gamma^w_1\cdots \gamma^w_{2k}\right)}{\overline{w}_{2k-1}} \ge (1+\delta)\sigma\right) \ge \frac{\beta^2\delta}{3}.
		\end{equation}
		
		Now it remains to show that \eqref{eq:limsup} is in contradiction with Lemma \ref{lem:pre-speed}.
		We know that $\frac{\sqz\left(\gamma^w_1\cdots \gamma^w_{2k}\right)}{k} \to \sigma(\kappa) > 0$ almost surely by Lemma \ref{lem:pre-speed}.
		Moreover $\frac{w_1+\cdots + w_{2k - 2}}{k-1} \to \EE(L_*\tilde\kappa) > 0$ almost surely by the law of large numbers, hence $\frac{\overline{w}_{2k - 1}}{k} = \frac{w_0+\cdots + w_{2k - 2}}{k} \to \EE(L_*\tilde\kappa)$. 
		\begin{equation}
			\lim_{k \to +\infty} \frac{\sqz\left(\gamma^w_1\cdots \gamma^w_{2k}\right)}{\overline{w}_{2k-1}} = \frac{\sigma(\kappa)}{\EE(L_*\tilde\kappa)} = \sigma.
		\end{equation} 
		which contradicts \eqref{eq:limsup}. Hence $\limsup \frac{\sqz(\gamma_n)}{n} \le \sigma$ almost surely, which concludes the proof.
	\end{proof}

\subsection{Contraction property}	

	In this paragraph, we prove the following theorem. Given $\nu$ a strongly irreducible and proximal probability measure, we write $\sigma(\nu)$ for the quantity $\sigma$ defined in Lemma \ref{lem:conv}. 
	
	\begin{Def}\label{def:l-infty}
		Let $E$ be a Euclidean vector space, let $\Gamma= \mathrm{End}(E)$. We define the set of contracting sequences:
		\begin{equation*}
			\Omega'(E) := \left\{(\gamma_n)\in\Gamma^\NN\,\middle|\,
			\begin{aligned}
				& \forall n \in\NN,\; \overline\gamma_n \neq \{0\} \quad\text{and}\\
				& \forall k \in\NN,\; \forall \eps \in (0,1), \\ 
				& \limsup_{m,n \to +\infty} \max_{\substack{u \in U^\eps(\gamma_{k} \cdots \gamma_{n}) \setminus\{0\}, \\ u' \in U^\eps(\gamma_{k} \cdots \gamma_{m}) \setminus\{0\}}} \dist([u],[u']) = 0
			\end{aligned}\right\}.
		\end{equation*}
		We define $T := \Gamma^\NN \to \Gamma^\NN$ to be the Bernoulli shift and we define $l^\infty :\Omega'(E) \to \mathbf{P}(E)$ to be the only map such that:
		\begin{equation*}
			\forall (\gamma_n) \in \Omega(E),\; \limsup_{m,n \to +\infty} \max_{u \in U^\eps(\gamma_{0} \cdots \gamma_{n-1}) \setminus\{0\}} \dist([u],l^\infty(\gamma)) = 0.
		\end{equation*}
		Let:
		\begin{equation*}
			\Omega'(E) := \left\{ \gamma \in \Omega(E) \,\middle|\, \forall k\in\NN, \gamma_k l^\infty(T^{k+1}\gamma) = l^\infty(T^k\gamma) \right\}.
		\end{equation*}
	\end{Def}
	
	Note that given $E$ a Euclidean vector space, the space $\Omega(E)$ defined in Definition \ref{def:l-infty} is measurable and $T$-invariant. Moreover $l^\infty$ is $T$-equivariant on $\Omega(E)$ in the sense that $l^\infty(\gamma) = \gamma_0 l^\infty(T\gamma)$.
	
	Note also that $\Omega'(E) \neq \Omega(E)$. For example let $E = \RR^2$, and let $\pi_1$ and $\pi_2$ be the orthogonal projections onto the first and second coordinates. 
	If $\gamma_0 = \pi_1$ and $\gamma_k = \pi_1 +2\pi_2$ for all $k \ge 1$, then $\overline\gamma_n = \pi_1$ for all $n \ge 1$ and $[\gamma_k\cdots \gamma_n] \underset{m \to \infty}{\longrightarrow} [\pi_2]$ for all $k \ge 1$.
	So $l^\infty(\gamma)$ is the first coordinate axis and $\infty(T\gamma)$ is the second coordinate axis.
	Hence $\gamma_0 l^\infty(T\gamma) = [0] \neq l^\infty(\gamma)$ so $\gamma \in \Omega'(E) \setminus \Omega(E)$.
	
	\begin{Th}\label{th:limit-uv}
		Let $E$ be a Euclidean vector space, let $\Gamma= \mathrm{End}(E)$, let $\nu$ be a strongly irreducible and proximal probability distribution on $\Gamma$. Let $\gamma = (\gamma_n)_{n\in\NN}\sim\nu^{\otimes\NN}$. 
		Then $\gamma \in\Omega(E)$ almost surely. 
		Let $\alpha < \sigma(\nu)$. 
		There exist constants $C,\beta > 0$ such that:
		\begin{equation}\label{eq:limit-u}
			\PP\left(\exists u \in U^1(\overline\gamma_n)\setminus\{0\},\; \dist([u], l^\infty(\gamma))\ge \exp(-\alpha n)\right) \le C\exp(-\beta n).
		\end{equation}
		Moreover, for all $v \in E$, we have:
		\begin{equation}\label{eq:limit-v}
			\PP\left(\dist([\overline{\gamma}_n v], l^\infty(\gamma))\ge \exp(-\alpha n)\,|\, \overline{\gamma}_n v \neq 0\right) \le C\exp(-\beta n).
		\end{equation}
	\end{Th}
	
	\begin{proof}
		Let $\eps \in (0,1)$ and let $\tilde\kappa_0,\tilde\kappa_1,\tilde\kappa_2$ be as in Theorem \ref{th:pivot} for $\rho = \frac{1}{4}$ and $K = 10$. 
		Let $(\widetilde{\gamma}^w_k) \sim \tilde\kappa_0 \otimes \left(\tilde\kappa_1 \otimes \tilde\kappa_2\right)^{\otimes\NN}$. 
		By Corollary \ref{cor:limit-line}, there is a limit line $l^\infty$ such that:
		\begin{equation*}
			\forall m \in\NN, \forall u \in U^1\left(\overline{\gamma}^w_{m}\right)\setminus\{0\},\; \dist([u], l^\infty) \le \frac{4}{\eps} \exp(-\sqz\left(\overline{\gamma}^w_{m}\right)).
		\end{equation*}
		Then we necessarily have $l^\infty = l^\infty(\gamma)$ whenever $\gamma\in\Omega(E)$.
		To all integer $n \in\NN$, we associate $q_n =: \max\{q\in\NN\,|\, \overline{w}_{2q} \le n\}$ and a random integer $r_n\sim\mathcal{G}_\rho$ such that $\gamma^w_{2q_n-2r_n-1}\AA^\eps(\gamma_{\overline{w}_{2q_n-2r_n}}\cdots \gamma_{n-1})$ or $r_n \ge q_n$.
		Then by Lemma \ref{lem:herali}, if we assume that $r_n < q_n$, then $\overline\gamma^w_{2q_n-2r_n}\AA^\eps(\gamma_{\overline{w}_{2q_n-2r_n}}\cdots \gamma_{n-1})$, hence by Lemma \ref{lem:c-prod}, we have:
		\begin{equation*}
			\forall u \in U^1\left(\overline{\gamma}_n\right)\setminus\{0\},\;\forall u' \in U^1\left(\overline{\gamma}^w_{2q_n-2r_n}\right)\setminus\{0\},\; \dist([u],[u']) \le \frac{4}{\eps} \exp\left(-\sqz\left(\overline{\gamma}^w_{2q_n-2r_n}\right)\right).
		\end{equation*}
		In this case, by triangular inequality:
		\begin{equation*}
			\forall u \in U^1\left(\overline{\gamma}_n\right)\setminus\{0\},\; \dist([u],l^\infty) \le \frac{8}{\eps} \exp\left(-\sqz\left(\overline{\gamma}^w_{2q_n-2r_n}\right)\right).
		\end{equation*}
		By Corollary \ref{cor:curexp}, and by Lemma \ref{lem:sumexp} applied to the sequence $(w_k)$, the random variable $(\overline{w}_{q_n - r_n} - n)$ have a bounded exponential moment that does not depend on $n$. 
		Therefore, by \eqref{compo:shift} Lemma \ref{lem:proba:ldevcompo}, the random sequence $(\overline{w}_{q_n - r_n})$ satisfies large deviation inequalities around the speed $1$. 
		Then $\left(\sqz\left(\overline{\gamma}^w_{2q_n-2r_n}\right)\right)_n$ satisfies large deviations inequalities below the speed $\sigma(\nu)$ by Theorem \ref{th:ldev-sqz} and by \eqref{compo:compo} in Lemma \ref{lem:proba:ldevcompo}. 
		Hence we have \eqref{eq:limit-u}. 
		The above reasoning also works for $T^{k}\gamma$ for all $k$ therefore $\gamma\in\Omega'(E)$ almost surely.
		
		Now to show \eqref{eq:limit-v}, we use the same reasoning.
		We identify $E$ with $\mathrm{Hom}(\KK, E)$.
		Let $v \in E \setminus\{0\}$.
		Note that $U^1\left(\overline{\gamma}_n v\right) = \overline{\gamma}_n v \KK$.
		For all $n \in\NN$, we define a random integer $r_n^v$ such that $\gamma^w_{2q_n-2r_n^v-1}\AA^\eps(\gamma_{\overline{w}_{2q_n-2r_n^v}}\cdots \gamma_{n-1}v)$ or $r_n^v \ge q_n$. 
		Then by the same reasoning as for the proof of \eqref{eq:limit-u}, for all $n$ such that $r_n^v < q_n$ and $\overline\gamma_nv\neq 0$, we have:
		\begin{equation*} 
			\dist([\overline{\gamma}_nv],l^\infty) \le \frac{8}{\eps} \exp\left(-\sqz\left(\overline{\gamma}^w_{2q_n-2r_n^v}\right)\right).
		\end{equation*}
		Note moreover that $\left(\PP(\overline{\gamma}_n v = 0)\right)_{n\in\NN, v\neq 0}$ is bounded above by a constant by Lemma \ref{lem:descri-ker}.
		Hence we have \eqref{eq:limit-v}.
		 
		Let $v \in E \setminus\underline{\ker}(\nu)$.
		The above reasoning implies that $l^\infty = \lim [\overline{\gamma}_n v]$ almost surely so $l^\infty(\gamma) = \gamma_0 l^\infty(T\gamma)$ almost surely. Moreover, for all $k \in \NN$, the random sequence $T^k\gamma$ has distribution $\nu^{\otimes \NN}$, so we also have $l^\infty(T^k\gamma) = \gamma_k l^\infty(T^{k+1}\gamma)$. Therefore $\gamma \in \Omega(E)$ almost surely.
	\end{proof}

	Given $g \in\mathrm{End}(E)$ and $v \in E$, we write $g[v]$ or $[g][v]$ for $[gv]$, this is an element of $\mathbf{P}(E) \sqcup \{[0]\}$. 
	That way we have a measurable (but not everywhere continuous) semi-group action $\mathrm{End}(E) \curvearrowright \mathbf{P}(E) \sqcup \{[0]\}$, and a convolution product $\mathrm{Prob}(\mathrm{End}(E)) \times \mathrm{Prob}(\mathbf{P}(E) \sqcup \{[0]\}) \to \mathrm{Prob}(\mathbf{P}(E) \sqcup \{[0]\})$.
	Note that given $g \in \mathrm{End}(E) \setminus\{0\}$ and $v \in E \setminus \{0\}$, such that $gv \neq 0$, the map $([h], [x]) \mapsto [hx]$ is continuous at $([g], [v])$.
	
	Let un now prove Corollary \ref{cor:ergodic}, which says that given $E$ a Euclidean vector space and $\nu$ a strongly irreducible and proximal probability distribution on $\mathrm{End}(E)$, we have a unique $\nu$-stationary probability measure $\xi_\nu^\infty$ on $\mathbf{P}(E)$. Moreover, for all probability measure $\xi$ on $\mathbf{P}(E) \setminus\underline{\ker}(\nu)$, the sequence $(\nu^{*n}*\xi)_{n\in\NN}$ converges exponentially fast to $\xi_\nu^\infty$ for the dual of the Lipschitz norm.

	\begin{proof}[Proof of Corollary \ref{cor:ergodic}]
		Let $\nu$ be any strongly irreducible and proximal distribution on $\mathrm{End}(E)$.
		Let $\xi_\nu^\infty$ be the distribution of $l^\infty_*(\nu^{\otimes\NN})$.
		Let $({\gamma}_n)\sim\nu^{\otimes\NN}$.
		Let $l := l^\infty(\gamma)$ and let $l' := l^\infty \circ T(\gamma)$. 
		Then by \eqref{eq:limit-v} in Theorem \ref{th:limit-uv} applied to a vector $v \in E \setminus\underline{ker}(\nu)$, we have $l = \gamma_0 l'$ almost surely.
		Moreover, $l$ and $l'$ both have distribution $\xi_\nu^\infty$ and $\gamma_0$ and $l'$ are independent so $\xi_\nu^\infty$ is $\nu$-stationary.
		
		Let $\xi$ be a probability measure on $\mathbf{P}(E)$ that is supported on $\mathbf{P}(E)\setminus\underline{\ker}(\nu)$, let $\lambda \ge 0$ and let $f:\mathbf{P}(E)\to\RR$ be $\lambda$-Lipschitz. 
		Let $(l,(\gamma_n))\sim \xi \otimes\nu^{\otimes\NN}$ and let $l^\infty := l^\infty(\gamma)$.
		Let $0 < \alpha < \sigma(\nu)$ and let $C, \beta > 0$ be as in Theorem \ref{th:limit-uv}. 
		Note also that such $\beta, C$ do not depend on $\xi$.
		Then note that $\mathbf{P}(E)$ has diameter $1$.
		Therefore, for all $\delta\in(0,1)$ and for all $n \in\NN$, we have:
		\begin{equation}\label{eq:maj-lip}
			\left|\EE(f(l^\infty)) - \EE(f(\overline{\gamma}_n l))\right| \le \EE\left(|f(l^\infty)-f(\overline{\gamma}_n l)|\right) \le \lambda\EE(\dist(l^\infty,\overline{\gamma}_nl) \le \lambda \PP(\dist(l^\infty,\overline{\gamma}_nl) \ge \delta) + \lambda \delta.
		\end{equation}
		Moreover, since $l$ is independent of $\gamma$, we know that $\PP(\dist(l^\infty,\overline{\gamma}_n l) \ge \exp(-\alpha n)) \le C \exp(-\beta n)$ for all $n$ by \eqref{eq:limit-v} in Theorem \ref{th:limit-uv}. So by \eqref{eq:maj-lip} applied to all $n\in\NN$, with $\delta = \exp(-\alpha n)$, we have:
		\begin{equation*}
			\forall n\in\NN,\;\left|\EE(f(l^\infty)) - \EE(f(\overline{\gamma}_n l))\right| \le \lambda C \exp(-\beta n) + \lambda \exp(-\alpha n) \le \lambda (C+1) \exp(-\min\{\alpha, \beta\} n).
		\end{equation*}
		To conclude, note that for all $n \in\NN$, the random variable $\overline{\gamma}_n l$ has distribution law $\nu^{*n}*\xi$.
	\end{proof}

\subsection{Asymptotic estimates for the spectral gap and dominant eigenspace}
	
	We use exactly the same strategy as for the proofs of Theorems \ref{th:ldev-sqz} and \ref{th:limit-uv} but we use Lemma \ref{lem:c-piv} instead of Lemma \ref{lem:r-piv}. 
	Note that given $E$ a vector space and $g,h \in\mathrm{End}(E)$, then $\prox(gh) = \prox(hg)$ so $gh$ is proximal if and only if $hg$ is and in this case, $E^+(gh) = g E^+(hg)$.

	Note that given $g \in \mathrm{End}(E)$ a proximal matrix, the constant sequence $\gamma : n \mapsto g$ is in $\Omega(E)$ and $l^\infty(\gamma) = E^+(g)$. 
	In this section we use Lemma \ref{lem:c-piv} applied to the extraction constructed in Theorem \ref{th:pivot-extract}.
	We could use the extraction constructed in Theorem \ref{th:pivot} because we do not care about moments or independence. 
	However need the notion of inductive alignment $\widetilde{\AA}^S$ as defined in the beginning of Section \ref{sec:pivot} to be able to use Lemma \ref{lem:c-piv} as we explained in Remark \ref{rem:c-piv}.
	
	\begin{Th}\label{th:lim-eigen}\label{th:ldevprox}
		Let $\nu$ be a strongly irreducible and proximal probability distribution over $\mathrm{End}(E)$. 
		Let $(\gamma_n)\sim\nu^{\otimes\NN}$ and let $l^\infty := l^\infty(\gamma)$. 
		Then we have:
		\begin{gather}
			\forall\alpha < \sigma(\nu),\;\exists C,\beta > 0,\;\forall n\in\NN,\; \PP\left(\prox(\overline{\gamma}_n)\le\alpha n\right) \le C\exp(-\beta n), \label{eqn:ldevprox}\\ 
			\forall\alpha < \sigma(\nu),\;\exists C,\beta > 0,\;\forall n\in\NN,\; \PP\left(\dist(E^+(\overline{\gamma}_n), l^\infty)\ge \exp(-\alpha n) \right) \le C\exp(-\beta n). \label{eq:lim-eigen}
		\end{gather}
	\end{Th}
	
	\begin{proof}
		Let $0 \le \eps \le 1$ and $\tilde\nu_s$ be as in Corollary \ref{cor:schottky} applied to $\nu$ with $\rho = \frac{1}{4}$ and $K = 10$.
		Let $\tilde\mu$ be as in Theorem \ref{th:pivot-extract} applied to $\nu$ and $\tilde\nu_s$ with $S = \mathbf{supp}(\Pi_*\tilde\nu_s)$ and $\AA = \AA^\eps$. 
		Let $(\widetilde{\gamma}^w_k) \sim \tilde\mu$. 
		To all integer $n \in\NN$, we associate $q_n =: \max\{q\in\NN\,|\, \overline{w}_{2q} \le n\}$ and $c_n$ the smallest integer such that:
		\begin{equation*}
			\gamma^w_{2q_n-2c_n-1}\AA \left(\gamma_{\overline{w}_{2q_n-2c_n}}\cdots\gamma_{n-1}\gamma_0\cdots\gamma_{\overline{w}_{2c_n}}\right)\quad\text{and}\quad \left(\widetilde\gamma_{\overline{w}_{2q_n-2c_n-2}}\cdots\gamma_{n-1}\gamma_0\cdots\gamma_{\overline{w}_{2c_n}}\right) \AA \gamma^w_{2c_n}
		\end{equation*}
		or $2c_n \ge q_n$. 
		By Lemma \ref{lem:c-piv}, $c_n$ has a bounded exponential moment that does not depend on $n$.
		Moreover $\gamma^w_{2q_n-2c_n-3}\widetilde{\AA}^S \widetilde\gamma^w_{2q_n-2c_n-2}$ and $\gamma^w_{2q_n-2c_n-2}\AA \gamma^w_{2q_n-2c_n-1}$. 
		Therefore, we have
		\begin{equation*}
			\gamma^w_{2q_n-2c_n-3}\widetilde{\AA}^S\left(\widetilde\gamma_{\overline{w}_{2q_n-2c_n-2}}, \dots,\gamma_{n-1}, \gamma_0 , \dots, \gamma_{\overline{w}_{2c_n}}\right)
		\end{equation*} 
		By Proposition \ref{prop:pivot-is-aligned}, we have:
		\begin{equation*}
			\gamma^w_{2q_n-2c_n-3} \AA^\frac{\eps}{2} \left(\widetilde\gamma_{\overline{w}_{2q_n-2c_n -2}} \cdots \gamma_{n-1} \gamma_0 \cdots \gamma_{\overline{w}_{2c_n}}\right).
		\end{equation*}
		By definition of $c_n$ and by Theorem \ref{th:pivot-extract} and Proposition \ref{prop:pivot-is-aligned} applied to $(\gamma^w_k)_{2c_n < k < 2q_n -2c_n}$, which satisfies \eqref{item:atilde} in Theorem \ref{th:pivot-extract}, we also have:
		\begin{equation*}
			\left(\widetilde\gamma_{\overline{w}_{2q_n-2c_n -2}} \cdots \gamma_{n-1} \gamma_0 \cdots \gamma_{\overline{w}_{2c_n}}\right) \AA^\eps \gamma^w_{2c_n +1} \AA^\frac{\eps}{2} \gamma^w_{2c_n +2} \AA^\eps \cdots \AA^\frac{\eps}{2} \gamma^w_{2q_n-2c_n-4} \AA^\eps \gamma^w_{2q_n-2c_n-3}
		\end{equation*}
		Moreover, each of these matrices have a squeeze coefficient larger than $4 |\log(\eps)| + 7 |\log(2)|$ by Proposition \ref{prop:pivot-is-aligned} again.
		Let $a_n := \overline{w}_{2q_n-2c_n-3}$ and let $h_n := \gamma_{a_n}\cdots\gamma_{n-1}\gamma_0\cdots \gamma_{a_n-1}$. 
		Then by Lemma \ref{lem:alipart} we have $h_n \AA^\frac{\eps}{4} h_n$ and $\overline{\gamma}_{a_n}\AA^\frac{\eps}{4} h_n$ when $2c_n + 2 \le q_n$. 
		Hence by Lemma \ref{lem:eigen-align}, we have:
		\begin{equation}
			\prox(h_n) \ge \sqz(h_n) - 2|\log(\eps)| - 6\log(2)
		\end{equation}
		Moreover, we have $\left(\gamma^w_{2c_n +1} \cdots \gamma^w_{2q_n-2c_n-3}\right) \AA^\frac{\eps}{4}\left(\gamma_{\overline{w}_{2q_n-2c_n-2}} \cdots \gamma_{n-1} \gamma^w_0 \cdots \gamma^w_{2c_n}\right)$, so by \eqref{eqn:lenali} in Lemma \ref{lem:c-prod}, we have:
		\begin{equation}\label{eq:sqz-hn}
			\sqz(h_n) \ge \sqz\left(\gamma^w_{2c_n +1} \cdots \gamma^w_{2q_n-2c_n-3}\right) - 2|\log(\eps)| - 4\log(2).
		\end{equation}
		
		Now we claim that $\left(\sqz\left(\gamma^w_{2c_n +1} \cdots \gamma^w_{2q_n-2c_n-3}\right)\right)_{n\in\NN}$ satisfies large deviations inequalities below the speed $\sigma(\nu)$. Note that this implies that we have \eqref{eqn:ldevprox} because $\prox(h_n) = \prox(\overline{\gamma}_n)$ by conjugation.
		Let $\alpha < \sigma$ and let $0 < \delta < 1/2$ be such that $\frac{\alpha}{1-2\delta} < \sigma$. By Theorem \ref{th:ldev-sqz}, there are constants $C, \beta > 0$ such that:
		\begin{equation*}
			\forall i < j,\; \PP\left(\sqz(\gamma_i\cdots \gamma_{j-1}) \le (j-i) \frac{\alpha}{1-2\delta}\right) \le C \exp(-\beta (j-i)).
		\end{equation*}
		Write $C' := C (1-\exp(-\beta))^{-2}$ and $\beta' := \beta(1-2\delta) > 0$, then we have:
		\begin{align*}
			\PP\left(\exists i \le \delta n,\, \exists j \ge n- \delta n,\;\sqz(\gamma_i\cdots \gamma_{j-1}) \le (j-i) \frac{\alpha}{1-2\delta}\right) 
			& \le \sum_{i = 0}^{\lfloor \delta n \rfloor}\sum_{j = \lceil n-\delta n \rceil}^{n} C \exp(-\beta (j-i))\\
			& \le \sum_{i = -\infty}^{\lfloor \delta n \rfloor}\sum_{j = \lceil n-\delta n \rceil}^{+\infty} C \exp(-\beta (j-i))\\
			& \le  C' \exp(-\beta' n).
		\end{align*}
		Note that, if we take $i = \overline{w}_{2c_n+1}$ and $j = \overline{w}_{2q_n - 2c_n - 2}$, then $\gamma_i\cdots \gamma_{j-1} = \gamma^w_{2c_n +1} \cdots \gamma^w_{2q_n-2c_n-3}$. Hence we have:
		\begin{multline}\label{eq:dom-sqz-prox}
			\PP\left(\sqz\left(\gamma^w_{2c_n +1} \cdots \gamma^w_{2q_n-2c_n-3}\right) \le \alpha n\right) \le \PP\left(\delta n < \overline{w}_{2c_n+1}\right) + \PP\left(\overline{w}_{2q_n - 2c_n - 2} < (1-\delta) n\right) \\
			+ \PP\left(\exists i \le \delta n,\,\exists j \ge n-\delta n,\;\sqz(\gamma_i\cdots \gamma_{j-1}) \le (j-i) \frac{\alpha}{1-2\delta}\right)
		\end{multline}
		Moreover $\overline{w}_{2c_n+1}$ has finite exponential moment so $\left(\PP\left(\delta n < \overline{w}_{2c_n+1}\right)\right)_{n\in\NN}$ decreases exponentially fast
		Moreover, by \eqref{compo:compo} applied to $(\overline{w}_m)$ and $(q_n)$ and \eqref{compo:shift} applied to $(c_n)$ in Lemma \ref{lem:proba:ldevcompo}, the random sequence $\left(\overline{w}_{2q_n - 2c_n - 3}\right)_{n\in\NN}$ satisfies large deviations inequalities below the speed $1$ so $\left(\PP\left(\overline{w}_{2q_n - 2c_n - 2} < (1-\delta) n\right)\right)_{n\in\NN}$ decreases exponentially fast.
		Hence, by \eqref{eq:dom-sqz-prox}, the sequence $\left(\PP\left(\sqz\left(\gamma^w_{2c_n +1} \cdots \gamma^w_{2q_n-2c_n-3}\right)\le \alpha n\right)\right)_{n\in\NN}$ decreases exponentially fast in $n$, which proves the claim so we have \eqref{eqn:ldevprox}.
		
		Now we prove \eqref{eq:lim-eigen}. By Lemma \ref{lem:eigen-align} applied to $h_n \AA^\eps h_n$, we have:
		\begin{equation}
			\forall u\in U^{\frac{\eps}{4}}(h_n)\setminus\{0\}, \;\dist\left(E^+(h_n), [u]\right) \le \frac{16}{\eps} \exp(-\sqz(h_n)).
		\end{equation}
		Let $e_n$ be a random non-zero vector such that $e_n \in E^+(h_n)$ when $2c_n + 2 \le q_n$.
		By \eqref{eq:sqz-hn} and by the above reasoning, $\left(\sqz(h_n)\right)_n$ satisfies large deviations inequalities below the speed $\sigma(\nu) > 0$. Hence there exist constants $C, \beta > 0$ such that:
		\begin{equation}
			1 - C \exp(-\beta n) \le \PP\left(\forall u\in U^{\frac{\eps}{4}}(h_n)\setminus\{0\} , \;\dist\left(E^+(h_n), [u]\right) \le \frac{\eps}{8}\right) \le \PP\left(\overline{\gamma}_{a_n} \AA^\frac{\eps}{8} e_n\right).
		\end{equation}
		Moreover, By Lemma \ref{lem:herali} and by the above reasoning, we have $\overline{\gamma}_{a_n} \AA^\frac{\eps}{4} h_n$. 
		Hence by lemma \ref{lem:product-is-lipschitz}, we have:
		\begin{equation*}
			\PP\left(\forall u\in U^{\frac{\eps}{4}}(h_n)\setminus\{0\} , \;\dist\left(E^+(h_n), [u]\right) \le \frac{\eps}{8}\right) \le \PP\left(\overline{\gamma}_{a_n} \AA^\frac{\eps}{8} e_n\right)
		\end{equation*}
		Moreover, when $\overline{\gamma}_{a_n} \AA^\frac{\eps}{8} e_n$, then by Lemma \ref{lem:c-prod}, we have:
		\begin{equation*}
			\forall u\in U^{1}(\overline{\gamma}_{a_n})\setminus\{0\}, \; \dist\left([\overline{\gamma}_{a_n}e_n], [u]\right)\le \frac{8}{\eps}\exp(-\sqz(\overline{\gamma}_{a_n})).
		\end{equation*}
		Then by Corollary \ref{cor:limit-line} and by triangular inequality, we have:
		\begin{equation*}
			\dist\left([\overline{\gamma}_{a_n}e_n], l^\infty(\gamma)\right)\le \frac{12}{\eps}\exp(-\sqz(\overline{\gamma}_{a_n})).
		\end{equation*}
		By conjugacy, we have $\overline{\gamma}_{a_n}e_n \KK = E^+(\overline{\gamma}_n)$. Moreover, the random sequence $(a_n)$ satisfies large deviations inequalities below the speed $1$ so by Theorem \ref{th:ldev-sqz} and by \eqref{compo:compo} in Lemma \ref{lem:proba:ldevcompo}, the random sequence $\left(\sqz(\overline{\gamma}_{a_n})\right)_n$ satisfies large deviations inequalities below the speed $\sigma(\nu)$. This proves \eqref{eq:lim-eigen}.
	\end{proof}
	
	\begin{proof}[Proof of Theorems \ref{th:escspeed} and \ref{th:cvspeed}]
		Let $E$ be a Euclidean space.
		Let $\nu$ be a strongly irreducible and proximal probability distribution over $\mathrm{End}(E)$.
		Let $(\gamma_n) \sim \nu^{\otimes\NN}$.
		Let $\sigma(\nu) := \lim \frac{\sqz(\overline{\gamma}_n)}{n}$ be as in Lemma \ref{lem:conv}.
		Then $\sigma(\nu) > 0$ by Theorem \ref{th:ldev-sqz}.
		Then \eqref{eqn:ldsqz} in Theorem \ref{th:escspeed} is a direct consequence of \eqref{eq:ldev-sqz} in Theorem \ref{th:ldev-sqz} and \eqref{eqn:ldevprox} in Theorem \ref{th:ldevprox}. 
		The above proves Theorem \ref{th:escspeed}.
		Moreover \eqref{eqn:erc} is \eqref{eq:limit-v} and \eqref{eqn:lim-E+} is \eqref{eq:lim-eigen} which proves Theorem \ref{th:cvspeed}.
	\end{proof}

\subsection{Limit flag for totally strongly irreducible distributions}\label{sec:flag}

	Now we can give the following corollary which is a reformulation of the former results, written in a way that should remind of Oseledets' multiplicative ergodic Theorem. 
	
	Let $E$ be a Euclidean vector space and let $\nu$ be a probability distribution on $\mathrm{End}(E)$. We say that $\nu$ is totally strongly irreducible if for all $k \in\{1, \dots, \dim(E)-1\}$ the measure $\bigwedge^k_*\nu$ is strongly irreducible. 
	
	Let $\nu$ be a strongly irreducible probability distribution and let $(\gamma_n) \sim \nu^{\otimes\NN}$. If $\nu$ is not proximal, we define $\sigma(\nu) := 0$, note that then by \ref{lem:exiprox}, we have $\frac{\sqz(\overline{\gamma}_n)}{n} \to 0$ almost surely. 
	If $\nu$ is proximal, we define $\sigma(\nu)$ as in Theorem \ref{th:ldev-sqz}.

	\begin{Def}\label{def:Theta}
		Let $\nu$ be a distribution on $\mathrm{End}(E)$ that is totally strongly irreducible. For all $1\le j \le \underline{\rank}(\nu)$, we define $\sigma_j(\nu) := \sigma\left(\bigwedge^j_*\nu\right)$.
		We define:
		\begin{equation}
			\Theta(\nu) := \left\{1\le j \le \underline{\rank}(\nu) \,\middle|\, \sigma_j(\nu)\neq 0\right\}
		\end{equation}
	\end{Def}
	
	Given $h$ a matrix and $j \le \rank(h)$, we define $\sqz_j(h) := \log\left(\frac{\|\bigwedge^{j} h\|^2}{\|\bigwedge^{j-1} h\|\|\bigwedge^{j+1} h\|}\right) = \sqz(\bigwedge^j h)$, for $j > \rank(h)$, we use the convention $\sqz_j(h) = 0$ and for all $j \ge 1$, we define: $\prox_j(h) := \prox\left(\bigwedge^j h\right) = \lim_{n\to + \infty} \frac{\sqz_j(h^n)}{n}$.
	
	\begin{Th}[Convergence of the Cartan projection with large deviations]\label{th:flag}
		Let $E$ be a Euclidean spaces and let $\nu$ be a totally strongly irreducible probability distribution on $\mathrm{End}(E)$ of rank at least $\dim(E)-1$. Let $(\gamma_n)\sim\nu^{\otimes \NN}$. 
		For all $1 \le j \le d-1$ we have almost surely $\frac{\sqz_j(\overline{\gamma}_n)}{n} \to \sigma_j(\nu)$. 
		Moreover, for all $j \in \Theta(\nu)$ and for all $\alpha_j < \sigma_j(\nu)$, there exist constants $C,\beta >0$ such that:
		\begin{gather}\label{eqn:limcartanproj}
			\forall n\in\NN,\; \PP\left(\sqz_j(\overline{\gamma}_n)\le\alpha_j n\right) \le C\exp(-\beta n). \\
			\forall n\in\NN,\; \PP\left(\prox_j(\overline{\gamma}_n) \le \alpha_j n\right) \le C\exp(-\beta n). \label{eqn:limlox}
		\end{gather}
	\end{Th}
	
	\begin{proof}
		If we assume that $j \in \Theta(\nu)$, then \eqref{eqn:limcartanproj} is a reformulation of Theorem \ref{th:ldev-sqz} for $\bigwedge^j_*\nu$ and \ref{eqn:limlox} is a reformulation of Theorem \ref{th:ldevprox}. 
		Otherwise $\sigma_j(\nu) = 0$. We know from Lemma \ref{lem:boundary} that there is a constant $B$ such that $\sqz_j(\overline{\gamma}_n) \le B$ almost surely and for all $n$ so $\frac{\sqz_j(\overline{\gamma}_n)}{n} \to 0 = \sigma_j(\nu)$.
	\end{proof}
	
	Let $E$ be a Euclidean space of dimension $d \ge 2$.
	We denote by $\mathrm{Gr}(E)$ the set of vector subspaces of $E$.
	Given $0 \le k \le d$, we denote by $\mathrm{Gr}_k(E)$ the set of subspaces of $E$ that have dimension $k$ \ie $(\mathrm{Gr}_k(E))_{1\le k \le d}$ are the level sets of the function $\dim : \mathrm{Gr}(E) \to \{0,\dots, d\}$.
	Note that for all $k$, the set $\mathrm{Gr}_k(E)$ naturally embeds into $\mathbf{P}(\bigwedge^k E)$ by the map $V \mapsto [v_1 \wedge \cdots \wedge v_k]$ for $v_1 \wedge \cdots \wedge v_k$ any basis of $V$. 
	Note that the image of $\mathrm{Gr}_k(E)$ by this embedding is a compact subset, we will abusively denote by $\mathrm{Gr}_k(E)$ the image of $\mathrm{Gr}_k(E)$ in $\mathbf{P}(\bigwedge^k E)$.
	We write $\dist$ for the distance map on $\mathrm{Gr}_k(E)$ pulled back from the distance on $\mathbf{P}(\bigwedge^k E)$ associated to the Euclidean metric on $\bigwedge^k E$.
	
	We call flag in $E$ a totally ordered set of subspaces of $E$ \ie for all $V, W \in F$, we have $V \subset W$ or $W \subset V$. 
	We write $\mathrm{Fl}(E)$ for the space of flags in $E$. 
	Given $\Theta \subset \{1,\dots, d-1\}$, we denote by $\mathrm{Fl}_{\Theta}(E)$ the space of flags $F \in \mathrm{Fl}(E)$ such that $\dim(F) = \Theta$.
	Given $k \in \Theta \subset \{1,\dots, d-1\}$ and given $F \in \mathrm{Fl}(E)$, we write $F_k$ for the single element of the set $F \cap \mathrm{Gr}_k(E)$.

	Note that $\mathrm{End}(E)$ naturally acts on the left on $\mathrm{Fl}(E)$.
	Moreover, for all $\Theta \subset \{1,\dots, d-1\}$, the group $\mathrm{GL}(E)$ acts continuously on $\mathrm{Fl}_{\Theta}(E)$.
	We remind that we denote by $T$ the Bernoulli shift.
	
	\begin{Def}
		Let $E$ be a Euclidean vector space and let $\Theta \subset \{1, \dots, \dim(E)\}$
		We define:
		\begin{gather*}
			\Omega'_{\Theta}(E) := \bigcap_{k\in\Theta \setminus \{0\}} \left(\left(\Wedge^k\right)^{\otimes\NN}\right)^{-1}\Omega\left(\Wedge^k E\right),\\
			\Omega_\Theta(E) := \left\{\gamma\in\Omega'_{\Theta}(E)\,\middle|\,\forall k \in\Theta,\; l^\infty\circ \left(\Wedge^k\right)^{\otimes\NN} (\gamma) \in\mathrm{Gr}_k(E)\right\}.
		\end{gather*} 
		We also define the $T$-equivariant measurable map:
		\begin{equation*}
			F^\infty_\Theta = (F^\infty_k)_{k\in\Theta}: \Omega_\Theta(E)\longrightarrow \mathrm{Fl}_{\Theta}(E),
		\end{equation*}
		with $F_k^\infty = l^\infty \circ (\Wedge^k)^{\otimes\NN}$ for all $k$.
	\end{Def}
	
	Using the Cartan decomposition on the exterior product, we can show that in fact $\Omega'_{\Theta}(E) = \Omega_{\Theta}(E)$ but that is not the purpose of this article.
	
	\begin{Th}[Convergence to the limit flag]\label{th:oseledets}
		Let $E$ be a Euclidean vector space.
		Let $d := \dim(E)$, assume that $ d \ge 2$.
		Let $\nu$ be a totally strongly irreducible probability distribution on $\mathrm{GL}(E)$. 
		Let $\Theta := \Theta(\nu)$ and let $\gamma \sim \nu^{\otimes\NN}$.
		Then $\gamma \in \Omega_\Theta(E)$ almost surely.
		Let $(\alpha_k)_{k\in\Theta} \in \prod_{k\in\Theta}(0, \sigma_k(\nu))$ be a non-random family of real numbers.
		Then there exist constants $C, \beta > 0$ such that for all non-random flag $F \in\mathrm{Fl}_{\Theta}$, we have:
		\begin{equation}\label{eq:lim-flag}
			\PP\left(\exists k \in\Theta,\;\dist\left(F^\infty_k(\gamma),\overline\gamma_n F_k\right)\ge \exp(-\alpha_k n) \right) \le C \exp(-\beta n).
		\end{equation}
	\end{Th}
	
	\begin{proof}
		This is a reformulation of Theorem \ref{th:limit-uv} in terms of flags.		
		Let $\nu$ be a totally strongly irreducible probability measure on $\mathrm{GL}(E)$ and let $k\in \Theta(\nu)$. 
		Then $\bigwedge^k_*\nu$ is strongly irreducible and proximal.
		Let $(\gamma_n) \sim \nu^{\otimes \NN}$. 
		By Theorem \ref{th:limit-uv} applied to $(\bigwedge^k\gamma_n)_{n\in\NN}$, we have $(\bigwedge^k)^{\otimes\NN}(\gamma)\in \Omega\left(\Wedge^k E\right)$. 
		Then we claim that $l^\infty\circ(\bigwedge^k)^{\otimes\NN}(\gamma) \in \mathrm{Gr}_k(E)$ almost surely.
		Hence $ \gamma \in \Omega'_{\Theta}(E)$ almost surely.
		
		Let $(v_1, \dots, v_k)\in E^k$ be a non-random free family. 
		By Theorem \ref{th:limit-uv}, we have : $\bigwedge^k \overline\gamma_n(v_1, \dots, v_k) \underset{n \to +\infty}{\longrightarrow} l^\infty\circ(\bigwedge^k)^{\otimes\NN}(\gamma)$ almost surely. Moreover, for all $n \in\NN$, we have $\bigwedge^k \overline\gamma_n[v_1\wedge \cdots \wedge v_k] = [\overline\gamma_n v_1 \wedge \cdots \wedge \overline\gamma_n v_k] \in \mathrm{Gr}_k(E)$, and $\mathrm{Gr}_k(E)$ is closed in $\mathbf{P}(\bigwedge^k E)$.  Hence $l^\infty\circ(\bigwedge^k)^{\otimes\NN}(\gamma) \in\mathrm{Gr}_k(E)$.
		Hence $ \gamma \in \Omega_{\Theta}(E)$ almost surely.
		To conclude, note that \eqref{eq:lim-flag} is simply the reformulation of \eqref{eq:limit-v} for all the measures $\bigwedge^k_*\nu$ with $k \in \Theta$.
	\end{proof}
	
	Note that in the case of $\nu$ an totally strongly irreducible probability distribution on $\mathrm{SL}(E)$, one can actually take the pivotal extraction to be aligned in all Cartan projections. With a correct adaptation of the works of~\cite{CFFT22}, one should be able to prove the following.
	
	\begin{Conj}[Poisson boundary]\label{conj:poisson}
		Let $E$ be a Euclidean space od dimension $d \ge 3$.
		Let $\nu$ be an totally strongly irreducible probability distribution supported on a discrete subgroup of $\mathrm{SL}(E)$. Assume that $\nu$ has finite entropy then the Poisson boundary of $\nu$ is isomorphic to $\mathrm{Fl}_{\Theta(\nu)}(E)$ endowed with the $\nu$-stationary probability distribution $F^\infty_*\nu^{\otimes\NN}$.
	\end{Conj}
	
	It follows directly from~\cite{CFFT22} when $E = \RR^2$ because any discrete subgroup in $\mathrm{SL}(\RR^2)$ is either non-elementary hyperbolic and therefore we can apply~\cite{CFFT22}, or virtually cyclic and therefore not strongly irreducible or finite.

	\subsection{Law of large numbers for the coefficients and for the spectral radius}\label{sec:lln}
	
	We remind that in Corollary \ref{cor:schottky}, we have shown that given $\nu$ a strongly irreducible and proximal probability distribution on $\mathrm{GL}(E)$ and given $0 < \rho < \frac{1}{5}$, there is an integer $m\in\NN$, a constant $0< \eps <1$ and a probability distribution $\tilde\nu_s$ that is absolutely continuous with respect to $\nu^{\otimes m}$, has compact support and whose product is $\rho$-Schottky for $\AA^\eps$.
	
	We use the notations of Definitions \ref{def:trunking} and \ref{def:corconv} for the trunking $\lceil \cdot \rceil$ and the coarse convolution $\cdot^{\uparrow k}$.
	We formulate Theorems \ref{th:llnc} and \ref{th:llnr} in the following technical way to explicit the dependency of the constants $C$ and $\beta$ in Theorem \ref{th:slln} in terms of $\nu$. 
	In particular, we can note that the lower bound of all $\beta$, so that Theorem \ref{th:slln} is satisfied (for $C = \beta^{-1}$), is given by a function of $\nu$ that is lower semi-continuous for the weak-$*$ topology on the space of strongly irreducible and proximal probability distributions on $\mathrm{GL}(E)$.

	\begin{Th}[Strong law of large numbers for the coefficients]\label{th:llnc}
		Let $E$ be a Euclidean vector space and let $\nu$ be a probability measure on $\mathrm{GL}(E)$. Let $(\gamma_n)\sim\nu^{\otimes\NN}$. 
		Let $\alpha, \eps \in (0,1)$, let $\rho \in (0,1/5)$ and let $\tilde\nu_s$ be a compactly supported probability distribution on $\mathrm{GL}(E)^m$ such that $\Pi_*\tilde\nu_s$ is $\rho$-Schottky for $\AA^\eps$ and supported on the set $\{\sqz \ge 10 |\log(\eps/2)|\}$ and $\alpha \tilde\nu_s \le \nu^{\otimes m}$. Let $B := \max_{k} \max N\circ\chi_k^m(\mathbf{supp}(\tilde\nu_s))$. 
		Then there exist constants $C,\beta > 0$ that depend only on $\alpha, \rho, m$ and such that:
		\begin{multline}\label{eqn:llnc}
			\forall f \in E^*\setminus\{0\},\;\forall v \in E\setminus\{0\}, \; \forall n\in\NN, \; \forall t\ge 0,\;
			\PP\left(\log\frac{\|f\|\|\overline{\gamma}_n\|\|v\|}{| f \overline{\gamma}_{n} v|} > t\right) \\
			\le C\exp(-\beta n) + \sum_{k=1}^{\infty}  C\exp\left(-\beta k\right) \left\lceil\frac{B \vee N_*\nu}{1-\alpha}\right\rceil^{\uparrow k} \left({t - 2|\log(\eps)| - 3\log(2)},+\infty\right). 
		\end{multline} 
	\end{Th}
	
	\begin{proof}
		Let $0 < \eps \le 1$ and $\tilde\nu_s$ be as in Corollary \ref{cor:schottky} for $K = 10$ and $\rho = \frac{1}{4}$.
		Let $\tilde\mu$ be as in Theorem \ref{th:pivot-extract} for $\Gamma := \mathrm{GL}(E)$, for $\AA = \AA^\eps$ and $S = \{\sqz \ge 10 |\log(\eps/2)|\}$. 
		Let $(\widetilde\gamma^w)\sim \tilde\mu$. 
		By Proposition \ref{prop:pivot-is-aligned}, we have $\gamma^w_{2k} \AA^\eps \gamma^w_{2k+1}\AA^\frac{\eps}{2} \gamma^w_{2k+2}$ for all $k \in\NN$.
		
		Let $n \in\NN$ be fixed. 
		Let $q_n := \max\{q\in\NN\,|\,\overline{w}_{2q} \le n\}$. 
		Then by \eqref{eq:tjr-scho} in Theorem \ref{th:pivot-extract}, the sequence $f\gamma^w_0,\gamma^w_1, \dots, \gamma^w_{2q_n - 1},\gamma_{\overline{w}_{2q_n}} \cdots\gamma_{n-1} v$ is $\rho$-ping-pong. 
		By Lemma \ref{lem:r-piv} applied to that sequence, we construct a random integer $r_n \sim \mathcal{G}_\rho$ that is independent of $\left(\widetilde\gamma^w_{2k}\right)_{k\in\NN}$ and such that $r_n \ge q_n$ or:
		\begin{equation*}
			\gamma^w_{2q_n - 2r_n - 1}\AA^\eps \left(\gamma_{\overline{w}_{2q_n -2r_n}} \cdots \gamma_{n-1} v\right).
		\end{equation*}
		By the transpose of Lemma \ref{lem:r-piv}, we construct a random integer $l \sim \mathcal{G}_\rho$ that is independent of $\left(\widetilde\gamma^w_{2k}\right)_{k\in\NN}$ and such that :
		\begin{equation*}
			\left(f \gamma^w_0\cdots \gamma^w_{2l}\right) \AA^\eps \gamma^w_{2l+1}.
		\end{equation*}
		Note that if $l + r_n < q_n$, then:
		\begin{equation*}
			\left(f \gamma^w_0\cdots \gamma^w_{2l}\right) \AA^\eps \gamma^w_{2l + 1}\AA^\frac{\eps}{2}\cdots \AA^\eps\gamma^w_{2q_n - 2r - 1}\AA^\eps \left(\gamma_{\overline{w}_{2q_n - 2r_n}} \cdots \gamma_{n-1} v\right).
		\end{equation*}
		Hence, by lemma \ref{lem:zouli-alignemont}, we have:
		\begin{equation*}
			\left(f \gamma^w_0\cdots \gamma^w_{2l}\right) \AA^\frac{\eps}{2} \left(\gamma^w_{2l+1}\cdots\gamma^w_{2q_n-2r-1}\right) \AA^\frac{\eps}{2} \left(\gamma_{\overline{w}_{2q_n -2r}} \cdots \gamma_{n-1} v\right).
		\end{equation*}
		Hence by Lemma \ref{lem:triple-ali}:
		\begin{equation*}
			|f\overline\gamma_n v|\ge \frac{\eps^2}{8}\left\|f \gamma^w_0\cdots \gamma^w_{2l}\right\| \left\| \gamma^w_{2l+1}\cdots\gamma^w_{2q_n-2r_n-1}\right\| \left\|\gamma_{\overline{w}_{2q_n -2r_n}} \cdots \gamma_{n-1} v\right\|.
		\end{equation*}
		Moreover by sub-multiplicativity:
		\begin{equation*}
			\|\overline\gamma_n\|\le \left\|\gamma^w_0\cdots \gamma^w_{2l}\right\| \left\| \gamma^w_{2l+1}\cdots\gamma^w_{2q_n-2r_n-1}\right\| \left\| \gamma_{\overline{w}_{2q_n -2r_n}} \cdots \gamma_{n-1}\right\|.
		\end{equation*}
		By definition of $N$, and by sub-additivity, we have:
		\begin{align*}
			\log\left\|\gamma^w_0\cdots \gamma^w_{2l}\right\| + \log\|f\| - \log\left\|f \gamma^w_0\cdots \gamma^w_{2l}\right\| 
			& \le \log\left(\left\|\gamma^w_0\right\| \cdots \left\|\gamma^w_{2l}\right\|\right) - \log\left(\left\|{\gamma^w_0}^{-1}\right\| \cdots \left\|{\gamma^w_{2l}}^{-1}\right\|\right) \\
			& \le N(\gamma^w_0\cdots \gamma^w_{2l}) \le \sum_{k=0}^{\overline{w}_{2l+1}-1} N(\gamma_k)
		\end{align*}
		and by the same argument:
		\begin{equation*}
			\log \left\|\gamma_{\overline{w}_{2q_n -2r_n}} \cdots \gamma_{n-1}\right\| + \log\|v\| - \log \left\|\gamma_{\overline{w}_{2q_n -2r_n}} \cdots \gamma_{n-1} v\right\|  \le \sum_{k = \overline{w}_{2q_n - 2 r_n}}^{n-1} N(\gamma_k).
		\end{equation*}
		Therefore:
		\begin{equation*}
			\log\|\overline\gamma_n\| - \log |f\overline\gamma_n v| \le \sum_{k=0}^{\overline{w}_{2l+1}-1} N(\gamma_k) + \sum_{k = \overline{w}_{2q_n - r_n}}^{n-1} N(\gamma_k) + 2|\log(\eps)| + 3\log(2).
		\end{equation*}
		Moreover, for all $t \ge B$, and for all $k \in\NN$, by \eqref{eq:densite} in Theorem \ref{th:pivot-extract} with $A = \{N > t\}$, one has:
		\begin{equation*}
			\PP(N(\gamma_k) > t\,|\,(w_j)_{j\in\NN}) \le \frac{N_*\nu(t,+\infty)}{1-\alpha}.
		\end{equation*}
		Note that, $q_n$ only depends on $(w_j)_{j\in\NN}$, moreover $l$ and $r_n$ are independent of $\left(\widetilde{\gamma}^w_{2j}\right)_{j\in\NN}$.
		Note also that when the sequence $(w_j)_{j\in\NN}$ is so that the index $k$ appears in an oddly indexed group (\ie there exists $j \in\NN$ such that $\overline{w}_{2j+1} \le k < \overline{w}_{2j+2}$), then $N(\gamma_k) \le B$ so $\PP(N(\gamma_k) > t\,|\,(w_j)_{j\in\NN}) = 0$ on that set and trivially, we also have $\PP(N(\gamma_k) > t\,|\,(w_j)_{j\in\NN}, l, r_n) = 0$ on that set and for all values of $l$ and $r_n$.
		Therefore, we have:
		\begin{equation*}
			\PP(N(\gamma_k) > t\,|\,l, q_n, r_n) \le \frac{N_*\nu(t,+\infty)}{1-\alpha}.
		\end{equation*}
		By Corollary \ref{cor:curexp}, $n -\overline{w}_{2q_n}$ has a bounded exponential moment, with a bound that depends only on $(\alpha, \rho, m)$ and not on $n$.
		By Lemma \ref{lem:sumexp}, the random variables $n - \overline{w}_{2q_n - r_n}$ and $\overline{w}_{2l+1}$ both have bounded exponential moment, with a bound that depends only on $(\alpha, \rho, m)$. 
		In other words, there are constants $C, \beta > 0$ that only depend on $(\alpha, \rho, m)$ by construction and such that:
		\begin{equation*}
			\forall k\in\NN,\; \PP\left(n - \overline{w}_{2q_n - 2r_n} + \overline{w}_{2l+1} = k \cap r_n \le q_n\right) \le C \exp(-\beta k).
		\end{equation*}
		Moreover by Lemma \ref{lem:sumexp}, the random variable $\overline{w}_{2l+2r_n}$ also has bounded exponential moment so we may also assume that:
		\begin{equation*}
			\PP(l+r_n \ge q_n) = \PP\left(n \le \overline{w}_{2l+2r_n}\right) \le C \exp(-\beta n).
		\end{equation*}
		Hence, by Lemma \ref{lem:descrisum}, we have for all $t \ge 0$:
		\begin{equation*}
			\PP\left(\log\frac{\|f\|\|\overline{\gamma}_n\|\|v\|}{| f \overline{\gamma}_{n} v|} > t\,\middle|\,l+r_n < q_n\right) 
			\le\sum_{k=1}^{n}  C\exp\left(-\beta k\right) \left\lceil\frac{B \vee N_*\nu}{1-\alpha}\right\rceil^{\uparrow k} \left({t - 2|\log(\eps)| - 3\log(2)},+\infty\right).
		\end{equation*}
		This proves \eqref{eqn:llnc}.
	\end{proof}

	\begin{Th}[Strong law of large numbers for the spectral gap]\label{th:llnr}
		Let $E$ be a Euclidean vector space and let $\nu$ be a probability measure on $\mathrm{GL}(E)$.
		Let $(\gamma_n)\sim\nu^{\otimes\NN}$. 
		Let $\alpha, \eps \in (0,1)$, let $\rho \in (0,1/5)$ and let $\tilde\nu_s$ be a compactly supported probability distribution on $\mathrm{GL}(E)^m$ such that $\Pi_*\tilde\nu_s$ is $\rho$-Schottky for $\AA^\eps$ and supported on the set $\{\sqz \ge 10 |\log(\eps/2)|\}$ and $\alpha \tilde\nu_s \le \nu^{\otimes m}$. Let $B := \max_{k} \max N\circ\chi_k^m(\mathbf{supp}(\tilde\nu_s))$. 
		Then there exist constants $C,\beta > 0$ that depend only on $\alpha, \rho, m$ and such that:
		\begin{multline}\label{eqn:llnr}
			\forall n\in\NN, \; \forall t\ge 0,\;
			\PP\left(\log\frac{\|\overline{\gamma}_n\|}{\rho_1(\overline{\gamma}_{n})} > t\right) \\
			\le \sum_{k=1}^{\infty}  C\exp\left(-\beta k\right) \left\lceil\frac{B \vee N_*\nu}{1-\alpha}\right\rceil^{\uparrow k} \left({t - 2|\log(\eps)| - 5\log(2)},+\infty\right). 
		\end{multline} 
	\end{Th}
	
	\begin{proof}
		Let $0 < \eps \le 1$ and $\tilde\nu_s$ be as in Corollary \ref{cor:schottky} for $K = 10$ and $\rho = \frac{1}{4}$.
		Let $\tilde\mu$ be as in Theorem \ref{th:pivot-extract} for $\Gamma := \mathrm{GL}(E)$, for $\AA = \AA^\eps$ and $S = \{\sqz \ge 10 |\log(\eps/2)|\}$. 
		Let $(\widetilde\gamma^w)\sim \tilde\mu$. 
		Let $n \in\NN$.
		Let $q_n : = \max\{k\in\NN\,|\,\overline{w}_{2k} \le n\}$. Let $c_n$ be as in Lemma \ref{lem:c-piv} applied to the pivotal sequence $(\gamma^w_0, \gamma^w_1, \dots, \gamma^w_{2q_n - 1},\gamma_{\overline{w}_{2q_n}}\cdots\gamma_{n-1})$. 
		Then $c_n$ has a finite exponential moment and is independent of $(\gamma^w_0, \gamma^w_2, \dots, \gamma^w_{2q_n - 2},\gamma_{\overline{w}_{2q_n}}\cdots\gamma_{n-1})$.
		By the same argument as in the proof of Theorem \ref{th:ldevprox}, using \eqref{eq:ev-sv} in Lemma \ref{lem:eigen-align}, we have:
		\begin{equation}
			\rho_1(\overline\gamma_n) \ge \|\gamma_{\overline{w}_{2q_n-2c_n-2}}\cdots \gamma_{n-1}\gamma_0 \gamma{\overline{w}_{2c_n+1}-1}\|\|\gamma^w_{2c_n + 1}\cdots\gamma^w_{2q_n-2c_n-3}\|\frac{\eps^2}{32}
		\end{equation}
		Hence, with the same reasoning as in the proof of Theorem \ref{th:llnc}, we have:
		\begin{equation}\label{eq:norm-rayon}
			\log\|\overline\gamma_n\|-\log(\rho_1(\overline{\gamma}_n)) \le \sum_{k = \overline{w}_{2q_n-2c_n-2}}^{n-1} N(\gamma_k) + \sum_{k = 0}^{\overline{w}_{2c+1}-1} N(\gamma_k) + 2|\log(\eps)| + 5\log(2)
		\end{equation}
		However \eqref{eq:norm-rayon} holds without conditions on $c_n$ (with the convention $\overline{w}_{-k} = 0$ for all $k \in \NN$), because:
		\begin{equation}
			\log\|\overline\gamma_n\|-\log(\rho_1(\overline{\gamma}_n)) \le \sum_{k = 0}^{n-1} N(\gamma_k).
		\end{equation}
		Moreover, for all $ t\ge B$, and for all $k \in\NN$, one has:
		\begin{equation*}
			\PP(N(\gamma_k) > t\,|\,c_n, q_n) \le \frac{N_*\nu(t,+\infty)}{1-\alpha}.
		\end{equation*}
		Moreover by Lemma \ref{lem:sumexp} and Theorem \ref{th:pivot-extract}; $\overline{w}_{2c_n+1} + n- \overline{w}_{2q_n-2c_n-2}$ has a bounded exponential moment with a bound that depends only on $(\alpha, \rho, m)$. 
	\end{proof}
	
	Theorems \ref{th:llnc} and \ref{th:llnr}
	
	\begin{proof}[Proof of Theorem~\ref{th:slln}]
		Let $E$ be a Euclidean vector space, let $\nu$ be a strongly irreducible and proximal probability measure on $\mathrm{GL}(E)$ and let $(\gamma_n) \sim \nu^{\otimes \NN}$.
		Let $\rho = \frac{1}{4}$ and $K = 10$ and let $\alpha, \eps \in (0,1)$ and $m\in\NN$ be as in Corollary \ref{cor:schottky}. 
		Let $B, C, \beta > 0$ be as in Theorem \ref{th:llnc}. 
		Up to taking the minimal $\beta$ and the maximal $C$, we may assume that $B, C, \beta > 0$ also satisfy the conclusions of \ref{th:llnc}.
		Let $D := 2|\log(\eps)| + 5\log(2)$.
		Note that if $N_* \nu = \delta_0$, then $N_* \nu^*m = \delta_0$ for all $k$ by sub-additivity of $N$. 
		Moreover $\prox \le N$ on $\mathrm{GL}(E)$ so $N_* \nu \neq \delta_0$ by proximality of $\nu$.
		Let $C_0, \beta_0$ be as in Lemma \ref{lem:simplification} for $\eta = N_* \nu$.
		Then for all $t \ge 0$, we have:
		\begin{equation*}
			\sum_{k=1}^{\infty}  C\exp\left(-\beta k\right) \left\lceil\frac{B \vee N_*\nu}{1-\alpha}\right\rceil^{\uparrow k} \left({t - 2|\log(\eps)| - 5\log(2)},+\infty\right) \le \sum_{k=1}^{\infty}  C_0\exp\left(-\beta_0 k\right) N_*\nu(t, + \infty)
		\end{equation*}
		Therefore \eqref{eqn:llnc} implies \eqref{eq:dom-coef} and \eqref{eqn:llnr} implies \eqref{eq:dom-radius}.
	\end{proof}

	\begin{proof}[Proof of Corollary~\ref{cor:regxi}]
		Let $V$ be a proper subspace of $E$, let $0< r \le 1$. Let $v \in E \setminus\{0\}$ and let $f\in E^*\setminus\{0\}$ be such that $f(V) = \{0\}$. Let $(\gamma_n)\sim \nu^{\otimes\NN}$ and let $l^\infty = l^\infty(\gamma)$. We claim that:
		\begin{equation}\label{eq:liminf-pp}
			\PP(l^\infty \in \mathcal{N}_r(V)) \le \liminf_{n\to +\infty} \PP\left(|f \overline\gamma_n v| < r \|f\|\|\overline{\gamma}_n\|\|v\|\right)
		\end{equation}
		Indeed, if $l^\infty \in \mathcal{N}_r(V)$, then by Theorem \ref{th:limit-uv}, we construct a random integer $n_0$ which has finite exponential moment and such that $[\overline\gamma_n v] \in \mathcal{N}_r(V)$ for all $n \ge n_0$. Note also that if $[\overline\gamma_n v] \in \mathcal{N}_r(V)$, then by Lemma \ref{lem:product-is-lipschitz} $|f \overline\gamma_n v| < r \|f\|\|\overline{\gamma}_n\|\|v\|$. Hence, for all $n \in\NN$, we have:
		\begin{equation*}
			\PP\left(|f \overline\gamma_n v| < r \|f\|\|\overline{\gamma}_n\|\|v\|\right) \ge \PP(l^\infty \in \mathcal{N}_r(V)) - \PP(n \le n_0).
		\end{equation*}
		Moreover $\PP(n \le n_0) \underset{n \to +\infty}{\longrightarrow} 0$, hence we have \eqref{eq:liminf-pp}. So by Theorem \ref{th:llnc}, we have \eqref{eq:regxi}.
	\end{proof}
	
	We said in the introduction that \ref{cor:regxi} is an amelioration of a result by Benoist and Quint in~\cite[Proposition~4.5]{CLT16}. 
	We show that the polynomial regularity of the invariant measure in Corollary \ref{cor:regxi} is actually optimal.
	Let $E = \RR^d$ and let $\nu := \nu_A * \nu_K$ where $\nu_K$ is the Haar measure on the (compact) group of isometries $\mathrm{O}(E)$ and $A$ is the distribution of the matrix
	\begin{equation*}
		M := \begin{pmatrix}
			\exp(T) & 0 & \cdots & 0 \\
			0 & 1 & \ddots & \vdots \\
			\vdots & \ddots & \ddots & 0 \\
			0 & \cdots & 0 & 1
		\end{pmatrix},
	\end{equation*}
	Where $T$ is any non-negative, real valued random variable. 
	Then $\nu$ is strongly irreducible and proximal and it actually has full support in $\mathrm{PGL}(E)$. 
	Then write $\xi^\infty_\nu$ for the invariant distribution on $\mathbf{P}(E)$ and write $e_1$ for the first base vector. 
	Then we have $\xi^\infty := \nu_A * \xi_K$ for $\xi_K$ the Lebesgue measure on $\mathbf{P}(E)$. Indeed $\nu_K * \xi = \xi_K$ for all $\xi$, by property of the Haar measure. As a consequence, we have for all $r\ge 0$:
	\begin{equation}\label{eq:contrex}
		 \frac{1}{d}\PP(T \ge \log(d) - \log(r)) \le \xi^\infty_\nu(\mathcal{B}(e_1,r)).
	\end{equation}
	Indeed a random variable of distribution law $\xi_K$ has first coordinate larger than $1/d$ with probability at least $1/d$. 
	Note also that $T = N_*\nu$ and \eqref{eq:contrex} implies that the distribution $\log_*\dist(e_1,\cdot)_*\xi_\infty^\nu$ is at least in the same polynomial integrability class as $T$.

	Another interesting question that is asked in~\cite[p~231]{livreBQ} is whether Corollary \ref{cor:lim-coef} still works if we drop the proximality assumption and replace it by a total strong irreducibility assumption. Indeed theorem \ref{th:oseledets} tells us that if we take a distribution $\nu$ and write $p(\nu)$ its proximality rank \ie $p(\nu) := \min \Theta(\nu)$ for $\Theta(\nu)$ as in Definition \ref{def:Theta}, then we can construct a random limit space of dimension $p(\nu)$. 
	Then with the same trick as in the proof of \ref{th:llnc}, we can show that the coefficient $w \overline\gamma_n u$ is up to an exponentially small error the product of a linear form $w'$ and a vector $u'$ whose norms are controlled in law by the same $\zeta_\nu^{C, \beta}$. 
	However, the fact that the kernel of $w'$ cuts orthogonally the $p(\nu)$-dimensional limit space that contains $u'$ does not give a lower bound on the product $\frac{|w'u'|}{\|w'\|\|u'\|}$. 
	
	For example in dimension $2$, we can take $\nu$ to be the law of a random rotation of angle $2^{-k}\pi$ with probability $\exp(-\exp(\exp(k)))$ for all $k \in\NN$ and the identity otherwise. 
	Then the random walk $(\overline{\gamma}_n)$ is recurrent so if we take $w$ and $u$ such that $w u =0$, then we almost surely have $ |w \overline{\gamma}_n u | = 0$ for infinitely many times $n\in\NN$.
	
	The question remains open if we consider the conditional distribution of $\log\frac{|w \overline{\gamma}_n u |}{\|\overline{\gamma}_n\|}$ with respect to the event $(w \overline{\gamma}_n u \neq 0)$ or simply assume that $w \overline{\gamma}_n u \neq 0$ almost surely and for all $n \in\NN$.

	\appendix

	\section{Probabilistic tools}\label{anex:prob}
	
	In this appendix, we will give the proofs of some classical results about sums of independent random variables. 
	We will use the following notations. By probability space, we mean a space $\Omega$ endowed with a $\sigma$-algebra $\mathcal{A}$, that is isomorphic to the Borel algebra of a compact metric space, and a probability measure $\PP$ that has no atoms. We call events or measurable subsets of $\Omega$ the elements of $\mathcal{A}$.
	
	Let $(\Omega, \mathcal{A}, \PP)$ be such a probability space. Given $\mathcal{B}$ a subalgebra of $\mathcal{A}$, and $\phi$ an $\mathcal{A}$-measurable, real valued function such that $\EE(\phi)$ is well defined (meaning that the positive part or the negative part of $\phi$ has finite expectation), we define $\EE(\phi\,|\,\mathcal{B})$ as the equivalence class of all $\mathcal{B}$-measurable random variables $\psi$ taking values in $\RR\cup\{\EE(\phi)\}$ such that for all $B\in\mathcal{B}$, we have $\EE(\mathds{1}_B \phi) = \EE(\mathds{1}_B \psi)$ where $\mathds{1}_B$ is the indicator function of $B$. If $\mathcal{B}$ is the $\sigma$-algebra generated by a measurable map $\gamma : \Omega \to \Gamma$, we also write $\EE(\phi\,|\,\gamma)$ for $\EE(\phi\,|\,\mathcal{B})$. 
	
	We call filtration over $(\Omega, \mathcal{A})$ a nested sequence of sub-$\sigma$-algebras of $\mathcal{A}$ \ie a sequence  $(\mathcal{F}_k)_{k\in\NN}$ such that for all $k$, we have $\mathcal{F}_k \subset \mathcal{F}_{k+1}$. 
	
	We will write $\NN := \NN_{\ge 0}$ for the set of non-negative integers and given a sequence $(w_m)\in\RR^\NN$, we define the sequence of partial sums as $(\overline{w}_m)_{m\in\NN} := (w_0 +\cdots +w_{m-1})_{m\in \NN}$.
	
	\subsection{About exponential large deviations inequalities}
	
	In this section, we show that all the probabilistic constructions that we use in the main body of the article preserve the property of having a finite exponential moment. 
	Note however that a product of two random variables that have a finite exponential moment may not have a finite exponential moment.
	
	\begin{Lem}[Distribution of the current step]\label{lem:curstep}
		Let $(\Omega, \PP)$ be a probability space and let $(\mathcal{F}_k)_{k \in\NN}$ be a filtration on $\Omega$. Let ${(w_k)}_{k\in\NN}$ be a random sequence of positive integers such that $w_k$ is $\mathcal{F}_{k+1}$-measurable for all $k$. For all $n \ge 0$, we define $r_n := \max\{r \in\NN\,|\,\overline{w}_{r_n} \le n\}$. 
		Let $h : \NN \to \RR_{\ge 0}$ be a function. Assume that:
		\begin{equation*}
			\forall t \in \NN,\; \forall k\ge 0,\;\PP\left(w_k = t\,\middle|\,\mathcal{F}_{k}\right)\le h{(t)}.
		\end{equation*}
		Then we have:
		\begin{equation*}
			\forall t\in\NN,\; \forall n\in\NN,\; \PP\left(w_{r_n} = t\right)\le t h{(t)}.
		\end{equation*}
	\end{Lem}
	
	\begin{proof}
		Let $t$ and $n$ be two non-negative integers. We have
		\begin{align*}
			\PP{({w}_{r_n} = t)} & = \sum_{r=0}^{+\infty}\PP{((r=r_n) \cap ({w}_{r}=t))}\\
			& = \sum_{r=0}^{+\infty}\PP{((n-t \le \overline{w}_r < n)\cap (w_r=t))}\\
			& = \sum_{r=0}^{+\infty}\sum_{u = 0}^{t}\PP{(\overline{w}_r = n-u)} \PP {(w_r = t \,|\, \overline{w}_r = n-u)}.
		\end{align*}
		Probabilities are non negative so the order of summation does not matter. hence:
		\begin{equation*}
			\PP{({w}_{r_n}=t)}  = \sum_{u = 0}^{t} \sum_{r=0}^\infty\PP{(\overline{w}_r = n-u)} \PP{(w_r = t \,|\, \overline{w}_r = n-u)}
		\end{equation*}
		For all $u\in\NN$ and all $r \in \NN$, the event $(\overline{w}_r = n-u)$ is in $\mathcal{F}_{r}$. So by hypothesis, we have $\PP{(w_r = t \,|\, \overline{w}_r = n-u)} \le h(t)$ for all $t$ and all $r, u \in\NN$ such that $\PP(\overline{w}_r = n-u) > 0$.
		Moreover, for all $u \in \NN$, we have:
		\begin{equation*}
			\sum_{r=0}^\infty\PP(\overline{w}_r = n-u) = \EE(\#\{r\in\NN\,|\,\overline{w}_r = n-u\}) \le 1,
		\end{equation*}
		because the map $r \mapsto \overline{w}_r$ is almost surely injective.
		Therefore, we have $\PP{({w}_{r_n}=t)} \le t h(t)$, which proves the claim.
	\end{proof}
	
	\begin{Cor}\label{cor:curexp}
		Let ${(w_k)}_{k\in\NN}$ be a random sequence of positive integers and ${(\mathcal{F}_k)}$ be a filtration such that $w_k$ is $\mathcal{F}_{k+1}$-measurable for all $k$.
		Assume that for some $C, \beta > 0$, we have:
		\begin{equation*}
			\forall t\in \NN, \forall k\ge 0,\PP\left(w_k \ge t\,\middle|\,\mathcal{F}_k\right)\le C \exp(-\beta t).
		\end{equation*}
		For all $n \in \NN$, define the random integer $r_n := \max\{r \in\NN\,|\,\overline{w}_{r_n} \le n\}$. 
		Then:
		\begin{equation*}
			\forall t\ge 0,\; \forall n\ge 0,\; \PP\left(w_{r_n} \ge t\,\middle|\,\mathcal{F}_k\right) \le \frac{C{(1+t\beta)}}{\beta^2}\exp{(-\beta t)}.
		\end{equation*}
	\end{Cor}
	
	\begin{proof}
		Define $h(t) := C\exp(-\beta)t$ for all $t \in \NN$. Then we have:
		\begin{equation*}
			\forall t\in \NN, \forall k\ge 0,\PP\left(w_k = t\,\middle|\,\mathcal{F}_k\right)\le h(t).
		\end{equation*} 
		So by Lemma \ref{lem:curstep}, for all $t,u\in\NN$, we have $\PP\left(w_{r_n}=u\right) \le uh(u)$. Now let $n \in \NN$ and $t \in \NN$. We have :
		\begin{align*}
			\PP\left(w_{r_n} \ge t\right) & = \sum_{u = t}^{\infty} \PP\left(w_{r_n}=u\right) \le \sum_{u = t}^{\infty} uh{(u)}\\
			& \le \sum_{u = t}^{\infty} uC\exp{(-\beta u)}\\
			& \le C \exp(-\beta t) \left(\sum_{u = 0}^{\infty} u \exp{(-\beta u)} + t\sum_{u = 0}^{\infty} \exp{(-\beta u)}\right)\\
			& \le C \exp(-\beta t) \left( \left(\frac{1}{1-\exp(-\beta)}\right)^2 + \frac{t}{1-\exp(-\beta)}\right).
		\end{align*}
		Moreover $\beta \ge 1-\exp(-\beta)$ by convexity. So we have $\PP\left(w_{r_n} \ge t\right)\le \frac{C{(1+t\beta)}}{\beta^2}\exp{(-\beta t)}$.
	\end{proof}

	We now give a nice formulation of standard large deviations inequalities for sums of random variables.

	\begin{Lem}[Sum of random variables that have finite exponential moment]\label{lem:sumexp}
		Let $(\Omega,\PP)$ be a probability space endowed with a filtration ${(\mathcal{F}_n)}_{n\in\NN}$. 
		Let $w$ be an $\NN$-valued random variable. Let $(x_n)$ be a random sequence of non-negative real numbers. 
		Assume that $x_n$ is  ${(\mathcal{F}_{n+1})}$-measurable for all $n$. Let $C,\beta>0$ be non-random constants. Assume that:
		\begin{gather}
			\EE(\exp(\beta w)) \le C\label{eqn:ema-w}\\
			\forall n\in\NN,\EE{(\exp{(\beta x_n)}\;|\;\mathcal{F}_{n})}\le C.\label{eqn:ema}
		\end{gather}
		Then the random variable $\overline{x}_w := \sum_{k=0}^{w-1}x_k$ has a finite exponential moment in the sense that:
		\begin{equation}\label{eqn:sumexp}
			\exists C', \beta' >0, \; \forall t \ge 0,\;\PP\left(\overline{x}_w \ge t\right)\le C'\exp(-\beta' t).
		\end{equation}
	\end{Lem}
	
	\begin{proof}
		First note that without loss of generality, we may assume that $C >1$.
		Fix $0 < \eps < \frac{\beta}{\log(C)}$. For every $j \in \NN$, write $\overline{x}_j:=\sum_{k=0}^{j-1}x_k$. Then for all non-random $t \ge 0$, one has:
		\begin{equation}
			\PP{(\overline{x}_w \ge t)} \le \PP{(w\ge t\eps)}+\PP\left(\overline{x}_{\lfloor t\eps\rfloor}\ge t\right).
		\end{equation}
		For all $j \in \NN$, we define $z_j:=\exp{(\beta \overline{x}_j)}$. We claim that $\EE{(z_j)}\le C^j$. Indeed, $z_0=1$ and for all $j \ge 0$, the random variable $z_{j}$ is $\mathcal{F}_j$ measurable. Now by looking at the conditional expectation and using \eqref{eqn:ema}, we have: 
		\begin{align*}
			\EE{(z_{j+1})} & = \EE{\left(\EE{(z_{j+1}|\mathcal{F}_{j})}\right)}\\
			&= \EE{\left(z_{j}\EE{(\exp{(\beta x_j)}|\mathcal{F}_{j})}\right)}\\
			&\le C\EE{(z_{j})}.
		\end{align*}
		This proves the claim. By Markov's inequality, we have $\PP{(z_j\ge\exp{(\beta t)})}\le\frac{C^j}{\exp{(\beta t)}}$ for all $j \in\NN$ and for all $t > 0$. Let $t > 0$. By the above argument and because the $x_n$'s are non-negative, we have:
		\begin{equation*}
			\PP\left(\overline{x}_{\lfloor t\eps\rfloor}\ge t\right) \le C^{\lfloor \eps t \rfloor}\exp(-\beta t) \le  C^{\eps t}\exp(-\beta t) =  \exp(-(\beta-\eps\log(C))t).
		\end{equation*}
		By Markov's inequality, applied to \eqref{eqn:ema-w}, we have $\PP{(w\ge t\eps)}\le C\exp{(-\beta t)}$. So if we write $\beta':= \beta-\eps\log{(C)}>0$ and $C':=C+1$, then we have $\PP{(Y\ge t)}\le C'\exp{(-\beta' t)}$, which proves \eqref{eqn:sumexp}.
	\end{proof}
	
	\begin{Lem}[Exponential moments approximate the expectation]\label{lem:proba:aprox}
		Let $M,\sigma,C,\beta>0$. 
		For all $\alpha<\sigma$, there is a constant $\beta_\alpha > 0$, that depends on $(M,\sigma,C,\beta,\alpha)$, such that for all random variable $x$, that satisfies $\EE{(\min\{x,M\})}\ge\sigma$ and $\EE(\exp(-\beta x))\le C$, we have:
		\begin{equation*}
			\EE{(\exp{(-\beta_\alpha x)})}\le\exp{(-\beta_\alpha \alpha)}.
		\end{equation*}
	\end{Lem}
	
	\begin{proof}
		Let $x$ be a random variable such that $\EE(\exp(- \beta x)) \le C$. Then $\PP{(x\le t)}\le C\exp{(\beta t)}$ for all $t\in\RR$. For all $0 < \beta' < \beta$ and for all $m\in\RR$ we have:
		\begin{align*}
			\EE{(\exp{(-\beta'x)}\mathds{1}{(x\le m)})} & = \int_{\exp{(-\beta' m)}}^{+\infty}\PP{(\exp{(-\beta'x)}\ge t)}\mathrm{d}t\\
			&\le\int_{\exp{(-\beta' m)}}^{+\infty} C t^{-\frac{\beta}{\beta'}}\mathrm{d}t\\
			& \le C \frac{\beta'}{\beta-\beta'}\exp{({(\beta-\beta')}m)} =: F{(m,\beta')}
		\end{align*}
		To all $0< \beta \le \frac{\beta'}{2}$, we associate the number $m_{\beta'} := \min\left\{0, 2\beta^{-1}\log\left(\frac{\beta'\beta}{2C}\right)\right\}$.
		Note that for all $0< \beta \le \frac{\beta'}{2}$, we have $\beta'-\beta \ge \beta/2$ and $m_\beta \le 0$, hence $(\beta-\beta')m_{\beta'} \le \log\left(\frac{\beta'\beta}{2C}\right)$ and $C \frac{\beta'}{\beta-\beta'} \le \frac{2C \beta'}{\beta}$. 
		Therefore $F(m_{\beta'}, \beta') \le \beta'^2$ for all  $0< \beta \le \frac{\beta'}{2}$, hence $\frac{F(m_{\beta'}, \beta')}{\beta'} \underset{\beta'\to 0}{\longrightarrow} 0$.
		
		Now assume moreover that $\EE{(\min\{x,M\})}\ge\sigma$ and let $m \le M$. Write $x' := \max\{m, \min\{x, M\}\}$. Then $ m \le x' \le M$. 
		Moreover, by convexity, we have for all $m \le y \le M$:
		\begin{align*}
			\exp{(-\beta'y)} & \le \frac{y-m}{M-m}\exp{(-\beta' M)}+\frac{M-y}{M-m}\exp{(-\beta'm)} \\
			& \le \frac{M\exp{(-\beta' m)}-m\exp{(-\beta'M)}}{M-m}
			- \frac{\exp(-\beta' m)-\exp(-\beta'M)}{M-m} y.
		\end{align*}
		Note moreover that $\EE(x') \ge \sigma$ and $\frac{\exp(-\beta' m)-\exp(-\beta'M)}{M-m} \ge 0$. Hence:
		\begin{multline*}
			\EE(\exp(-\beta x')) \le \frac{M\exp{(-\beta' m)}-m\exp{(-\beta'M)}}{M-m}
			%\\ 
			- \frac{\exp(-\beta' m)-\exp(-\beta'M)}{M-m}\sigma=:L{(m,\beta')}.
		\end{multline*}
		Moreover, $\beta'm_{\beta'}\underset{\beta'\to 0}{\longrightarrow} 0$ and $\beta'M\underset{\beta'\to 0}{\longrightarrow} 0$ so $\frac{M\exp{(-\beta' m_{\beta'})}-m_{\beta'}\exp{(-\beta'M)}}{M-m_{\beta'}} \underset{\beta'\to 0}{\longrightarrow} 1$. Moreover $\exp$ is derivable at $0$ so we have:
		\begin{equation*}
			\frac{\exp(-\beta' m_{\beta'})-\exp(-\beta'M)}{\beta'(M-m_{\beta'})} \underset{\beta'\to 0}{\longrightarrow} 1.
		\end{equation*}
		Hence we have $\frac{L(m_{\beta'},\beta') -1}{\sigma\beta'} \underset{\beta'\to 0}{\longrightarrow} 1$. 
		
		By the previous arguments, we have $\frac{L(m_{\beta'},\beta') + F(m_{\beta'},\beta') - 1}{\beta' \sigma}\underset{\beta'\to 0}{\longrightarrow} 1$.
		Then for all $\alpha < \sigma$, there exists $0 < \beta' \le \frac{\beta}{2}$ such that $\frac{L(m_{\beta_\alpha},\beta_\alpha) + F(m_{\beta_\alpha},\beta_\alpha) - 1}{\beta' \sigma} \ge \frac{\alpha}{\sigma}$ and therefore $L(m_{\beta'},\beta') + F(m_{\beta'},\beta') \ge 1-\alpha\beta'$, let $\beta_\alpha$ be the maximal such $\beta'$.
		Then for all $\alpha < \sigma$, we have $L(m_{\beta_\alpha},\beta_\alpha) + F(m_{\beta_\alpha},\beta_\alpha) \ge 1-\alpha\beta_\alpha$ by continuity of $F$, $m$, and $L$.
		Let $\alpha < \sigma$.
		Then for all random variable $x$ that satisfies the hypothesis of Lemma \ref{lem:proba:aprox}, we get:
		\begin{align*}
			\EE{(\exp{(-\beta_\alpha x)})} & \le L{(m_{\beta_\alpha},\beta')}+F{(m_{\beta_\alpha},\beta_\alpha)}\\
			&\le 1 -\beta_\alpha\alpha \le \exp{(-\beta_\alpha\alpha)}.\qedhere
		\end{align*}
	\end{proof}

	\begin{Lem}[Classical large deviations inequalities from below]\label{lem:proba:ldev}
		Let $(\Omega,\PP)$ be a probability space endowed with a filtration ${(\mathcal{F}_n)}_{n\in\NN}$. Let $(x_n)_{n\in\NN}$ be a random sequence of real numbers such that $x_n$ is $\mathcal{F}_{n+1}$-measurable for all $n \in \NN$. Let $C,\beta > 0$. Assume that for all $n\in\NN$, we have $\EE{(\exp{(-\beta x_n)}\,|\,\mathcal{F}_n)}\le C$. Let ${(\sigma_M)}_{M \in \NN}$ be a non-random, real valued, non-decreasing sequence such that $\EE{(\min\{x_n,M\}\,|\,\mathcal{F}_n)}\ge\sigma_M$ for all $M,n\in\NN$. Write $\sigma := \lim_{t\to +\infty}\sigma_t$. Then we have:
		\begin{equation*}
			\forall \alpha<\sigma,\;\exists \beta_{\alpha}>0,\;\forall n\in\NN,\;\PP{(\overline x_n\le\alpha n)}\le \exp{(-\beta_\alpha n)}.
		\end{equation*}
	\end{Lem}
	
	\begin{proof}
		Let $\alpha< \sigma$ and let $\alpha<\alpha'<\alpha''<\sigma$. Let $M\in\NN$ be such that $\EE{(\min\{x_n,M\}|\mathcal{F}_n)}\ge\alpha''$ for all $n\in\NN$. Then by Lemma \ref{lem:proba:aprox}, there is a constant $\beta'>0$ such that for all $n\in\NN$, we have:
		\begin{equation*}
			\EE(\exp{(-\beta' x_n)}\,|\,\mathcal{F}_n)\le \exp(-\beta'\alpha').
		\end{equation*}
		Then by induction on $n$, we have $\EE{(\exp{(-\beta' \overline{x}_n)})}\le\exp{(-n\beta'\alpha')}$ for all $n\in\NN$. Then by Markov's inequality, we get $\PP{(\overline x_n\le\alpha n)}\le\exp{(-n\beta'{(\alpha'-\alpha)})}$, which proves the claim for $\beta_\alpha := \beta' (\alpha'-\alpha)$. 
	\end{proof}
	
	Note that the existence of the sequence $(\sigma_M)$ such that $\EE{(\min\{x_n,M\}\,|\,\mathcal{F}_n)}\ge\sigma_M$ for all $M,n\in\NN$ is satisfied when $(x_n)$ is i.i.d. and $\EE(x_0) \ge \sigma$. 
	However, if we only assumed that $\EE{(x_n\,|\,\mathcal{F}_n)} \ge 1$ for all $n$, then we may assume that for all $n$, $x_n$ takes value $n^2$ with probability $n^{-2}$ and $0$ otherwise and in this case $\PP(\forall n\in\NN,\,\overline{x}_n = 0) = \prod_{n = 1}^{\infty} (1-n^{-2}) > 0$.
	
	\begin{Cor}\label{cor:ldev-classique}
		Let $(x_n)_n$ be a random independent sequence of real numbers. 
		Assume that there exists $\beta > 0$ such that $\EE(\exp(-\beta x_0)) < +\infty$.
		Then the random sequence $(\overline{x_n})_{n\in\NN}$ satisfies large deviations inequalities below the speed $\EE(x_0)$.
	\end{Cor}
	
	\begin{proof}
		Let $\beta >0$ be such that $\EE(\exp(-\beta x_0)) < +\infty$ and let $C = \EE(\exp(-\beta x_0))$.
		For all $M \in\NN$, let $\sigma_M := \EE(\min\{x_0,M\})$.
		Let $(\mathcal{F}_n)$ be the cylinder filtration associated to $(x_n)$. 
		Then $\sigma_M \to \EE(x_0)$ so by Lemma \ref{lem:proba:ldev}, the random sequence $(\overline{x_n})_{n\in\NN}$ satisfies large deviations inequalities below the speed $\EE(x_0)$.
	\end{proof}

	\begin{Def}[Large deviations inequalities]\label{def:proba:ldev}
		Let $(x_n)_{n\in\NN}$ be a random sequence of real numbers and let $\sigma\in\RR\cup\{+\infty\}$ be a constant. We say that $(x_n)_{n\in\NN}$ satisfies large deviations inequalities below the speed $\sigma$ if we have:
		\begin{equation*}
			\forall \alpha<\sigma,\; \exists C,\beta>0,\; \forall n\in\NN,\;\PP{(x_n\le\alpha n)}\le C \exp{(-\beta n)}.
		\end{equation*}
		We say that $(x_n)_{n\in\NN}$ satisfies large deviations inequalities above the speed $-\sigma$ if $(-x_n)_{n\in\NN}$ satisfies large deviations inequalities below the speed $\sigma$.
	\end{Def}
	
	\begin{Rem}\label{rem:ldevconv}
		Let ${(x_n)}$ be a random sequence of real numbers and let $\sigma\in\RR\cup\{+\infty\}$. Assume that $(x_n)$ satisfies large deviations inequalities below the speed $\sigma$. Then by Borel Cantelli's Lemma, we have almost surely $\liminf \frac{x_n}{n} \ge \alpha$ for all $\alpha < \sigma$ so $\liminf \frac{x_n}{n} \ge \sigma$ almost surely. If we moreover assume that $\sigma$ is finite and if $(x_n)$ satisfies large deviations inequalities above the same speed $\sigma$, then $\lim \frac{x_n}{n} = \sigma$ almost surely.
	\end{Rem}
	
	\begin{Rem}[Convenient reformulation of Definition \ref{def:proba:ldev}]\label{rem:convenient-reformulation}
		We call decreasing exponential function a function of type $n \mapsto C \exp(-\beta n)$ with $C > 0$ and $\beta > 0$. 
		
		Note that saying that a random sequence $(x_n)$ satisfies large deviations inequalities below a speed $\sigma$ means that for all $\alpha < \sigma$, the function $n \mapsto \PP(x_n < \alpha n)$ is bounded above by a decreasing exponential function.
		It is equivalent to saying that for all $\alpha < \sigma$, the function $n \mapsto \PP(\exists m \ge n,\; x_{m} < \alpha n)$ is bounded above by a decreasing exponential function.
		
		Note also that a sum of finitely many decreasing exponential functions is bounded above by a decreasing exponential function. 
	\end{Rem}
	
	Now we show that random sequences that satisfy large deviations inequalities behave well under some compositions. 
	
	\begin{Lem}\label{lem:proba:ldevcompo}
		Let $(\Omega, \PP)$ be a probability space.
		Let $\sigma, \sigma'\in\RR\cup\{+\infty\}$. 
		Let ${(x_n)}_{n\in\NN}$ and ${(x'_n)}_{n\in\NN}$ be two random sequences of real numbers that satisfy large deviations inequalities below the speeds $\sigma$ and $\sigma'$ respectively. Let ${(y_n)_{n\in\NN}}$ be a random sequence of real numbers. Let $C_y, \beta_y > 0$. Assume that $\PP(y_n \le -t) \le C_y \exp(-\beta_y t)$ for all $n \in\NN$ and all $t \ge 0$. Let $(k_n)_{n\in\NN}$ be a random non-decreasing sequence of non-negative integers and let $\kappa \in (0, + \infty)$. Then:
		\begin{enumerate}
			\item\label{compo:shift} The shifted sequence ${(x_n + y_n)}_{n\in\NN}$ satisfies large deviations inequalities below the speed $\sigma$.
			\item\label{compo:min} The minimum ${(\min\{x_n,x'_n\})}_{n\in\NN}$ satisfies large deviations inequalities below the speed $\min\{\sigma,\sigma'\}$.
			\item\label{compo:max} The maximum ${(\max\{x_n,x'_n\})}_{n\in\NN}$ satisfies large deviations inequalities below the speed $\max\{\sigma,\sigma'\}$.
			\item\label{compo:sum} For all $\lambda,\lambda'\ge 0$, the sum ${(\lambda x_n+ \lambda' x'_n)}_{n\in\NN}$ satisfies large deviations inequalities below the speed $\lambda\sigma + \lambda'\sigma'$.
			\item\label{compo:compo} Assume that ${(k_n)}_{n\in\NN}$ satisfies large deviations inequalities below the speed $\kappa$. Then the composition ${(x_{k_n})}_{n\in\NN}$ satisfies large deviations inequalities below the speed $\kappa\sigma$.
			\item\label{compo:rec} Let ${(r_m)}_{m\in\NN}$ be the reciprocal function of $(k_n)_{n\in\NN}$, defined by $r_m := \max\{n\in\NN\,|\,k_n\le m\}$ for all $m\in\NN$. Assume that ${(k_n)}_{n\in\NN}$ satisfies large deviations inequalities below the speed $\kappa$. Then $(r_m)_{m\in\NN}$ satisfies large deviations inequalities above the speed $\kappa^{-1}$.
		\end{enumerate}
	\end{Lem}
	
	\begin{proof}
		We first prove \eqref{compo:shift}. Let $\alpha < \alpha' < \sigma$. By assumption, there are two constants $C_x, \beta_x > 0$ such that $\PP(x_n\le\alpha' n) \le C_x\exp(-\beta_x n)$.
		 Write $\beta:=\min\{\beta_x,\beta_y{(\alpha' - \alpha)}\}$ and $C := C_x +C_y$. Then we have:
		\begin{align*}
			\forall n\in\NN,\;\PP{(x_n + y_n \le \alpha n)} & \le \PP{(x_n \le \alpha' n)}+\PP{(y\le{(\alpha - \alpha')}n)}\\
			& \le C_x\exp{(-\beta_x n)} + C_y\exp{(-\beta_y(\alpha'-\alpha) n)}\\
			& \le C\exp\left(-\beta n\right).
		\end{align*}
		In other words the function $n \mapsto \PP{(y_n+x_n\le\alpha n)}$ is bounded by the sum of the functions $n \mapsto \PP{(x_n\le\alpha' n)}$ and $n \mapsto \PP{(y\le{(\alpha-\alpha')}n)}$ which are themselves bounded by decreasing exponential functions so their sum also is by Remark \ref{rem:convenient-reformulation}.
		
		Now to prove \eqref{compo:min} assume that $\sigma \le \sigma'$. This is not restrictive since $(\sigma, (x_n)_{n\in\NN})$ and $(\sigma',(x'_n)_{n\in\NN})$ play symmetric roles in Lemma \ref{lem:proba:ldevcompo}. Then for all $\alpha < \sigma$, we have:
		\begin{equation*}
			\forall n\in\NN, \; \PP(\min\{x_n, x'_n\} \le \alpha n) \le \PP(x_n \le \alpha n) + \PP(x'_n \le \alpha n).
		\end{equation*}
		Both terms of the sum on the right are bounded above by decreasing exponential functions of $n$ so $\PP{(\min\{x_n,x'_n\} \le \alpha n)}$ is bounded above by a decreasing exponential function of $n$. 
		
		To prove \eqref{compo:max}, we again assume that $\sigma \le \sigma'$. Then for all $\alpha < \sigma'$, we have $\PP{(\max\{x_n,x'_n\}\le\alpha n)}\le\PP{(x'_n\le\alpha n)}$ and by assumption $\PP(x'_n\le\alpha n)$ is bounded above by a decreasing exponential function of $n$.
		
		Now we prove \eqref{compo:sum}, let $\alpha_+ < \lambda \sigma + \lambda'\sigma'$. Let $\alpha <\sigma$ and $\alpha'<\sigma'$ be such that $\alpha_+ = \lambda\alpha + \lambda'\alpha'$.
		Such $\alpha, \alpha'$ always exist.
		Now note that:
		\begin{equation*}
			\PP{(\lambda x_n+\lambda'x'_n\le\alpha_+ n)} \le \PP{(x_n\le\alpha n)}+\PP{(x'_n\le\alpha' n)}
		\end{equation*}
		an both terms are bounded by decreasing exponential functions, which proves \eqref{compo:sum}.
		
		Now we prove \eqref{compo:compo}. Let $\alpha < \kappa\sigma$ and let $\alpha'<\sigma$ and $\alpha''<\kappa$ be such that $\alpha = \alpha'\alpha''$. 
		Now let $C_x,\beta_x>0$ be such that $\PP{(x_n\le\alpha' n)} \le C_x \exp{(-\beta_x n)}$ for all $n\in \NN$ and let $C_k,\beta_k$ be such that $\PP{(k_n\le \alpha'' n)}\le C_k\exp{(-\beta_k n)}$ for all $n\in\NN$. 
		Such $C_x, \beta_x, C_k,\beta_k$ exist by assumption. 
		For all $n\in\NN$, we have:
		\begin{align*}
			\PP{(x_{k_n}\le\alpha n)} & \le \PP{(x_{k_n}\le\alpha' k_n \cap k_n\ge\alpha'' n)}+\PP{(k_n\le\alpha'' n)}\\
			& \le \sum_{k\ge \alpha'' n}\PP{(x_{k}\le\alpha' k)}+\PP{(k_n\le\alpha'' n)}\\
			& \le \sum_{k\ge \alpha'' n}{C_x}\exp{(-\beta_x k)}+C_k\exp{(-\beta_k n)}\\
			& \le \frac{C_x}{\beta_x}\exp{(-\beta_x \alpha'' n)}+C_k\exp{(-\beta_k n)}\\
			& \le \left(\frac{C_x}{\beta_x}+C_k\right)\exp\left(-\min\{\beta_x\alpha'',\beta_k\} m_0\right),
		\end{align*}
		Which proves \eqref{compo:compo}.
		
		To prove \eqref{compo:rec}, we use a similar method. Let $\alpha<\kappa$ and $C,\beta>0$ be such that $\PP{(k_n\le\alpha n)}\le C\exp{(-\beta n)}$ for all $n\in\NN$. Such $C, \beta$ exist for all $\alpha < \kappa$. Then for all $m_0\in\NN$, we have:
		\begin{align*}
			\PP{(r_{m_0}\ge \alpha^{-1} m_0)} & \le\PP{(\exists m\ge m_0,\; r_m\ge\alpha^{-1} m)}\\
			& \le\PP{(\exists m\ge m_0,\; \exists n\in\NN,\; {(n\ge\alpha^{-1} m)}\wedge {(k_n\le m)})}\\
			&\le \PP{(\exists n\ge \alpha^{-1}m_0,\; k_n\le \alpha n)}\\
			&\le \frac{C}{\beta}\exp{(-\beta\alpha^{-1} n)}.
		\end{align*}
		Now note that for all $\alpha'>\kappa^{-1}$, we have $\alpha^{-1}<\kappa$. The above reasoning tells us that for all $\alpha'>\kappa^{-1}$, we have constants $C', \beta'> 0$ such that $\PP(r_m \ge \alpha'_m) \le C'\exp(-\beta m')$ for all $m$ (namely $C' = \frac{C}{\beta}$ and $\beta' = \beta \alpha^{-1}$ with $C, \beta$ as above for $\alpha := \alpha'^{-1}$).
	\end{proof}

	\subsection{About moments} 
	
	In this section, we prove useful results about sums of random variables that have a finite polynomial moment.
	
	\begin{Rem}
		Note that a probability distribution $\eta$ on $\RR_{\ge 0}$ is characterized by the right-continuous and non-increasing map $t \mapsto \eta(t, +\infty)$.
	\end{Rem}
	
	\begin{Def}[$\mathrm{L}^p$-integrability]\label{def:lpnorm}
		Let $p\in{(0,+\infty)}$ and let $\eta$ be a probability distribution on $\RR_{\ge 0}$. We define the strong $\mathrm{L}^p$ moment of $\eta$ as:
		\begin{equation}\tag{strong-L$^p$}\label{eqn:strlp}
			M_p{(\eta)} := \int_{0}^{+\infty}t^{p-1}\eta(t,+\infty)dt.
		\end{equation}
		We define the weak $\mathrm{L}^p$ moment of $\eta$ as:
		\begin{equation}\tag{weak-L$^p$}\label{eqn:wiklp}
			W_p{(\eta)} := \sup_{t\ge 0}t^p\eta\left(t,+\infty\right) < +\infty.
		\end{equation}
		We say that $\eta$ is strongly $\mathrm{L}^p$ if $M_p(\eta) < + \infty$ and we say that $\eta$ is weakly $\mathrm{L}^p$ if $W_p(\eta) < + \infty$.
	\end{Def}

	\begin{Def}[Trunking]\label{def:trunking}
		Let $\eta$ be a non-negative measure on $\RR_{\ge 0}$ that has finite total mass \ie a non-negative multiple of a probability distribution. We call trunking of $\eta$ the distribution $\lceil\eta\rceil$ characterized by:
		\begin{equation*}
			\forall t\ge 0,\; \lceil\eta\rceil{(t,+\infty)}=\min\{1,\eta'{(t,+\infty)}\}.
		\end{equation*}
		Note that if $\eta$ has total mass less that one, then $\lceil\eta\rceil = \eta + (1 - \eta{(\RR_{\ge 0})}) \delta_0$.
	\end{Def}
	
	\begin{Def}[Push-up]\label{def:pushup}
		Let $\eta$ be a probability distribution on $\RR\ge 0$ and let $B \ge 0$ be a constant. We define the push-up of $\eta$ by $B$ as the probability distribution $B \vee \eta$ on $\RR_{\ge B}$ characterized by:
		\begin{equation}
			\forall t \ge B, \; (B \vee \eta)(t, +\infty) = \eta(t, +\infty).
		\end{equation}
		In other words, for any random variable $x \sim \eta$, we have $\max\{x,B\} \sim B \vee \eta$.
	\end{Def}

	\begin{Def}[Coarse convolution]\label{def:corconv}
		Let $\eta$ be a probability distribution on $\RR_{\ge 0}$ and let $k\ge 1$ be an integer. We define the coarse convolution $\eta^{\uparrow k}$ as:
		\begin{equation*}
			\forall t \ge 0,\; \eta^{\uparrow k}{(t,+\infty)} := \min\left\{1, k \eta\left(\frac{t}{k}, +\infty\right)\right\}.
		\end{equation*}
	\end{Def}
	
	\begin{Lem}\label{lem:corconv}
		Let $k\ge 1$ be an integer, let $\eta$ be a probability distribution on $\RR_{\ge 0}$ and let $x_1,\dots,x_k$ be random variables such that:
		\begin{equation}
			\forall t \ge 0,\;\forall i\in\{1,\dots,k\},\; \PP{(x_i >t)} \le \eta{(t,+\infty)}.
		\end{equation}
		Then we have:
		\begin{equation}
			\forall t\ge 0,\;\PP{(x_1 + \dots + x_k > t)} \le \eta^{\uparrow k}{(t,+\infty)}.
		\end{equation}
	\end{Lem}
	
	\begin{proof}
		Let $t\ge 0$. We have:
		\begin{align*}
			\PP{(x_1+\cdots+ x_k >t)} & \le \PP\left(\exists i\in\{1,\dots,k\},\; x_i > \frac{t}{k}\right) \\
			& \le k \eta\left(\frac{t}{k},+\infty\right).\qedhere
		\end{align*}
	\end{proof}
	
	\begin{Lem}\label{lem:convol}
		Let $\eta$ be a probability distribution on $\RR_{\ge 0}$ and let $k \in \NN_{\ge 1}$. We have:
		\begin{align}
			W_p{(\eta^{\uparrow k})} & \le k^{p+1} W_p{(\eta)} \label{kW} \\
			M_p{(\eta^{\uparrow k})} & \le k^{p+1} M_p{(\eta)} \label{kM}
		\end{align}
	\end{Lem}
	
	\begin{proof}
		For the weak moment, we have:
		\begin{align*}
			W_p{(\eta^{\uparrow k})} & = \max t^{p}\eta^{\uparrow k} {(t,+\infty)} \\
			& \le \max t^{p} k \eta\left(\frac{t}{k},+\infty\right) \\
			& \le \max (kt')^{p} k \eta\left(t',+\infty\right) \\
			& \le k^{p+1} W_p{(\eta)}.
		\end{align*}
		This proves \eqref{kW}.
		For the strong moment, by integration by parts, we have:
		\begin{align*}
			M_p{(\eta^{\uparrow k})} & = \int_{0}^{+\infty}t^{p-1}\eta^{\uparrow k} {(t,+\infty)}dt \\
			& \le \int_{0}^{+\infty} t^{p-1} k \eta\left(\frac{t}{k},+\infty\right)dt.
		\end{align*}
		Then by the linear change of integration variable $t = ku$, we have:
		\begin{equation*}
			\int t^{p-1} k \eta\left(\frac{t}{k},+\infty\right)dt = \int u^{p-1} k^{p+1} \eta\left(u,+\infty\right)du = k^{p+1}M_p(\eta).
		\end{equation*}
		This proves \eqref{kM}.
	\end{proof}

	\begin{Lem}\label{lem:descrisum}
		Let $n$ be a random integer and let $x_1,\dots, x_n$ be non-negative real random variables. Let $B,C_1$ be such that:
		\begin{equation}\label{beemovie}
			\forall t\ge B,\; \forall k\in\NN,\; \forall m \le k,\; \PP{(x_m\ge t\,|\, n = k)}\le C_1 \eta{(t)}.
		\end{equation}
		Let $C_2,\beta>0$ be such that:
		\begin{equation}
			\forall k\in\NN,\;\PP{(n=k)}\le C_2\exp{(-\beta k)}.
		\end{equation}
		Then for $C := C_1 C_2$, we have:
		\begin{equation*}
			\forall t\ge 0,\; \PP{(x_1 + \cdots + x_n > t)} \le \sum_{k=0}^\infty C\exp{(-\beta k)}{(B \vee \eta)}^{\uparrow k}.
		\end{equation*}
	\end{Lem}
	
	\begin{proof}
		First note that \eqref{beemovie} with Definition \ref{def:pushup} and Lemma \ref{lem:corconv} implies that for all $k\in\NN$, we have:
		\begin{equation*}
			\forall t\ge 0,\, \PP{(x_1 + \cdots + x_k > t\,|\,n=k)} \le C_1 {(B \vee \eta)}^{\uparrow k}{(t,+\infty)}.
		\end{equation*}
		We do the computation, for all $t\ge 0$:
		\begin{align*}
			\PP{(x_1 + \cdots + x_n > t)} & = \sum_{k=0}^\infty\PP{( n = k )}\PP{( x_1 + \cdots + x_k > t \,|\, n = k )}\\
			& \le \sum_{k=0}^\infty C_2 \exp{(-\beta k)} C_1 {(B \vee \eta)}^{\uparrow k}{(t,+\infty)}.\qedhere
		\end{align*}
	\end{proof}
	
	\begin{Lem}\label{lem:simplification}
		Let $\eta$ be a non-trivial probability distribution on $\RR_{\ge 0}$ \ie $\eta \neq \delta_0$. Then, for all $B, C, D, \beta>0$, there are constants $C_0,\beta_0$ such that for all $t > 0$, we have:
		\begin{equation*}
			\sum_{k=0}^\infty C\exp{(-\beta k)}{(B \vee \eta)}^{\uparrow k}{(t - D,+\infty)} \le \sum_{k=0}^\infty C_0\exp{(-\beta_0 k)}\eta{(t/k,+\infty)}.
		\end{equation*}
	\end{Lem}
	
	\begin{proof}
		Note that ${(B \vee \eta)}^{\uparrow k}{(t - D,+\infty)} \le k \eta{(t/k-B - D,+\infty)}$ for all $t, B, D$ and for all $ k > 0 $. 
		Let $B' = B + D$.
		Note that for all $\beta''<\beta$, we have $\lim_k k\exp{({(\beta''-\beta)} k)} = 0$ and $\exp{({(\beta-\beta'')} k)} \ge 1$ so for $C''$ large enough, we have $k\exp{(-\beta k)}\le C'' \exp{(-\beta'' k)}$ for all $k$. Take such a $\beta''>0$ and such a $C''$. Now we re-index the sum by taking $k'= 2 k$ and write $\beta':=\beta/2$ and $C'= C C''$, then we have:
		\begin{align*}
			\sum_{k=0}^\infty C\exp{(-\beta k)}{(B \vee \eta)}^{\uparrow k}{(t - D,+\infty)} & \le \sum_{k=0}^\infty C C''\exp{(-\beta'' k)}\eta{(t/k-B',+\infty)} \\
			& \le  \sum_{k'=0}^\infty C'\exp{(-\beta' k')}\eta{(2t / k'-B',+\infty)} \\
			& \le \sum_{k'= 0 }^{\lceil t/B' \rceil - 1} C'\exp{(-\beta' k')}\eta{(t / k',+\infty)} + \sum_{k'= \lceil t/B' \rceil}^\infty C'\exp{(-\beta' k')} \\
			& \le \sum_{k'= 0 }^{+ \infty} C'\exp{(-\beta' k')}\eta{(t / k',+\infty)} + \frac{C'}{\beta'} \exp{(-\beta' t/ B')}.
		\end{align*}
		Now we use the fact that $\eta \neq \delta_0$ and take $a>0$ such that $\eta{(a,+\infty)}>0$. Then for all $t > 0$ and all $C_0 \ge 0$ and all $\beta_0 > 0 $, we have:
		\begin{equation*}
			\sum_{k= 0 }^{+ \infty} C_0\exp{(-\beta_0 k)}\eta{(t / k,+\infty)} \ge C_0 \exp(-\beta_0\lceil t / a \rceil) \eta\left(\frac{t}{\lceil t / a \rceil},+\infty\right) \ge C_0 \exp{(-\beta_0(t / a +1))} \eta{(a,+\infty)}.
		\end{equation*}
		Let $0 < \beta_0 \le \beta'$ be small enough, so that $-\beta_0\lceil t / a \rceil \le \beta' t/B' + \log(2)$ for all $t \ge 0$. Then for all $C_0 > 0$, we have:
		\begin{equation*}
			\frac{C'}{\beta'} \exp{(-\beta' t/ B')}  \le \left(\frac{C'}{\beta'C_0\eta(a,+\infty)}\right)\sum_{k = 0 }^{+ \infty} C_0\exp{(-\beta_0 k)}\eta{(t / k, +\infty)}
		\end{equation*}
		Let $0 < C_0$ be large enough so that $\frac{C'}{C_0}+\frac{C'}{\beta'C_0\eta(a,+\infty)} \le 1$. Then we have:
		\begin{align*}
			\sum_{k'= 0 }^{+ \infty} C'\exp{(-\beta' k')}\eta{(t / k',+\infty)} + \frac{C'}{\beta'} \exp{(-\beta' t/ B')} & \le \left(\frac{C'}{C_0}+\frac{C'}{\beta'C_0\eta(a,+\infty)}\right)\sum_{k = 0 }^{+ \infty} C_0\exp{(-\beta_0 k)}\eta{(t / k, +\infty)}\\
			& \le \sum_{k = 0 }^{+ \infty} C_0\exp{(-\beta_0 k)}\eta{(t / k, +\infty)}
		\end{align*}
		Hence:
		\begin{gather*}
			\sum_{k=0}^\infty C\exp{(-\beta k)}{(B \vee \eta)}^{\uparrow k}{(t - D,+\infty)}  \le \sum_{k = 0 }^{+ \infty} C_0\exp{(-\beta_0 k)}\eta{(t / k, +\infty)}. 
		\end{gather*}
	\end{proof}
	
	\begin{Lem}\label{lem:sumlp}
		Let $\eta$ and $\kappa$ be probability distributions on $\RR_{\ge 0}$. Let $C,\beta>0$ be constants. Assume that for all $t > 0$, we have
		\begin{equation*}
			\kappa(t, +\infty) \le \sum_{k=0}^\infty C\exp{(-\beta k)}\eta(t/k, +\infty).
		\end{equation*}
		Let $p\in\RR_{>0}$. Assume that $\eta$ is strongly or weakly $\mathrm{L}^p$, then $\kappa$ also is and we have:
		\begin{align*}
			M_p(\kappa) & \le M_p(\eta) \sum_{k=0}^\infty C\exp{(-\beta k)} k^p \\
			W_p(\kappa) & \le W_p(\eta) \sum_{k=0}^\infty C\exp{(-\beta k)} k^p.
		\end{align*} 
	\end{Lem}
	
	\begin{proof}		
		Let $p >0$. We claim that $M_p{(\kappa)}\le  \sum_{k=0}^\infty C\exp{(-\beta k)} k^p M_p{(\eta)}$, which is finite when $M_p(\eta)$ is. To prove that claim, we simply compute the moments, using the fact that all the quantities we look at are non-negative:
		\begin{align*}
			M_p(\kappa) & = \int_{0}^{\infty} t^{p-1} \kappa(t,+\infty) dt \\
			& \le \int_{0}^{\infty} t^{p-1} \sum_{k=0}^\infty C\exp{(-\beta k)}\eta(t/k, +\infty) dt\\
			& \le \sum_{k=0}^\infty C\exp{(-\beta k)}\int_{0}^{\infty} t^{p-1} \eta(t/k, +\infty) dt \\
			& \le \sum_{k=0}^\infty C\exp{(-\beta k)}\int_{0}^{\infty} (ku)^{p-1} \eta(u, +\infty) k du \\
			& \le \sum_{k=0}^\infty C\exp{(-\beta k)} k^p M_p(\eta).
		\end{align*}
		This proves the claim.
		We now claim that $W_p{(\kappa)}\le  \sum_{k=0}^\infty C\exp{(-\beta k)} k^p W_p{(\eta)}$. For that claim, we do the same computation:
		\begin{align*}
			W_p(\kappa) & = \sup_{t > 0} t^{p} \kappa(t,+\infty)\\
			& \le \sup_{t > 0}\left( t^{p} \sum_{k=0}^\infty C\exp{(-\beta k)}\eta(t/k, +\infty)\right)\\
			& \le \sum_{k=0}^\infty C\exp{(-\beta k)}\sup_{t > 0} t^{p} \eta(t/k, +\infty) \\
			& \le \sum_{k=0}^\infty C\exp{(-\beta k)}\sup_{t > 0} (ku)^{p} \eta(u, +\infty) \\
			& \le \sum_{k=0}^\infty C\exp{(-\beta k)} k^p W_p(\eta).
		\end{align*}
		This proves the claim, which concludes the proof of Lemma \ref{lem:sumlp}.
	\end{proof}

	\bibliographystyle{alpha}
	\bibliography{biblio.bib}

\end{document}